\documentclass[12pt,a4paper]{article}
\usepackage{amsmath}
\usepackage{amssymb}
\usepackage{amsthm}
\usepackage{amsfonts}
\usepackage{latexsym}

\theoremstyle{plain}
\newtheorem{thm}{Theorem}[section]
\newtheorem{cor}{Corollary}[section]
\newtheorem{lem}{Lemma}[section]
\newtheorem{prop}{Proposition}[section]
\newtheorem{step}{Step}[section]
\theoremstyle{remark}

\theoremstyle{definition}
\newtheorem{defn}{Definition}[section]

\newtheorem{rem}{Remark}[section]

\title{Equivariant Asymptotics of Szeg\"{o} kernels under Hamiltonian $U(2)$ actions}
\author{Andrea Galasso and Roberto Paoletti\footnote{\noindent{\bf Address:}
Dipartimento di Matematica e Applicazioni, Universit\`a degli Studi
di Milano Bicocca, Via R. Cozzi 55, 20125 Milano,
Italy; {\bf e-mail}: andrea.galasso@unimib.it, roberto.paoletti@unimib.it }}
\date{}

\begin{document}
\maketitle

\begin{abstract}
 Let $M$ be  complex projective manifold, and $A$ a positive line bundle on it. 
Assume that a compact and connected Lie group $G$ acts on $M$ in a Hamiltonian manner, and that this action linearizes
to $A$. Then there is an associated unitary representation of $G$ on the associated algebro-geometric
Hardy space. If the moment map is nowhere vanishing, the isotypical component are all finite dimensional;
they are generally not spaces of sections of some power of $A$. One is then led to study the local and global asymptotic
properties the isotypical component associated to a weight $k \, \boldsymbol{ \nu }$, when $k\rightarrow 
+\infty$. In this paper, part of a series dedicated to this general theme, we consider the case $G=U(2)$.  
\end{abstract}

\section{Introduction}

In many interesting and natural situations, 
an Hamiltonian action of a Lie group $G$ on a Hodge manifold can be linearized to a polarizing
positive line bundle; when this happens, there
is an induced unitary representation of $G$ on a certain Hardy space, intrinsically related to the holomorphic
structure of the line bundle. 
One is then led
to investigate the decomposition of the latter Hardy space into isotypical components over the irreducible
representations of $G$, and how this decomposition reflects the geometry of the underlying action.
In particular, if the corresponding moment map is never vanishing,
then all the isotypical components are finite-dimensional. 

For example, in the very special case where
$G=S^1$ acts trivially on $M$ and the moment map is taken to be $\Phi_G=1$, the corresponding
isotypical components
are (naturally isomorphic to) the spaces of global holomorphic sections of powers of $A$. 
In general, however, the isotypical components in point don't correspond to subspaces of holomorphic sections 
of some higher tensor power of the polarizing line bundle; in other words, they generally  
split non-trivially under the
structure $S^1$-action on $X$.

From the point of view of geometric quantization, 
the most appropriate heuristic framework for the present discussion is 
the setting of \lq homogeneous\rq \, quantization treated in \cite{guillemin-sternberg hq} (and 
of course \cite{boutet-guillemin}). 
In fact, a motivation for the present analysis is to revisit the general theme of \cite{guillemin-sternberg hq}  
in the specific context
of Toeplitz quantization (in the sense of \cite{boutet-guillemin}) by means of the approach to algebro-geometric Szeg\"{o}
kernels developed in \cite{zelditch-theorem-of-Tian}, \cite{bsz}, \cite{sz};
this circle of ideas is ultimately based on the microlocal theory of the Szeg\"{o} kernel as an FIO developed in \cite{boutet-sjostraend}.

In this work, we shall consider the case $G=U(2)$,
and focus on the asymptotics of the isotypical components
pertaining to a given \textit{ladder representation}, in the 
terminology of \cite{guillemin-sternberg hq}. In other words, we shall fix a ray in weight space, 
and study the asymptotic behavior of the isotypes when the representation drifts to infinity along the ray.
When $G$ is a torus, this problem was studied in \cite{pao-IJM}, \cite{pao-loa}, \cite{camosso};
the case $G=SU(2)$ is the object of \cite{gal-pao}. 
To make this more precise,
it is in order to set the geometric stage in detail.

Let $M$ be a connected $d$-dimensional complex projective manifold, with complex structure $J$.
Let $(A,h)$ be a positive line bundle on $M$; in other words, $A$ is an holomorphic ample line bundle on $M$, 
$h$ is an Hermitian
metric on $A$, and the curvature form of the unique covariant derivative
$\nabla$  on $A$ compatible with both the complex and Hermitian structures
has the form $\Theta= -2\,\imath \,\omega$, where $\omega$ is a K\"{a}hler form on $M$. 
We shall denote by $\rho$ the corresponding Riemannian structure on $M$, given by
\begin{equation}
 \label{eqn:riemannian structure}
\rho_m (v,w) := 
\omega_m\big(J_m(v), \,w\big) \quad 
(m\in M,\, v,w\in T_mM).
\end{equation}

If $A^\vee\supset X\stackrel{\pi}{\rightarrow} M $ is the unit circle bundle in the dual of $A$, 
then $\nabla$ naturally corresponds to a connection 1-form $\alpha$ on $X$,
such that $\mathrm{d}\alpha = 2\,\pi^*(\omega)$. Hence $(X,\alpha)$ is a contact manifold.

We shall adopt 
\begin{equation}
 \label{eqn:volume forms}
\mathrm{d}V_M : = \frac{1}{d!} \, \omega ^{\wedge d} \quad \mathrm{and} \quad
\mathrm{d}V_X : = \frac{1}{ 2\pi }\,\alpha \wedge \pi^* \left(\mathrm{d}V_M \right)
\end{equation}
as volume forms on $M$ and $X$, respectively;
integration will always be meant with respect to the corresponding densities.

Furthermore, $\alpha$ determines an invariant splitting of the tangent bundle of $X$ as
\begin{equation}
 \label{eqn:vertical horizontal tbs}
TX=\mathcal{V} (X / M )\oplus \mathcal{H}(X/M),
\end{equation}
where $\mathcal{V} (X / M ) : = \ker ( \mathrm{ d } \pi )$
is the \textit{vertical} tangent bundle, and $\mathcal{H}(X/M) := \ker ( \alpha )$
is the \textit{horizontal} tangent bundle. Given $V\in \mathfrak{X}(M)$ (the Lie algebra
of smooth vector fields on $M$), we shall denote by $V^\sharp\in \mathfrak{X}(X)$ its horizontal lift to $X$.
If the vector field $ \partial / \partial \theta \in
\mathfrak{X}(X)$ is the 
generator of the structure $S^1$-action, then 
$ \partial _\theta $ spans 
$\mathcal{V} (X / M )$, and 
$\langle \alpha , \partial _\theta \rangle =1$.

The holomorphic structure on $M$, pulled-back to  
$\mathcal{H}(X/M) $, endows $X$ with a CR structure. Explicitly,
the complex structure $J$ on $M$ naturally lifts to a vector bundle endomorphism of $TX$, 
also denoted by $J$, such that $J (\partial_\theta)=0$ and 
\begin{equation}
 \label{eqn:JonX}
J\left(\upsilon^\sharp\right)= J(\upsilon)^\sharp \quad \big(\upsilon \in \mathfrak{X}(M)\big).
\end{equation}

The corresponding
Hardy space $H(X)\subset L^2(X)$ encapsulates the holomorphic structure of $A$ and
its tensor powers. 
%
The corresponding orthogonal projector and its distributional kernel 
are called, respectively, the \textit{Szeg\"{o} projector} and the \textit{Szeg\"{o} kernel} of $X$; they will
be denoted 
\begin{equation}
\label{eqn:Pi and kernel}
 \Pi : L^2(X) \rightarrow H( X ), \quad \Pi (\cdot, \cdot ) \in \mathcal{D}'( X\times X ).
\end{equation}

Consider the unitary group $U(2)$, and its Lie algebra $\mathfrak{ u } ( 2 )$, the space of skew-Hermitian 
$2\times 2$ matrices; in the following,
we shall set $G = U ( 2 )$ and $\mathfrak{ g } = \mathfrak{ u } ( 2 )$
for notational convenience. The standard invariant scalar product 
$\langle \beta _1 , \beta _2 \rangle_{\mathfrak{g}} : = \mathrm{trace} \left( \beta _1 \, \overline{ \beta }_2 ^t \right)$
yields a unitary isomorphism $ \mathfrak{ g } \cong \mathfrak{ g }^\vee $ intertwining the 
adjoint and coadjoint representations of $G$.

Suppose given an holomorphic Hamiltonian action 
$ \mu : G\times M \rightarrow M $ on the K\"{a}hler manifold
$( M , J , 2\,\omega )$, 
with moment map
$ \Phi_G : M \rightarrow \mathfrak{ g }^\vee \cong \mathfrak{g} $. 
For every $\xi \in \mathfrak{ g }$,
let $ \xi _M \in \mathfrak{X} ( M )$ be its associated vector field on $M$.
Then 
\begin{equation}
 \label{eqn:contact lift xi}
\xi_X : = \xi _M ^\sharp - \langle \Phi_G, \xi \rangle \, \partial _\theta
\end{equation}
is a contact vector field on $(X,\alpha)$ \cite{k}, and the map $\xi \mapsto \xi_X$
is an infinitesimal action of $\mathfrak{g}$ on $(X, \alpha)$.

We shall assume that \textit{the latter infinitesimal action can be integrated to an action of
$G$ on $X$}, i.e. that $\mu$ lifts to an action
$\widetilde{\mu} : G\times X\rightarrow X$ preserving the contact and CR structures. 
Then pull-back of functions, given by $g\cdot s := \widetilde{\mu}_{g^{-1}}^* ( s )$, 
is a unitary representation of $G$ on $L^2(X)$ leaving $H(X) \subset L^2( X )$
invariant. This yields a unitary representation
\begin{equation}
 \label{eqn:unitary representation H}
\widehat{\mu} : G \rightarrow U \big ( H(X) \big ).
\end{equation}

By the Theorem of Peter and Weyl (\cite{b-td}, \cite{st-gtp}), $H(X)$ decomposes as a Hilbert space direct sum of
finite-dimensional irreducible representations of $G$. The latter are in 1:1 correspondence with the pairs $\boldsymbol{ \nu } = (\nu _1, \nu_2 )$ of 
integers satisfying $\nu_1 > \nu _2$ \cite{var}; namely, $\boldsymbol{\nu}$ corresponds to 
the irreducible representation 
\begin{equation}
 \label{eqn:explicit irreducible representation}
V_{ \boldsymbol{\nu} } : = { \det } ^{ \nu_2 } \otimes \mathrm{Sym}^{\nu_1 - \nu_2 -1} \left( \mathbb{C} ^2 \right);
\end{equation}
the restriction of its character $\chi_{\boldsymbol{\nu}}$ to the standard torus $T\leqslant G$  
is given by 
\begin{equation}
 \label{explicit character}
\chi_{\boldsymbol{\nu}}:
\begin{pmatrix}
 t_1&0\\
0&t_2
\end{pmatrix}
\mapsto 
\frac{t_1^{\nu_1}\,t_2^{\nu_2}-t_1^{\nu_2}\,t_2^{\nu_1}}{t_1-t_2}.
\end{equation}

 Therefore, there is an equivariant unitary isomorphism
\begin{equation*}
 H(X) \cong \bigoplus _{ \nu_1>\nu_2 } H(X)_{ \boldsymbol{ \nu } },
\end{equation*}
where $H(X)_{ \boldsymbol{ \nu } } \subseteq H(X)$ is the $\boldsymbol{ \nu }$-isotypical component.  
Correspondingly, 
\begin{equation}
 \label{eqn:equivariant szego decomp}
\Pi = \sum _{ \nu_1 > \nu_2} \Pi _{ \boldsymbol{ \nu } },
\end{equation}
where $\Pi _{ \boldsymbol{ \nu } } :L^2(X) \rightarrow H(X)_{ \boldsymbol{ \nu } }$ is the orthogonal projector
(recall (\ref{eqn:Pi and kernel})).

In general, $ H(X)_{ \boldsymbol{ \nu } }$ 
may well be infinite dimensional; however, if $\mathbf{0}\not \in \Phi_G(M)$ then 
$\dim \big( H(X)_{ \boldsymbol{ \nu } } \big) < +\infty$ for every $\boldsymbol{ \nu }$ (see \S 2 of \cite{pao-IJM}). 
In this case,
each $\Pi _{ \boldsymbol{ \nu } }$ is a smoothing operator, with a distributional kernel  
\begin{equation}
 \label{eqn:equivariant szego kernel component}
\Pi_{ \boldsymbol{ \nu } } (\cdot, \cdot) \in \mathcal{C}^\infty (X \times X ).
\end{equation}

In particular,
\begin{equation}
 \label{eqn:dimension trace}
\dim H(X)_{ \boldsymbol{ \nu } } = \int _X \Pi_{ \boldsymbol{ \nu } } ( x, x ) \, \mathrm{d}V_X (x).
\end{equation}

Let us fix a weight $ \boldsymbol{ \nu } \in \mathbb{Z}^2 \setminus \{\mathbf{ 0 }\}$, and 
look at the concentration behavior of $\Pi_{ k \boldsymbol{ \nu } } (\cdot, \cdot)$ when $k\rightarrow +\infty$.
The Abelian analogue of this problem was studied in \cite{pao-IJM} and \cite{pao-loa}.  

\begin{defn}
 \label{defn:coadjoint orbit and cone}
If $ \boldsymbol{ \nu } \in \mathbb{Z}^2 $, let 
$$
D_{  \boldsymbol{ \nu } }
:=
\begin{pmatrix}\nu_1 & 0\\
0&\nu_2\end{pmatrix}.
$$
Let us introduce the following loci.
\begin{enumerate}
 \item $\mathcal{ O }_{ \boldsymbol{ \nu } }\subset \mathfrak{g}$ is
the (co)adjoint orbit of $\imath \, D_{  \boldsymbol{ \nu } }$;
\item $
\mathcal{C} (\mathcal{ O }_{ \boldsymbol{ \nu } } )
:=
\mathbb{R}_+\cdot \mathcal{ O }_{ \boldsymbol{ \nu } }
$
is the cone over $\mathcal{ O }_{ \boldsymbol{ \nu } }$;
\item in $M$ and $X$, respectively, we have the inverse images
$$
M^G_{\mathcal{ O }_{ \boldsymbol{ \nu } }}
: = \Phi _G^{ -1 }\big( \mathcal{C} (\mathcal{ O }_{ \boldsymbol{ \nu } } ) \big),
\quad 
X^G_{\mathcal{ O }_{ \boldsymbol{ \nu } }}
:= \pi^{-1} \left( M^G_{\mathcal{ O }_{ \boldsymbol{ \nu } }} \right).
$$
\end{enumerate}
We shall occasionally write $\mathcal{ O }$ in place of $\mathcal{ O }_{ \boldsymbol{ \nu } }$.
Finally, let us define $\mathcal{C}^\infty$ functions 
$$
m\in M^G_{\mathcal{O}_{ \boldsymbol{ \nu }}}\mapsto h_m\,T\in G/T,\quad 
m\in M^G_{\mathcal{O}_{ \boldsymbol{ \nu }}}\mapsto \lambda_{\boldsymbol{\nu}}(m)\in (0,+\infty)
$$
by the equality
\begin{equation}
 \label{eqn:lambda hm}
\Phi_G(m)=\imath\,\lambda_{\boldsymbol{\nu}}(m)\,h_m D_{\boldsymbol{\nu}}\, h_m^{-1}.
\end{equation}

\end{defn}

Our first result is the following.

\begin{thm}
\label{thm:rapid decrease fixed}
Assume that $\mathbf{0}\not\in \Phi_G(M)$, and $\Phi_G$ is transverse to $\mathcal{C} (\mathcal{ O }_{ \boldsymbol{ \nu } } )$. 
Let us define the $G\times G$-invariant subset of $X\times X$
$$
\mathcal{Z}_{\boldsymbol{ \nu }}:=
\left\{ ( x , y ) \in X^G_{\mathcal{ O }_{ \boldsymbol{ \nu } }}\times X^G_{\mathcal{ O }_{ \boldsymbol{ \nu } }}\, : \, 
y \in G \cdot x \right\}.
$$
Then, uniformly on
compact subsets of
$(X\times X) \setminus \mathcal{Z}_{\boldsymbol{ \nu }}$,
we have 
$$
\Pi_{ k\boldsymbol{ \nu } } (x ,y )= O\left( k^{ -\infty} \right).
$$
\end{thm}

\begin{cor}
 \label{cor:rapid decrease fixed}
Uniformly on compact subsets of $X\setminus X^G_{\mathcal{ O }_{ \boldsymbol{ \nu } }}$, we have
$$\Pi_{ k\boldsymbol{ \nu } } (x ,x )= O\left( k^{ -\infty} \right)
\quad \mathrm{for}\quad k\rightarrow +\infty$$
\end{cor}

The hypothesis of Theorem \ref{thm:rapid decrease fixed} imply that $M^G_{\mathcal{ O }_{ \boldsymbol{ \nu } }}$
is a compact and smooth real hypersurface of $M$.
Our next step will be to clarify the geometry of $M^G_{\mathcal{ O }_{ \boldsymbol{ \nu } }}$. To this end,
we need to introduce some further loci related to the action.

\begin{defn}
 \label{defn:definition of Dnu}
Let
\begin{equation}
 \label{eqn:defn of MGnu}
M^G_{ \boldsymbol{ \nu } }
: = \Phi_G^{ -1 }\big( \imath\,\mathbb{ R }_+ \cdot D_{  \boldsymbol{ \nu } } \big),
\quad 
X^G_{ \boldsymbol{ \nu } } :=
\pi ^{-1}\left( M^G_{ \boldsymbol{ \nu } } \right) .
\end{equation}
\end{defn}

\begin{rem}
 \label{rem:relation MGnuO}
Obviously, $M^G_{ \boldsymbol{ \nu } }\subseteq M^G_{\mathcal{ O }_{ \boldsymbol{ \nu } }}$. 
Under the assumptions of Theorem \ref{thm:rapid decrease fixed},
$M^G_{ \boldsymbol{ \nu } }$ is a compact submanifold of $M$, of real codimension $3$.
Clearly, $M^G_{\mathcal{ O }_{ \boldsymbol{ \nu } }}= G \cdot M^G_{ \boldsymbol{ \nu } }$
by the equivariance of $\Phi_G$ (given a $G$-space $Z$, and a subset $Z_1\subseteq Z$,
we shall denote by $G\cdot Z_1$ the $G$-saturation of $Z_1$ in $Z$).
\end{rem}

Let $T\leqslant G$ be the standard maximal torus of unitary diagonal matrices,
and let $\mathfrak{t}$ be its Lie algebra. Thus $\mathfrak{t}$ is the space of skew-Hermitian diagonal matrices,
and is also $T$-equivariantly identified with the coalgebra $\mathfrak{ t }^\vee$.
In obvious manner $T\cong S^1\times S^1$ and $\mathfrak{t} \cong \imath \, \mathbb{ R }^2$.
We shall alternatively think of elements of $\mathfrak{t}$ either as vectors or as matrices, depending on the context.

Given the isomorphisms $\mathfrak{g}^\vee \cong \mathfrak{g}$ and $\mathfrak{t}^\vee \cong \mathfrak{ t }$,
the restriction epimorphism $\mathfrak{ g } ^\vee \rightarrow \mathfrak{ t } ^\vee$ corresponds to the diagonal map 
\begin{equation}
 \label{eqn:diagonal map}
\mathrm{diag} : \mathfrak{ g } \rightarrow \imath\, \mathbb{ R }^2, \quad 
\imath \, \begin{pmatrix}
 a &  z\\
\overline{z} & b
\end{pmatrix}
\mapsto
\imath \,
\begin{pmatrix}
 a \\
b
\end{pmatrix}
\quad
 (a , b \in \mathbb{ R },\, z\in \mathbb{ C }).
\end{equation}
The action of $T$ on $M$ induced by restriction of $\mu$ is also Hamiltonian, with moment map
\begin{equation}
 \label{eqn:moment map T}
\Phi _T =\mathrm{diag}\circ \Phi_G:M \rightarrow \mathfrak{t}.
\end{equation}

Let us introduce the loci 
\begin{equation}
 \label{eqn:defn di MTnu}
M^T_{ \boldsymbol{ \nu }} := \Phi_T^{-1} \left(\mathbb{ R }_+\cdot \imath\,\boldsymbol{ \nu }\right),
\quad 
X^T_{ \boldsymbol{ \nu }} : = \pi^{-1} \left(M^T_{ \boldsymbol{ \nu }}\right)
\end{equation}
Let us assume that $\mathbf{0}\not\in \Phi_T(M)$ and that $\Phi_T$ is transverse to $\mathbb{ R }_+\cdot \imath\,\boldsymbol{ \nu }$;
then $M^T_{ \boldsymbol{ \nu }}$  is a compact smooth 
real hypersurface of $M$.
Since
$M^G_{ \boldsymbol{ \nu } } \subseteq M^T_{ \boldsymbol{ \nu } }$, we have
$M^G_{\mathcal{ O }_{ \boldsymbol{ \nu } }} \subseteq G\cdot M^T_{ \boldsymbol{ \nu }}$.

In \S \ref{sctn:construction of Upsilon}, we shall construct a vector field 
$\Upsilon=\Upsilon_{\mu,\boldsymbol{ \nu } }$ tangent to $M$ along $M^G_{\mathcal{ O }_{ \boldsymbol{ \nu } }}$,
naturally associated to the action and the weight,
which is nowhere vanishing and everywhere normal to $M^G_{\mathcal{ O }_{ \boldsymbol{ \nu } }}$.

\begin{thm}
 \label{thm:geometry locus}
Let us assume that:
\begin{enumerate}
 \item $\Phi_G:M\rightarrow \mathfrak{g}$ and $\Phi_T:M\rightarrow \mathfrak{t}$ are both transverse to
$\mathbb{R}_+ \cdot \imath\,D_{\boldsymbol{ \nu }}$;
\item $\mathbf{0}\not\in \Phi_T(M)$ (hence also $\mathbf{0}\not\in \Phi_G(M)$);
\item $M^G_{ \boldsymbol{ \nu }}\neq \emptyset$ 
(equivalently, $M^G_{ \mathcal{ O }_{ \boldsymbol{ \nu }}}\neq \emptyset$);
\item $\nu_1+\nu_2\neq 0$.
\end{enumerate}

Then 
\begin{enumerate}
 \item $M^G_{\mathcal{ O }_{ \boldsymbol{ \nu } }}$ is a connected and orientable smooth hypersurface in 
$M$, and separates $M$ in two connected components: the \lq outside \rq \, $A := M\setminus G\cdot M^T_{ \boldsymbol{ \nu }}$ and 
the \lq inside\rq \, $B := G\cdot M^T_{ \boldsymbol{ \nu }} \setminus M^G_{\mathcal{ O }_{ \boldsymbol{ \nu } }}$;

\item the normal bundle to $M^G_{\mathcal{ O }_{ \boldsymbol{ \nu } }}$ in $M$ is the real line sub-bundle
of $\left.TM\right|_{ M^G_{\mathcal{ O }_{ \boldsymbol{ \nu } }} }$ spanned by $\Upsilon$; 

\item
$\Upsilon$ is \lq outer\rq\, oriented if $\nu_1+\nu_2>0$ and \lq inner\rq\, oriented if $\nu_1+\nu_2<0$;

\item $M^G_{\mathcal{O}_{\boldsymbol{ \nu }}}\cap M^T_{{\boldsymbol{ \nu }}} =M^G_{{\boldsymbol{ \nu }}}$, and the two 
hypersurfaces meet tangentially along $M^G_{{\boldsymbol{ \nu }}}$.
\end{enumerate}

\end{thm}

\begin{rem}
 \label{rem:description of B}
Let us clarify the meaning of the partition 
$M= A\dot{\cup} M^G_{\mathcal{O}_{\boldsymbol{ \nu }}} \dot{\cup} B$.
Clearly, $G\cdot M^T_{ \boldsymbol{ \nu }} = \overline{B}$, 
$A=\left(G\cdot M^T_{ \boldsymbol{ \nu }} \right) ^c$.
For any $m\in M$, let $\mathcal{O}_{\Phi(m)} := \Phi_G ( G \cdot m )$
be the coadjoint orbit of $\Phi_G ( m )$, and let $\lambda_1>\lambda_2$
be the eigenvalues of $-\imath\,\Phi_G(m)$;  
as follows either by direct verification or by invoking Horn's Theorem, the projection of $\mathcal{O}_{\Phi(m)}$
in $\mathfrak{t}$ is the segment $J_m$ joining 
$\imath\,
\begin{pmatrix}
 \lambda_1&\lambda_2
\end{pmatrix}^t$ 
and 
$\imath\,
\begin{pmatrix}
 \lambda_2&\lambda_1
\end{pmatrix}^t$.
Then we have:
\begin{enumerate}
 \item $m\in A$ if and only if the orthogonal projection of $\mathcal{O}_{\Phi(m)}$
in $\mathfrak{ t }$, $\mathrm{diag}( \mathcal{O}_{\Phi(m)} )$, is disjoint from 
$\imath\,\mathbb{R}_+ \cdot \boldsymbol{ \nu }$;
\item $m\in M^G_{\mathcal{O}_{\boldsymbol{ \nu }}}$ 
if and only if $\mathrm{diag}( \mathcal{O}_{\Phi(m)} )\cap (\imath\,\mathbb{R}_+ \cdot \boldsymbol{ \nu })$ is 
an endpoint of $J_m$;
\item $m\in B$ if and only if $\mathrm{diag}( \mathcal{O}_{\Phi(m)} )\cap (\imath\,\mathbb{R}_+ \cdot \boldsymbol{ \nu })$ 
is an interior point of $J_m$.
\end{enumerate}

\end{rem}

The next step will be to provide some more precise quantitative information 
on the rate of decay of $\Pi_{ k\boldsymbol{ \nu } }(\cdot,\cdot)$
on the complement of $\mathcal{Z}_{\boldsymbol{ \nu }}$. Namely, we shall show that $\Pi_{ k\boldsymbol{ \nu } }(x,y)$
is still rapidly decreasing when either $y\rightarrow G\cdot x$ at
a sufficiently slow rate, or when at least one of $x$ and $y$ belongs to the \lq outer\rq\, 
component $A$, and converges to $X^G_{\mathcal{ O }_{ \boldsymbol{ \nu } }}$
sufficiently slowly. 

Let us consider on $X$ the Riemannian structure which is uniquely determined by the following conditions:
\begin{enumerate}
\item (\ref{eqn:vertical horizontal tbs}) is an orthogonal direct sum;
\item $\pi:X\rightarrow M$ is a Riemannian submersion;
\item the $S^1$-orbits have unit length.
\end{enumerate}
The corresponding density is $\mathrm{d}V_X$.
Let $\mathrm{dist}_X :X\times X\rightarrow [0,+\infty)$ denote the associated distance function.


\begin{thm}
 \label{thm:rapid decrease slow}
In the situation of Theorem \ref{thm:rapid decrease fixed}, 
assume in addition that $G$ acts freely on $X^G_{\mathcal{O}}$. 
For any fixed $C,\,\epsilon >0$, we have
$\Pi_{ k\boldsymbol{ \nu } }(x,y) = O \left( k^{-\infty} \right)$
uniformly for 
\begin{equation}
 \label{eqn:controlled distance}
\max \left\{ \mathrm{dist}_X ( x , G \cdot y ) , \,
\mathrm{dist}_X \left ( x,  G \cdot X^T_{ \boldsymbol{ \nu } } \right) \right \}
\ge C \, k^{ \epsilon -1/2}.
\end{equation}

\end{thm}

Let us clarify the meaning of Theorem \ref{thm:rapid decrease slow}. 
The closed loci
$\mathcal{R}_k\subset X\times X$ defined by  (\ref{eqn:controlled distance})
form a nested sequence $\mathcal{R}_1 \subseteq \mathcal{R}_2\subseteq \cdots$. 
For any fixed $C, \epsilon >0$ there exist
positive constants $C_j=C_j(C,\epsilon)>0$, $j=1,2,\ldots$, 
such that the following holds. Given any sequence
in $X\times X$ with 
$( x_k , y_k ) \in \mathcal{R}_k$ for $k=1,2,\ldots$, we have 
$$
\big | \Pi_{ k\boldsymbol{ \nu } }(x_k , y_k ) \big| 
\le C_j \, k^{ - j} 
$$
for every $k$.

In Theorems \ref{thm:border pointwise asymptotics} and \ref{thm:border rescaled asymptotics}
below, we shall consider the diagonal and near-diagonal asymptotic behavior of $\Pi_{ k\boldsymbol{ \nu } }$ 
along $X^G_{\mathcal{O}}$. In the setting of Theorem \ref{thm:geometry locus},
every $x\in X^G_{\mathcal{ O }_{ \boldsymbol{ \nu } }}$ has discrete stabilizer subgroup
in $X$. 
To simplify our exposition, we shall make the stronger assumption that $\widetilde{\mu}$ is actually
free along $X^G_{\mathcal{ O }_{ \boldsymbol{ \nu } }}$.
Before giving the statement, some further notation is needed.

\begin{defn}
 \label{defn:vector fields}
If $\xi\in \mathfrak{g}$, we shall denote by
$\xi_M\in \mathfrak{X}(M)$ and $\xi_X\in \mathfrak{X}(X)$ the vector fields
induced by $\xi$ on $M$ and $X$, respectively. If $\boldsymbol{\nu}\in \mathbb{Z}^2$, we have the vector fields
$(\imath\,D_{\boldsymbol{\nu}})_M$ 
and $(\imath\,D_{\boldsymbol{\nu}})_X$; similarly, for any $g\,T\in G/T$, we have the vector fields 
$\mathrm{Ad}_g(\imath\,D_{\boldsymbol{\nu}})_M$ and
$\mathrm{Ad}_g(\imath\,D_{\boldsymbol{\nu}})_X$. To simplify notation, we shall set\footnote{occasionally, we shall use the more precise notation $(\imath\,\boldsymbol{\nu})_M(m)$, but this
should cause no confusion, 
since we are making no explicit use of complexifications in this paper.}
$$
\boldsymbol{\nu}_M:=(\imath\,D_{\boldsymbol{\nu}})_M, \quad \boldsymbol{\nu}_X:=(\imath\,D_{\boldsymbol{\nu}})_X,
$$
and
$$
\mathrm{Ad}_g(\boldsymbol{\nu})_M:=\mathrm{Ad}_g(\imath\,D_{\boldsymbol{\nu}})_M, \quad 
\mathrm{Ad}_g(\boldsymbol{\nu})_X:=\mathrm{Ad}_g(\imath\,D_{\boldsymbol{\nu}})_X.
$$
Occasionally, we shall use the abridged notation $\xi (m)$ for $\xi_M(m)$, $\xi (x)$ for $\xi_X(x)$
with no further mention.
\end{defn}

\begin{defn}
Let $\|\cdot\|_m:T_mM\rightarrow \mathbb{R}$ and $\|\cdot\|_x:T_xX\rightarrow \mathbb{R}$
be the norm functions. If $\boldsymbol{\nu}=(\nu_1,\nu_2)\in \mathbb{Z}^2$, $\nu_1>\nu_2$,
let us set $\boldsymbol{\nu}_\perp:=(-\nu_2,\nu_1)$.
With the notation introduced in Definitions \ref{defn:coadjoint orbit and cone}
and \ref{defn:vector fields}, let us define a $\mathcal{C}^\infty$ function 
$
\mathcal{D}_{\boldsymbol{\nu}}: M^G_{\mathcal{ O }_{ \boldsymbol{ \nu } }}\rightarrow (0,+\infty)
$
by posing 
$$
\mathcal{D}_{\boldsymbol{\nu}} (m):= 
\frac{\|\boldsymbol{\nu}\|}{\left\|\mathrm{Ad}_{h_m}(\boldsymbol{\nu}_\perp)_M(m)\right\|_m}.
$$
\end{defn}

\begin{rem}
 Since by assumption $\widetilde{\mu}$ is locally free on $X^G_{\mathcal{ O }_{ \boldsymbol{ \nu } }}$, but not necessarily 
on $M^G_{\mathcal{ O }_{ \boldsymbol{ \nu } }}$, the latter definition warrants an explanation,
since it might happen that $\xi_M(m)=0$ for $\xi\in \mathfrak{g}$ not zero and $m\in M^G_{\mathcal{ O }_{ \boldsymbol{ \nu } }}$. 
However, if
$x\in X^G_{\mathcal{ O }_{ \boldsymbol{ \nu } }}$ and $m=\pi(x)$, then it follows from (\ref{eqn:contact lift xi})
and the definition of $h_m\,T$ that
$\mathrm{Ad}_{h_m}(\boldsymbol{\nu}_\perp)_X(x)=\mathrm{Ad}_{h_m}(\boldsymbol{\nu}_\perp)_M(m)^\sharp$, whence
$$
\left\|\mathrm{Ad}_{h_m}(\boldsymbol{\nu}_\perp)_M(m)\right\|_m=\|\mathrm{Ad}_{h_m}(\boldsymbol{\nu}_\perp)_X(x)\|_x>0.
$$
\end{rem}

Let us record one more piece of notation. If $V_3$ is the area of the unit sphere $S^3 \subseteq \mathbb{R}^4$,
let us set
$$
D_{ G/T } : = ( 2\,\pi )^{ -1/2 } \, V_3^{-1}.
$$

\begin{thm}
 \label{thm:border pointwise asymptotics}
Under the same hypothesis as in Theorem \ref{thm:geometry locus}, let us assume in addition that
$G$ acts freely on $X^G_{\mathcal{O}_{\boldsymbol{\nu}}}$. 
Then uniformly in $x\in X^G_{\mathcal{O}_{\boldsymbol{\nu}}}$ 
we have for $k\rightarrow +\infty$ 
an asymptotic expansion of the form
\begin{eqnarray*}
 \Pi_{ k\boldsymbol{ \nu } }(x,x) & \sim &
\frac{D_{ G/T }}{\sqrt{ 2 }} \, \frac{1}{ \| \Phi_G ( m ) \|^{ d +1/2 } }
\, \left( \frac{ k \, \| \boldsymbol{ \nu } \| }{ \pi } \right)^{ d -1/2 } \cdot \mathcal{D}_{\boldsymbol{\nu}} (m)
\\
&& \cdot\left[ 1 +\sum_{j \ge 1} k^{- j / 4}\, a_j (\boldsymbol{ \nu } , m ) \right].
\end{eqnarray*}

\end{thm}

We can refine the previous asymptotic expansion at a fixed diagonal point  
$(x,x)\in X^G_{\mathcal{O}_{\boldsymbol{\nu}}} \times X^G_{\mathcal{O}_{\boldsymbol{\nu}}}$
to an asymptotic expansion for near-diagonal rescaled displacements; however, for the sake of simplicity 
we shall restrict the directions of the displacements.

\begin{defn}
 \label{defn:evaluation subspace}
If $m\in M$, let 
$\mathfrak{g}_M(m)\subseteq T_mM$ be the image of the linear evaluation map 
$\mathrm{val}_m : \mathfrak{g} \rightarrow T_mM$, $\xi \mapsto \xi_M(m)$; also, let 
$\mathfrak{g}_M(m)^{\perp_\omega}\subseteq T_mM$  be its symplectic orthocomplement with respect
to $\omega_m$, and let $\mathfrak{g}_M(m)^{\perp_g}\subseteq T_mM$  be its Riemannian orthocomplement with respect
to $g_m$. Hence, 
$$
\mathfrak{g}_M(m)^{\perp_h} := \mathfrak{g}_M(m)^{\perp_\omega} \cap \mathfrak{g}_M(m)^{\perp_g}
\subseteq T_mM
$$
is the Hermitian othocomplement of the complex subspace generated by $\mathfrak{g}_M(m)$
with respect to $h_m := g_m -\imath \, \omega_m$.
\end{defn}

\begin{defn}
 \label{defn:psi2}
If $\mathbf{v}_1, \,\mathbf{v}_2\in T_mM$, following \cite{sz}
let us set 
\begin{equation}
 \label{eqn:defn of psi_2}
\psi_2 (\mathbf{v}_1, \,\mathbf{v}_2 ) :=
-\imath \,\omega_m  (\mathbf{v}_1, \,\mathbf{v}_2 ) -\frac{1}{2} \, \|\mathbf{v}_1 - \mathbf{v}_2\|_m ^2.
\end{equation}
Here $\|\mathbf{v}\|_m := g_m ( \mathbf{v} , \mathbf{v} )^{ 1 /2 }$. The same invariant can be introduced
in any Hermitian vector space. Given the choice of a system of Heisenberg local coordinates centered at 
$x\in X$ (\cite{sz}), there is built-in unitary isomorphism $T_m M  \cong \mathbb{ C } ^d$; with this
implicit, (\ref{eqn:defn of psi_2}) will be used with $\mathbf{v}_j\in \mathbb{C}^d$.
\end{defn}

The choice of Heisenberg local coordinates centered at $x\in X$ gives a meaning to the expression 
$x + ( \theta , \mathbf{ v })$ for $(\theta, \mathbf{v})\in (-\pi,\pi) \times \mathbb{R}^{2 d}$
with $\| \mathbf{v} \|$ of sufficiently small norm. When $\theta =0$, we shall write
$x+\mathbf{v}$.

\begin{thm}
 \label{thm:border rescaled asymptotics}
Let us assume the same hypothesis as in Theorem \ref{thm:border pointwise asymptotics}. 
Suppose $C>0$, $\epsilon \in (0, 1/6)$,
and if $x\in X$ let us set $m_x:= \pi(x)$.
Then, uniformly in $x\in X^G_{\mathcal{O}_{\boldsymbol{\nu}}}$ and $\mathbf{v}_1, \,\mathbf{v}_2
\in \mathfrak{g}_M(m_x)^{\perp_h}$ satisfying $\| \mathbf{v}_j \| \le C \, k^\epsilon$, we have
for $k\rightarrow +\infty$ an asymptotic expansion
\begin{eqnarray*}
\lefteqn{ \Pi_{ k\boldsymbol{ \nu } }\left(x + \frac{1}{\sqrt{ k }} \,\mathbf{v}_1 , x + \frac{1}{\sqrt{ k }} \,\mathbf{v}_2 \right) }\\
& \sim &
\frac{D_{ G/T }}{\sqrt{ 2 }} \, \frac{ e^{  \psi_2 (\mathbf{v}_1 , \mathbf{v}_2) /
 \lambda _{ \boldsymbol{\nu} } ( m_x ) }    }{ \| \Phi_G ( m_x ) \|^{ d +1/2 } }
\, \left( \frac{ k \, \| \boldsymbol{ \nu } \| }{ \pi } \right)^{ d -1/2 } \cdot \mathcal{D}_{\boldsymbol{\nu}} (m_x)
\\
&& \cdot\left[ 1 +\sum_{j \ge 1} k^{- j / 4}\, a_j (\boldsymbol{ \nu } , m_x; \mathbf{v}_1 , \mathbf{ v }_2 ) \right],
\end{eqnarray*}
where $a_j (\boldsymbol{ \nu } , m_x; \cdot , \cdot )$ is a polynomial function of degree $\le \lceil 3 j/2\rceil$.

\end{thm}

Furthermore, we shall provide an integral formula of independent interest for the
asymptotics of $\Pi_{ k\boldsymbol{ \nu } }(x',x')$ when $x'\rightarrow X^G_{\mathcal{O}_{\boldsymbol{\nu}}}$
at a \lq fast\rq\, pace from the \lq outside\rq \, 
(that is, $x'\in \overline{A}$ in the notation of Theorem \ref{thm:geometry locus})
(\S \ref{sctn:integral formula rescaled asymptotics}).
While the latter formula is a bit too technical to be described in this introduction, by global integration
it leads to a lower bound on $\dim H(X)_{ \boldsymbol{ \nu } }$ which can be stated in a compact form.
By (\ref{eqn:dimension trace}), with the notation of 
Theorem \ref{thm:geometry locus}, we have 
\begin{equation}
 \label{eqn:dimension trace split}
\dim H(X)_{ \boldsymbol{ \nu } } = \dim_{ in } H(X)_{ \boldsymbol{ \nu } }  + \dim_{ out } H(X)_{ \boldsymbol{ \nu } },
\end{equation}
where
$$
\dim_{out} H(X)_{ \boldsymbol{ \nu } } : = \int _A \Pi_{ \boldsymbol{ \nu } } ( x, x ) \, \mathrm{d}V_X (x),
$$ 
and similarly for $\dim_{ in } H(X)_{ \boldsymbol{ \nu } }$, with $A$ replaced by $B$.
Hence an asymptotic estimate for $\dim_{out} H(X)_{ k\,\boldsymbol{ \nu } }$ when $k\rightarrow +\infty$
implies an asymptotic lower bound for 
$\dim H(X)_{ k\, \boldsymbol{ \nu } } $. In Theorem \ref{thm:outer dimension estimate} below, 
we shall show that 
$\dim_{out} H(X)_{ k\,\boldsymbol{ \nu } } $ is given by an asymptotic expansion of descending fractional powers of
$k$, the leading power being $k^{ d-1 }$.

\begin{thm}
 \label{thm:outer dimension estimate}
Under the assumptions of Theorem \ref{thm:border pointwise asymptotics}, 
$\dim_{out} H(X)_{ k\,\boldsymbol{ \nu } } $ is given by an asymptotic expansion in descending powers
of $k^{1/4}$ as $k\rightarrow +\infty$, with leading order term
$$
\frac{1}{4} \, 
D_{ G/T } \, 
\,
\left(\frac{ k \, \|\boldsymbol{ \nu } \|}{ \pi }  \right)^{d -1 } \, 
\int_{M^G_{\mathcal{O} }}\,\frac{1}{ \|\Phi_G (m) \|^{d }}\cdot \mathcal{D}_{\boldsymbol{\nu}} (m)
\, \mathrm{ d }V_{M^G_{\mathcal{O} } }(m).
$$
\end{thm}

Let us make some final remarks.

First, there is a wider scope for the results of this paper, 
since it builds on microlocal techniques that can be also applied in the almost complex symplectic setting. 
For the sake of simplicity, we have restricted our discussion to the complex projective setting; 
nonetheless, assuming the theory in \cite{sz} (which in turn builds on \cite{boutet-sjostraend} and \cite{boutet-guillemin}), 
the present results can be extended to the case where $M$ is a compact symplectic manifold with an integral symplectic form 
and a polarizing (or quantizing) line bundle $A$ on it. 
More precisely, given an Hamiltonian compact Lie group action on $M$ linearizing to $A$, 
one can find an invariant compatible almost complex structure, and then rely on the theory of generalized 
Szeg\"{o} kernels developed in \cite{sz} to extend the present arguments and constructions.

In closing, it seems in order to clarify further the relation of the present work to the general literature. 
The asymptotics of Bergman and Szeg\"{o} kernels have attracted significant interest in recent years, 
involving algebraic, complex and symplectic geometry, as well as harmonic analysis. 
Generally, the emphasis has been placed on the perspective of Berezin-Toeplitz quantization, 
where the parameter of the asymptotics is the index of the Fourier component with respect to the structure $S^1$-action. 
Natural variants include additional symmetries, stemming from a linearizable Hamiltonian Lie group action. 
It would be unreasonable for space reasons to give here an account of this body of work, 
but we refer to \cite{boutet-guillemin}, \cite{zelditch-theorem-of-Tian}, \cite{bsz}, \cite{sz}, 
\cite{mm}, \cite{mz}, \cite{schlichenmaier}, \cite{charles}
and references therein. For some interesting recent extensions
in the same spirit to a more abstract geometric setting, see \cite{hhl} and \cite{hh}.   

In particular, the microlocal approach of \cite{boutet-guillemin}, \cite{zelditch-theorem-of-Tian}, \cite{bsz}, \cite{sz}, 
of special relevance for the present work, is based of on the theory of the Szeg\"{o} kernel as a Fourier integral operator 
(see \cite{boutet-sjostraend}), and has been exploited in \cite{pao-adv}, \cite{pao-jsg0} to obtain local asymptotics in the 
$G$-equivariant Berezin-Toeplitz context. 

This said, the perspective of the present work is quite different, and closer in spirit to \cite{guillemin-sternberg hq},
inasmuch as the structure $S^1$-action remains in the background and does not play any privileged role in the asymptotics 
(except of course in defining the underlying geometry); rather, as in \cite{pao-IJM}, the additional symmetry is considered \textit{per se}, 
on the the same footing as the standard circle action in the usual TYZ expansion.  
As in the toric case \cite{pao-IJM}, this changes considerably the geometry of the asymptotics. 

The present work covers part of the PhD thesis of the first author at the University of Milano Bicocca.

\section{Examples}
\label{sctn:example}

\subsection{Example 1}
\label{scnt:esempio1}
Let $A$ be the hyperplane line bundle on $M=\mathbb{P}^3$; then the unit circle bundle $X\subseteq A^\vee\setminus (0)$
 may be identified with $S^7\subset \mathbb{C}^4\setminus\{\mathbf{0}\}$, 
and the projection $\pi:X\rightarrow\mathbb{P}^3$ with the Hopf map.

Consider the unitary representation of $G$ on $\mathbb{C}^4\cong \mathbb{C}^2\oplus\mathbb{C}^2$ given by
 \begin{equation}
  \label{eqn:defn of action}
A\cdot (Z,W)=(AZ,AW);
 \end{equation}
here $Z=(z_1,z_2)^t,\,W=(w_1,w_2)^t\in \mathbb{C}^2$.
 This linear action yields by restriction a contact action $\widetilde{\mu}:G\times S^7\rightarrow S^7$,
 and descends to an holomorphic action $\mu:G\times \mathbb{P}^3\rightarrow \mathbb{P}^3$. 
If $\omega_{FS}$ is the Fubini-Study form on $\mathbb{P}^3$, then $\mu$ is Hamiltonian with respect 
 to $2\,\omega_{FS}$. The moment map 
 is 
 \begin{equation}
  \label{eqn:PhiG}
  \Phi_G:[Z:W]\in \mathbb{P}^3\mapsto 
 \frac{\imath}{\|Z\|^2+\|W\|^2}\,[z_i\,\overline{z}_j+w_i\,\overline{w}_j]\in \mathfrak{g}.
 \end{equation}
Furthermore, $\widetilde{\mu}$ is the contact lift of $\mu$.

 From this, one can draw the following conclusions:
 
 \begin{lem}
  \label{lem:PhiG}
Under the previous assumptions, we have: 
\begin{enumerate}
 \item $-\imath\,\Phi_G([Z:W])$ is a convex linear combination of the orthogonal projections 
 onto the subspaces of $\mathbb{C}^2$ spanned by $Z$ and $W$, respectively;
 \item $-\imath\,\Phi_G([Z:W])$ has rank $2$ if and only
 if $Z$ and $W$ are linearly independent, rank $1$ otherwise;
 \item $\Phi_G(M)=\imath\,K$, where $K$ denotes the set of all positive semidefinite Hermitian matrices of
 trace $1$;
 \item the determinant of $-\imath\,\Phi_G([Z:W])$ is 
 $$
 \det\big(-\imath\,\Phi_G([Z:W])\big)=\frac{|Z\wedge W|^2}{(\|Z\|^2+\|W\|^2)^2},
 $$
where $Z\wedge W=z_1\,w_2-z_2\,w_1\in \mathbb{C}$;
 \item the eigenvalues of $-\imath\,\Phi_G([Z:W])$ are both real and given by
 $$
 \lambda_{1,2}([Z:W])=\frac{1}{2}\,\left(1\pm \sqrt{1-\frac{4\,|Z\wedge W|^2}{(\|Z\|^2+\|W\|^2)^2}}\right).
 $$
\end{enumerate}

 \end{lem}

Let us fix $\boldsymbol{\nu}\in \mathbb{Z}^2$ with 
$\nu_1>\nu_2\ge 0$. 
Let as above $\mathcal{O}_{\boldsymbol{\nu}}\subseteq \mathfrak{g}$ denote the coadjoint orbit of
$\imath\,D_{\boldsymbol{\nu}}$.
With $M=\mathbb{P}^3$, 
the locus $M^G_{\mathcal{O}_{\boldsymbol{\nu}}}=\Phi_G^{-1}(\mathbb{R}_+\cdot \mathcal{O}_{\boldsymbol{\nu}})$ is given by the condition
$$
\nu_2\,\lambda_1([Z:W])-\nu_1\,\lambda_2([Z:W])=0.
$$
In view of Lemma \ref{lem:PhiG}, this implies:

\begin{cor}
 \label{cor:MOG}
 Under the previous assumptions,
 $$
 M^G_{\mathcal{O}_{\boldsymbol{\nu}}}=\left\{[Z:W]\in \mathbb{P}^3\,:
 \,\frac{|Z\wedge W|}{\|Z\|^2+\|W\|^2}=\frac{\sqrt{\nu_1\,\nu_2}}{\nu_1+\nu_2}\right\}.
 $$
\end{cor}

 Let us now consider transversality. 
By Lemma \ref{lem:transversality and locally free action} below (see also the discussion in \S 2 of \cite{pao-IJM}), $\Phi_G$ 
 is transverse to the ray $\mathbb{R}_+\cdot \imath\,D_{\boldsymbol{\nu}}$ in $\mathfrak{g}$ 
 if and only if $\widetilde{\mu}$ is locally free
 along $X^G_{\boldsymbol{\nu}}$ in (\ref{defn:definition of Dnu}) 
(that is, each $x\in X^G_{\boldsymbol{\nu}}$
has discrete stabilizer).
 
On the other hand, by (\ref{eqn:defn of action}) $\widetilde{\mu}$ is locally free at $(Z,W)\in S^7$
 if and only if $Z\wedge W\neq 0$, and this is equivalent to $\Phi([Z:W])$ having rank $2$; 
 this means that $-\imath\,\Phi_G([Z:W])$ has two positive eigenvalues. Thus we obtain the following.

\begin{cor}
 \label{cor:transversality for PhiG}
The following conditions are equivalent:
\begin{enumerate}
 \item $\Phi_G$ is
 transverse to $\mathbb{R}_+\cdot\imath\,D_{\boldsymbol{\nu}}$, and 
 $\Phi_G^{-1}(\mathbb{R}_+\cdot \imath\,D_{\boldsymbol{\nu}})\neq\emptyset$;
 \item $\Phi_G$ is
 transverse to $\mathcal{O}_{\boldsymbol{\nu}}$, and 
 $\Phi_G^{-1}(\mathbb{R}_+\cdot \mathcal{O}_{\boldsymbol{\nu}})\neq\emptyset$;
\item $\nu_1,\,\nu_2> 0$.
\end{enumerate}

\end{cor}

Let us now consider the restricted Hamiltonian action of $T$.  
Identifying $\mathfrak{t}$ with $\imath\,\mathbb{R}^2$, $\Phi_T:M\rightarrow\mathfrak{t}$ 
may be written:
 \begin{equation}
  \label{eqn:phiT}
  \Phi_T:[Z:W]\in \mathbb{P}^3\mapsto 
 \frac{\imath}{\|Z\|^2+\|W\|^2}\,\begin{pmatrix}
                                  |z_1|^2+|w_1|^2\\
                                  |z_2|^2+|w_2|^2
                                 \end{pmatrix}
\in \mathfrak{t}.
 \end{equation}
Thus we obtain
\begin{lem}
 \label{lem:PhiT}
Assume that $\nu_1>\nu_2\ge 0$; then:
\begin{enumerate}
 \item the image of $\Phi_T$ in $\mathfrak{t}\cong\imath\,\mathbb{R}^2$ is
$$
\Phi_T(M)=\imath\,\left\{\begin{pmatrix}
                           x\\
                           y
                          \end{pmatrix}:\,x+y=1,\,x,y\ge 0\right\};
$$
\item the locus $M^T_{\boldsymbol{\nu}}=\Phi_T^{-1}(\mathbb{R}_+\cdot \imath\,D_{\boldsymbol{\nu}})$ is given by
$$
M^T_{\boldsymbol{\nu}}=\left\{[Z:W]\in \mathbb{P}^3\,:\,\nu_2\,\left(|z_1|^2+|w_1|^2\right)=\nu_1\,\left(|z_2|^2+|w_2|^2\right)\right\};
$$
\item $\Phi_T$ is transverse to $\mathbb{R}_+\cdot\imath\,D_{\boldsymbol{\nu}}$ and $M^T_{\boldsymbol{\nu}}\neq \emptyset$ if and only if
$\nu_1,\,\nu_2>0$.
\end{enumerate}

\end{lem}

\begin{proof}
The first two statements follow immediately from (\ref{eqn:phiT}). As to the third,
let us recall again that
$\Phi_T$ is transverse to 
$\mathbb{R}_+\cdot \imath\,D_{\boldsymbol{\nu}}$
if and only if the action of $T$ on $X^T_{\boldsymbol{\nu}}\subset S^7$ is locally free \cite{pao-IJM}.

On the other hand, $T$ acts locally freely at $(Z,W)\in S^7$ if and only if
 $Z$ and $W$ are neither both scalar multiples of $\mathbf{e}_1$, nor both scalar multiples of $\mathbf{e}_2$, 
where $(\mathbf{e}_1,\,\mathbf{e}_2)$ is the standard basis of $\mathbb{C}^2$. 
By 2), there are no points $(Z,W)$ of this form in $X^T_{\boldsymbol{\nu}}$
if and only if $\nu_2>0$.
\end{proof}

Hence if $\nu_1,\,\nu_2>0$, then both $\Phi_G$ and $\Phi_T$ are transverse to $\mathbb{R}_+\cdot \boldsymbol{\nu}$,
and $M^G_{ \boldsymbol{\nu}}\neq \emptyset$, 
$M^T_{ \boldsymbol{\nu}}\neq \emptyset$. For instance, 
$$
\left[\sqrt{\frac{\nu_1}{\nu_1+\nu_2}}\,\mathbf{e}_1:\,\sqrt{\frac{\nu_2}{\nu_1+\nu_2}}\,\mathbf{e}_2\right]
\in M^G_{ \boldsymbol{\nu}}\cap M^T_{ \boldsymbol{\nu}}.
$$
More generally, we have the following.

\begin{lem}
 \label{lem:intersection}
 For any $\boldsymbol{\nu}$,
 $ M^G_{ \boldsymbol{\nu}}\cap M^T_{ \boldsymbol{\nu}}
 =\Phi_G^{-1}\left\{\imath\,
 (\nu_1+\nu_2)^{-1}\,D_{\boldsymbol{\nu}}\right\}$.
\end{lem}

\begin{proof}
By Lemma \ref{lem:PhiG}, $[Z:W]\in M^G_{ \boldsymbol{\nu}}$ if and only if 
$-\imath\,\Phi_G([Z:W])$
is similar to $D_{\boldsymbol{\nu}/(\nu_1+\nu_2)}$; on the other hand,
by Lemma \ref{lem:PhiT},
$[Z:W]\in M^T_{ \boldsymbol{\nu}}$ if and only if
for some $z\in \mathbb{C}$
$$
-\imath\,\Phi_G([Z:W])=
\begin{pmatrix}
 \nu_1/(\nu_1+\nu_2)&z\\
 \overline{z}&\nu_1/(\nu_1+\nu_2)
\end{pmatrix}.
$$
Equaling determinants, we conclude that $z=0$. This concludes the proof.

\end{proof}

Let $\mathfrak{g}_\imath\subseteq \mathfrak{g}$ be the affine hyperplane of the skew-Hermitian
matrices of trace $\imath$; we may interpret $\Phi_G$ as a smooth map 
$\Phi_G':\mathbb{P}^3\rightarrow \mathfrak{g}_\imath$.

\begin{lem}
 \label{lem:regular value}
If $\nu_1>\nu_2>0$, then $
\imath\,(\nu_1+\nu_2)^{-1}\,D_{\boldsymbol{\nu}}\in \mathfrak{g}_\imath
$
is a regular value of $\Phi_G'$.
\end{lem}

\begin{proof}
Clearly, the latter matrix is a regular value of $\Phi'_G$ if and only if $\Phi_G$ is transverse to
the ray $\mathbb{R}_+\cdot \imath\,D_{\boldsymbol{\nu}}$; 
thus the statement follows from Corollary \ref{cor:transversality for PhiG}.
\end{proof}

By Lemmata \ref{lem:intersection} and \ref{lem:regular value}, we obtain

\begin{cor}
 \label{cor:intersection submanifold}
 Suppose $\nu_1>\nu_2>0$. Then, with $M=\mathbb{P}^3$:
 \begin{enumerate}
   \item $M^G_{\mathcal{O}}$ and 
  $M^T_{\boldsymbol{\nu}}$ are smooth
  compact (real) hypersurfaces in $M$;
  \item $M^G_{\mathcal{O}}\cap M^T_{\boldsymbol{\nu}}$ is a smooth submanifold of $M$
  of real codimension $3$.
 \end{enumerate}
\end{cor}

Let us now describe the saturation $G\cdot M^T_{\boldsymbol{\nu}}$. 

\begin{lem}
 \label{cor:MGnu}
 Under the previous assumptions,
 $$
 G\cdot M^T_{\boldsymbol{\nu}}=\left\{[Z:W]\in \mathbb{P}^3\,:
 \,\frac{\|Z\wedge W\|}{\|Z\|^2+\|W\|^2}\le\frac{\sqrt{\nu_1\,\nu_2}}{\nu_1+\nu_2}\right\}.
 $$
\end{lem}

\begin{proof}
 Consider $[Z:W]\in \mathbb{P}^3$ with $(Z,W)\in S^7$. By definition, 
 $
 [Z:W]\in G\cdot M^T_{\boldsymbol{\nu}}
 $
 if and only if there exists $A\in G$ such that
 $[AZ:AW]\in  M^T_{\boldsymbol{\nu}}$; we may actually require without loss
 that $A\in SU(2)$ . Let us write
 $$
 A=
 \begin{pmatrix}
  a&-\overline{c}\\
  c&\overline{a}
 \end{pmatrix}\in SU(2),
 \,\,\,Z=\begin{pmatrix}
        z_1\\
        z_2
       \end{pmatrix},\,\,\,
W=\begin{pmatrix}
        w_1\\
        w_2
       \end{pmatrix};
$$
then 
$[AZ:AW]\in  M^T_{\boldsymbol{\nu}}$ if and only if (with some computations)
\begin{eqnarray}
0&=&\nu_2\,\left(|a\,z_1-\overline{c}\,z_2|^2+|a\,w_1-\overline{c}\,w_2|^2\right)
-\nu_1\,\left(|c\,z_1+\overline{a}\,z_2|^2+|c\,w_1+\overline{a}\,w_2|^2\right)\nonumber\\
&=&\nu_2\,
\left\|
\begin{pmatrix}
 z_1&z_2\\
 w_1&w_2
\end{pmatrix}\,
\begin{pmatrix}
 a\\
 -\overline{c}
\end{pmatrix}
\right\|^2
-
\nu_1\,
\left\|
\begin{pmatrix}
 z_1&z_2\\
 w_1&w_2
\end{pmatrix}\,
\begin{pmatrix}
 c\\
 \overline{a}
\end{pmatrix}
\right\|^2.
\end{eqnarray}
In other words, 
$[Z:W]\in G\cdot M^T_{\boldsymbol{\nu}}$ if and only if there exists 
an orthonormal basis $\mathcal{B}=(V_1,\,V_2)$ of $\mathbb{C}^2$ 
such that
\begin{equation}
 \label{eqn:key relation}
  \nu_2\,
\left\|
\begin{pmatrix}
 z_1&z_2\\
 w_1&w_2
\end{pmatrix}\,
V_1
\right\|^2=
\nu_1\,
\left\|
\begin{pmatrix}
 z_1&z_2\\
 w_1&w_2
\end{pmatrix}\,
V_2
\right\|^2.
\end{equation}

Now for any $V\in \mathbb{C}^2$ we have
\begin{eqnarray*}
 \left\|
\begin{pmatrix}
 z_1&z_2\\
 w_1&w_2
\end{pmatrix}\,
V
\right\|^2&=&V^t\,\begin{pmatrix}
 z_1&w_1\\
 z_2&w_2
\end{pmatrix}\,
\begin{pmatrix}
 \overline{z}_1&\overline{z}_2\\
 \overline{w}_1&\overline{w}_2
\end{pmatrix}\,
\overline{V}\\
&=&V^t\,\frac{1}{\imath}\,\Phi_G([Z:W])\,\overline{V}.
\end{eqnarray*}
If $\lambda_1(Z,W)\ge \lambda_2(Z,W)\ge 0$ are the eigenvalues of $-\imath\,\Phi_G([Z:W])$
(Lemma \ref{lem:PhiG}), we then obtain for any $V\in S^7$
\begin{equation}
 \label{eqn:eigenvalue inequality}
 \lambda_1(Z,W)\ge \left\|
\begin{pmatrix}
 z_1&z_2\\
 w_1&w_2
\end{pmatrix}\,
V
\right\|^2\ge \lambda_2(Z,W),
\end{equation}
with left (respectively, right) equality holding if and only if
$V$ is an eigenvector of $-\imath\,\Phi_G([Z:W])$ relative to $\lambda_1(Z,W)$
(respectively, $\lambda_2(Z,W)$). 
We conclude from (\ref{eqn:key relation}) and (\ref{eqn:eigenvalue inequality}) that if
$(Z,W)\in G\cdot X^T_{\boldsymbol{\nu}}$ then the following inequalities holds:
\begin{equation}
 \label{eqn:inequality}
 \nu_1\, \lambda_1(Z,W) \ge \nu_2\,\lambda_2(Z,W),\quad \nu_2\, \lambda_1(Z,W) \ge \nu_1\,\lambda_2(Z,W).
\end{equation}
While the former is trivial, since $\nu_1>\nu_2>0$ and $\lambda_1(Z,W)\ge \lambda_2(Z,W)\ge 0$, the latter 
is equivalent to the other
\begin{equation}
\label{eqn:key inequality}
  \frac{\sqrt{\nu_1\,\nu_2}}{\nu_1+\nu_2}\ge \|Z\wedge W\|.
\end{equation}

Suppose, conversely, that (\ref{eqn:key inequality}) holds. 
Then (\ref{eqn:inequality}) also holds. Let $(W_1,W_2)$ be an orthonormal basis of eigenvectors of
$-\imath\,\Phi_G\big([Z:W]\big)$ with respect to the eigenvalues $\lambda_1(Z,W) $ and $\lambda_2(Z,W)$,
respectively. Evaluating the two sides of 
(\ref{eqn:key relation}) with $V_1'=W_1$, $V_2'=W_2$ in place of $(V_1,V_2)$ we obtain
\begin{equation*}
  \nu_2\,
\left\|
\begin{pmatrix}
 z_1&z_2\\
 w_1&w_2
\end{pmatrix}\,
V_1'
\right\|^2=\nu_2\, \lambda_1(Z,W)\ge \nu_1\,\lambda_2(Z,W)
=\nu_1\,
\left\|
\begin{pmatrix}
 z_1&z_2\\
 w_1&w_2
\end{pmatrix}\,
V_2'
\right\|^2.
\end{equation*}
Using instead $V_1''=W_2$ and $V_2''=W_1$ in place of $(V_1,V_2)$, we obtain 
\begin{equation*}
  \nu_2\,
\left\|
\begin{pmatrix}
 z_1&z_2\\
 w_1&w_2
\end{pmatrix}\,
V_1''
\right\|^2=\nu_2\, \lambda_2(Z,W)\le \nu_1\,\lambda_1(Z,W)
=\nu_1\,
\left\|
\begin{pmatrix}
 z_1&z_2\\
 w_1&w_2
\end{pmatrix}\,
V_2''
\right\|^2.
\end{equation*}

Since $G=U(2)$ is connected, and acts transitively on the family of all orthonormal basis of $\mathbb{C}^2$,
we conclude by continuity that there exists an orthonormal basis $(V_1,V_2)$ on which (\ref{eqn:key relation}) is satisfied.

\end{proof}

In view of Corollary \ref{cor:MOG}, we deduce

\begin{cor}
 \label{cor:boundary}
 $M^G_{\mathcal{O}_{\boldsymbol{\nu}}}=\partial\left( G\cdot M^T_{\boldsymbol{\nu}}\right)$. 
\end{cor}

The boundary $\partial\left( G\cdot M^T_{\boldsymbol{\nu}}\right)$ consists of those 
$[Z:W]\in \mathbb{P}^3$ such that $-\imath\,\Phi_G([Z:W])$ is similar to $(\nu_1+\nu_2)^{-1}\,D_{\boldsymbol{\nu}}$,
while the interior $\left(G\cdot M^T_{\boldsymbol{\nu}}\right)^0$ consists of those $[Z:W]\in \mathbb{P}^3$ such that
$-\imath\,\Phi_G([Z:W])$ is similar to a matrix of the form
$$
\frac{1}{\nu_1+\nu_2}\,
\begin{pmatrix}
 \nu_1&z\\
 \overline{z}&\nu_2
\end{pmatrix},
$$
for some complex number $z\neq 0$. 

Finally, the locus $X'\subseteq X=S^7$ of those $(Z,W)$ at which $\widetilde{\mu}$ is not locally free
is defined by the condition $Z\wedge W=0$, and therefore it is contained in $\left(G\cdot M^T_{\boldsymbol{\nu}}\right)^0$.
It is the unit circle bundle over a non-singular quadric hypersurface in $\mathbb{P}^3$. The stabilizer subgroup of $(Z,W)\in S^7$
is trivial if $Z\wedge W\neq 0$, and it is isomorphic to $S^1$ otherwise.

For any fixed
$\boldsymbol{\nu}=(\nu_1,\nu_2)\in \mathbb{Z}^2$ with $\nu_1>\nu_2$, 
let consider how $V_{k\boldsymbol{\nu}}$ appears in the isotypical decomposition
of $H\left(X\right)$ under $\widehat{\mu}$ in (\ref{eqn:unitary representation H}). 
The Hopf map $\pi:X=S^7\rightarrow \mathbb{P}^3$ is the quotient map for the standard action
$r:S^1\times S^7\rightarrow S^7\subset\mathbb{C}^4$, given by complex scalar multiplication.  
The corresponding unitary representation of $S^1$ on $H(X)$ yields an
isotypical decomposition $H(X)=\bigoplus_{l\in \mathbb{Z}}H_l(X)$, where for $l\in \mathbb{N}$ we set
$$
H_l(X):=\left\{f\in H(X)\,:\,f\left(e^{i\theta}\,x\right)=e^{\imath\,l\theta}\,f(x)\,\forall\,x=(Z,W)\in X, \,e^{i\theta}\in S^1\right\}.
$$
As is well-known, there are natural $U(2)$-equivariant unitary isomorphisms 
\begin{eqnarray}
 \label{eqn:dec of HlX}
 H_l(X)&\cong& H^0\left(\mathbb{P}^3,\mathcal{O}_{\mathbb{P}^3}(l)\right)
\cong \mathrm{Sym}^{l}\left(\mathbb{C}^2\oplus\mathbb{C}^2\right)\\
&=&\bigoplus_{h=0}^l\mathrm{Sym}^{h}\left(\mathbb{C}^2\right)\otimes \mathrm{Sym}^{l-h}\left(\mathbb{C}^2\right).
\nonumber
\end{eqnarray}

On the other hand, a character computation yields the following. 

\begin{lem}
 \label{lem:character computation}
 For $p\ge q$,
\begin{equation*}
\mathrm{Sym}^{p}\left(\mathbb{C}^2\right)\otimes \mathrm{Sym}^{q}\left(\mathbb{C}^2\right) \cong
\bigoplus_{a=0}^q(\det)^{\otimes a}\otimes \mathrm{Sym}^{p+q-2a}\left(\mathbb{C}^2\right).
\end{equation*}
as $U(2)$-representations. 
\end{lem}

\begin{proof}
[Proof of Lemma \ref{lem:character computation}]
The character of $\mathrm{Sym}^{p}\left(\mathbb{C}^2\right)$ is $\chi_{(p+1,0)}$.
Since the character of a tensor product of representations is the product of the respective characters,
the character of $\mathrm{Sym}^{p}\left(\mathbb{C}^2\right)\otimes \mathrm{Sym}^{q}\left(\mathbb{C}^2\right)$
is $\chi':=\chi_{(p+1,0)}\cdot \chi_{(q+1,0)}$. Let us evaluate $\chi$ on a diagonal matrix $D_{\mathbf{z}}$
with diagonal $\mathbf{z}=(z_1,z_2)$. We obtain
\begin{eqnarray}
 \label{eqn:character product}
 \chi'(D_{\mathbf{z}})&=&\frac{z_1^{p+1}-z_2^{p+1}}{z_1-z_2}\cdot \left(z_1^q+z_1^{q-1}\,z_2+\cdots +z_1\,z_2^{q-1}+z_2^q\right)
 \nonumber\\
 &=&\frac{1}{z_1-z_2}\cdot \left(\sum_{j=0}^{q}z_1^{p+1+q-j}\,z_2^j-\sum_{j=0}^{q}z_1^j\,z_2^{p+1+q-j}\right)\nonumber\\
 &=&\sum_{j=0}^{q}\frac{1}{z_1-z_2}\cdot \left(z_1^{p+1+q-j}\,z_2^j-z_1^j\,z_2^{p+1+q-j}\right)\nonumber\\
 &=&\sum_{j=0}^{q}\chi_{(p+1+q-j,j)}(D_{\mathbf{z}}).
\end{eqnarray}
Now, a character is uniquely determined by its restriction to $T$, and on the other hand
the character of a direct sum is the sum of the characters; therefore,
in view of (\ref{eqn:explicit irreducible representation}), we conclude from (\ref{eqn:character product}) that
\begin{equation*}
 \mathrm{Sym}^{p}\left(\mathbb{C}^2\right)\otimes \mathrm{Sym}^{q}\left(\mathbb{C}^2\right)
 \cong \bigoplus_{j=0}^{q}V_{(p+1+q-j,j)}=\bigoplus_{j=0}^{q}{\det}^{\otimes j}\otimes \mathrm{Sym}^{p+q-2j}\left(\mathbb{C}^2\right).
\end{equation*}

\end{proof}

Therefore,
\begin{eqnarray}
 \label{eqn:Hl fine decomposition}
 H_l(X)\cong \bigoplus_{h=0}^lH_{l,h}(X),
\end{eqnarray}
where we set
\begin{equation}
\label{eqn:defnHlh}
 H_{l,h}(X):=\bigoplus_{a=0}^{\min (h,l-h)}(\det)^{\otimes a}\otimes \mathrm{Sym}^{l-2a}\left(\mathbb{C}^2\right).
\end{equation}

In order for the $a$-th summand in (\ref{eqn:Hl fine decomposition}) to be isomorphic to $V_{k\boldsymbol{\nu}}$,
we need to have $a=k\,\nu_2$ and $l-2a=k\,(\nu_1-\nu_2)-1$; hence in this special case $H(X)_{k\boldsymbol{\nu}}
\subseteq H_l(X)$ with $l=k\,(\nu_1+\nu_2)-1$. Let us estimate the multiplicity of 
$H(X)_{k\boldsymbol{\nu}}$ in $H_l(X)$. In order for the $a$-th summand with $a=k\,\nu_2$ to appear in $H_{lh}(X)$
in (\ref{eqn:defnHlh}) for some $h\le k\,(\nu_1+\nu_2)-1$ we need to have
\begin{eqnarray}
 a=k\,\nu_2&\le& \min \big(h,k\,(\nu_1+\nu_2)-1-h\big)\nonumber\\
&\Rightarrow&
k\,\nu_2\le h,\,\,k\,\nu_2\le k\,(\nu_1+\nu_2)-1-h\nonumber\\
&\Rightarrow&k\,\nu_2\le h\le k\,\nu_1-1.
\end{eqnarray}
Hence there are $k(\nu_1-\nu_2)-1$ values of $h$ for which $H_{l,h}(X)$ contains one copy of
$V_{k\boldsymbol{\nu}}$. The dimension of $H(X)_{k\boldsymbol{\nu}}$ is thus
$\big(k(\nu_1-\nu_2)-1\big)\,k(\nu_1-\nu_2)\sim k^2\,(\nu_1-\nu_2)^2+O(k)$.

\subsection{Example 2}

Next, we shall briefly describe an example on $M=\mathbb{P}^4$,
being much sketchier than in the previous case.
As before, $A$ will denote the hyperplane line bundle, and $X=S^9$ the dual unit circle bundle.

Let us consider the unitary action of $U(2)$ on 
$\mathbb{C}^5\cong \mathbb{C}^2\oplus \mathbb{C}^2\oplus \mathbb{C}$ given by
\begin{equation}
  \label{eqn:defn of action1}
A\cdot (Z,W,t)=(AZ,AW, \det(A)\, t);
 \end{equation}
here $Z=(z_1,z_2)^t,\,W=(w_1,w_2)^t\in \mathbb{C}^2$, $t\in \mathbb{C}$.
 We shall again denote by $\widetilde{\mu}:G\times S^9\rightarrow S^9$,
and $\mu:G\times \mathbb{P}^4\rightarrow \mathbb{P}^4$ the associated contact and Hamiltonian
actions. 
The moment map 
 is now
 \begin{equation}
  \label{eqn:PhiG1}
  \Phi_G:[Z:W:t]\in \mathbb{P}^4\mapsto 
 \frac{\imath}{\|Z\|^2+\|W\|^2+|t|^2}\,[z_i\,\overline{z}_j+w_i\,\overline{w}_j+\delta_{ij}\,|t|^2]\in \mathfrak{g}.
 \end{equation}
Thus $-\imath\,\Phi_G\big([Z:W:t]\big)\ge 0$ is a rescaling of $\|Z\|^2\,p_Z+\|W\|^2\,p_W+|t|^2\,I_2$, and its trace varies in
$[1,2]$. 
In particular, $\mathbf{0}\not\in \Phi_T(M)$.

Now, $(Z,W,t)\in S^9$ has non-trivial stabilizer under $\widetilde{\mu}$
if and only if either 
$t=0$ and $Z\wedge W=0$, or else $Z=W=0$. In the former case, 
$-\imath\,\Phi_G\big([Z:W:t]\big)$ is similar to $D_{(1,0)}$, and in the latter to
$I_2$. Therefore, $\Phi_G$ is transverse to $\mathbb{R}_+\cdot \imath\,D_{\boldsymbol{\nu}}$
for any $\boldsymbol{\nu}$ with $\nu_1>\nu_2>0$. 

Furthermore, if $(Z,W,t)\in S^9$ has non-trivial stabilizer $K$ in $T$  under $\widetilde{\mu}$
then $Z$ and $W$ are either both multiples of $e_1$, in which case $K\leqslant \{1\}\times S^1$,
or both multiples of $e_2$, in which case $K\leqslant S^1\times \{1\}$.  If $t\neq 0$, the condition
$\det(A)=1$ for $A\in K$ implies that $A=I_2$, so $K$ is trivial. If $t=0$, then 
$-\imath\,\Phi_G\big([Z:W:t]\big)$ is either $D_{(1,0)}$ or $D_{(0,1)}$.
On the other hand, if $Z=W=0$, then $-\imath\,\Phi_G\big([Z:W:t]\big)=I_2$. Thus $\Phi_T$ is
transverse to the ray $\mathbb{R}_+\cdot \imath\,\boldsymbol{\nu}$ if $\nu_1>\nu_2>0$.

Let us fix one such
$\boldsymbol{\nu}$, and
look for all the copies of $V_{k\boldsymbol{\nu}}$
within $H(X)\cong \bigoplus_{l=0}^{+\infty} H_l(X)$. 

For any $l=0,1,2,\ldots$, by Lemma \ref{lem:character computation} we have
\begin{eqnarray}
 \label{eqn:dec of HlX1}
 H_l(X)&=&\bigoplus_{p+q+r=l}\mathrm{Sym}^{p}\left(\mathbb{C}^2\right)\otimes \mathrm{Sym}^{q}\left(\mathbb{C}^2\right)
 \otimes {\det}^{\otimes r}\nonumber
\\
 &\cong&\bigoplus_{p+q+r=l}\bigoplus_{a=0}^{\min(p,q)}\mathrm{Sym}^{p+q-2a}\left(\mathbb{C}^2\right)
 \otimes {\det}^{\otimes (a+r)}
\end{eqnarray}

The general summand in (\ref{eqn:dec of HlX1}) is isomorphic to $V_{k\boldsymbol{\nu}}$ if and only if
\begin{equation}
 \label{eqn:pqra}
 a+r= k\nu_2,\quad p+q-2a=k\,(\nu_1-\nu_2)-1.
\end{equation}

Thus for any $r=0,\ldots,k\nu_2$ we can set $a=k\nu_2-r$ and then consider all the pairs $(p,q)$
such that
\begin{equation}
 \label{eqn:pq2r}
 p+q+2r=k\,(\nu_1+\nu_2)-1.
\end{equation}
We see from (\ref{eqn:pq2r}) that 
\begin{equation}
 \label{eqn:range of l}
k\,(\nu_1+\nu_2)-1\ge l=p+q+r=k\,(\nu_1+\nu_2)-1-r\ge k\,\nu_1-1;
\end{equation}
furthermore, equality holds on the left in (\ref{eqn:range of l}) when $r=0$
and on the right when $r=k\,\nu_2$; every intermediate value is assumed. 
Therefore in this case $H(X)_{k\boldsymbol{\nu}}\cap H_l(X)\neq (0)$
for every $l=k\,\nu_1-1,k\,\nu_1,\ldots,k\,(\nu_1+\nu_2)-1$, so that
$H(X)_{k\boldsymbol{\nu}}$ is not a space of sections of any power of $A$.

Finally, we see from (\ref{eqn:pqra}) and (\ref{eqn:pq2r}) that the copies of 
$V_{k\boldsymbol{\nu}}$ within $H(X)$ are in one-to-one correspondence with the 
triples $(p,q,r)$ of natural numbers such that
$0\le r\le k\,\nu_2$ and $p+q=k\,(\nu_1+\nu_2)-2r-1$. 
It follows that
$$
\dim \big(H(X)_{k\boldsymbol{\nu}}\big)= k^3\,\nu_1\,\nu_2\,(\nu_1-\nu_2)+O\left(k^2\right).
$$

\section{Proof of Theorem \ref{thm:rapid decrease fixed}}

\subsection{Preliminaries}

Before delving into the proof, let us collect some useful pieces of notation
and recall some relevant concepts and results.

\subsubsection{The Weyl integration formula}
\label{sctn:weyl}
For the following, see e.g. \S 2.3 of \cite{var}. Let $\mathrm{d}V_G$ and $\mathrm{d}V_T$ denote the Haar measures on $G$ and $T$ respectively
(or the respective smooth densities). 
They determine a \lq quotient\rq\, measure $\mathrm{d}V_{G/T}$ on $G/T$. 

\begin{defn}
 \label{den:weyl_int_formula_ingredients}
Let us define
$\Delta:T\rightarrow \mathbb{C}$ by setting
$$
\Delta (t):= t_1-t_2 \quad \big(t=(t_1,t_2)\in T\big);
$$
here we identify $T$ with $S^1\times S^1$ in the natural manner.

Furthermore, for any $f\in \mathcal{C}^\infty(G)$ let us define $A_f:T\rightarrow \mathbb{C}$
by setting 
$$
A_f(t):=
\int_{G/T}\,f\left(g\,t\,g^{-1}\right)\,\mathrm{d}V_{G/T}(g\,T).
$$
If $f$ is a class function, $A_f(t)=f(t)$ for any $t\in T$.
\end{defn}

Then the following holds.

\bigskip

\textbf{Theorem (Weyl)}
With the assumptions and notation above, 
$$
\int_G\, f(g)\,\mathrm{d}V_G(g)
= \frac{1}{2}\,\int_T\,A_f(t)\,\big|\Delta(t)\big|^2\,\mathrm{d}V_T(t).
$$

\subsubsection{Ladder representations}

For the following concepts, see \cite{guillemin-sternberg hq}. We shall use throughout the identification $T^*G\cong G\times \mathfrak{g}^\vee$
induced by right translations. If $R$ and $S$ are manifolds and $\Lambda\subset T^*R\times T^*S$ is a Lagrangian
submanifold, the corresponding canonical relation is
$$
\Lambda':= \big\{\big((r,\upsilon), (s,-\gamma)\big): 
\big((r,\upsilon), (s,\gamma)\big) \in \Lambda\big\}.
$$

\begin{defn}
 \label{def:chiL}
For every weight $\boldsymbol{ \nu }$, let $\chi_{ \boldsymbol{\nu} }: G \rightarrow \mathbb{C}$ 
be the character
of the associated irreducible representation, and let $d_{ \boldsymbol{ \nu }}=\nu_1-\nu_2$
be the dimension of its carrier space. Let us
denote by $L=L_{\boldsymbol{ \nu }} := (k\,\boldsymbol{ \nu })_{k=0}^{+\infty}$
the \textit{ladder sequence} of weights generated by $\boldsymbol{ \nu }$, and
set 
\begin{equation}
 \label{eqn:character distribution}
 \chi _ L  : =
 \sum_{ k = 1 } ^{ +\infty } d_{ k \boldsymbol{ \nu } } \, \chi_{ k \boldsymbol{ \nu } } \in
 \mathcal{D}' (G).
\end{equation}

\end{defn}

\begin{defn}
 \label{eqn:character lagrangian cone}
For every $f\in \mathcal{C}(\mathcal{O})$, let $G_f\leqslant G$ be the stabilizer subgroup of $f$, and let
$\mathfrak{g}_f\leqslant \mathfrak{g}$ be its Lie algebra. Let
$H_f\leqslant G_f$ be the closed connected codimension-1 subgroup with Lie subalgebra 
$\mathfrak{h}_f=\mathfrak{g}_f\cap f^\perp$.
The locus
\begin{equation}
 \label{eqn:defn Lambda L}
 \Lambda_L  : =  \left\{ ( g , r\,f ) \in G \times \mathfrak{g}^\vee : \,f \in \mathcal{O},\,r>0,\,
g \in H_f \right\} 
\end{equation}
is a Lagrangian submanifold of $T^*G$.
\end{defn}

Then we have the following.
\bigskip

\noindent
\textbf{Theorem}
 (Theorem 6.3 of \cite{guillemin-sternberg hq}) 
 \textit{$\chi_L$ is a Lagrangian distribution on $G$, and its associated conic Lagrangian submanifold of
$T^*G \cong G \times \mathfrak{g}^\vee $ is $\Lambda_L$ in
(\ref{eqn:defn Lambda L}).}

\bigskip

Consider the Hilbert space direct sum
$$
H(X)_L := \bigoplus_{k=1}^{ +\infty } H(X)_{ k\,\boldsymbol{ \nu } },
$$
and let $\Pi_L: L^2(X) \rightarrow L^2(X)_L$ denote the 
corresponding orthogonal projector, $\Pi_L(\cdot, \cdot ) \in \mathcal{D}'(X \times X)$ its Schwartz kernel.
Then 
\begin{equation}
 \label{eqn:integral for P L}
 \Pi_{L } ( x , y ) : =
 \int _G \overline{ \chi_{L } (g) }
 \, \Pi \left( \widetilde{\mu} _{g ^{-1} } ( x ) , y \right)  \, \mathrm{d}V _G(g).
\end{equation}
We shall express (\ref{eqn:integral for P L}) in functorial notation (cfr the discussion on page 374 of \textit{loc. cit.}),
and use basic results on the functorial behavior of wave fronts under pull-backs and push-forwards
(see for instance \S 1.3 of \cite{duist} and \S VI.3 of \cite{guillemin-sternberg ga}) to draw conclusions on the singularities of $\Pi_L$.
 
To this end, let us consider the map 
$$
f: G\times X \times X \rightarrow X\times X,\quad (g, x , y ) \mapsto \left(\widetilde{\mu}_{g^{-1}} ( x ) , y \right)
$$
and the distribution
$\widehat{\Pi} : = f^* (\Pi) \in \mathcal{D}' (G \times X \times X )$. 
Let 
\begin{equation}
 \label{eqn:defn of Sigma}
 \Sigma : =\{ (x, r \,\alpha _x ) : x\in X, r>0 \} \subset T^*X\setminus (0)
\end{equation}
denote the closed symplectic cone sprayed by the connection 1-form; by \cite{boutet-sjostraend}, 
the wave front of $\Pi$ satisfies 
\begin{equation}
 \label{eqn:wave front of Pi}
 \mathrm{WF}' (\Pi) =
\mathrm{diag}(\Sigma) \subset \Sigma\times \Sigma.
\end{equation}
It follows that $\mathrm{WF}' \left( \widehat{\Pi} \right) \subseteq f^*\big(\mathrm{diag}(\Sigma)\big)$.
This implies the following.

\begin{lem}
 \label{lem:canonical relation PiL1}
In terms of the identification $T^*G \cong G \times \mathfrak{g}^\vee$ induced by right translations,
the canonical relation of $\widehat{\Pi}$ is
\begin{eqnarray}
 \label{eqn:defn di LambdaM}
\mathrm{WF}' \left( \widehat{\Pi} \right)
& = & \Big\{ 
 \Big( 
 \big(  g , r\,\Phi_G (m_x ) \big), ( x, r\, \alpha_x ) , 
 (y, r\, \alpha_y ) \Big)\nonumber\\
 && : 
 g \in G, \, x \in X, \, r>0 , \,  y = \widetilde{\mu} _{g ^{-1} } ( x ) 
 \Big\};
\end{eqnarray}
recall that $m_x = \pi (x)$.
\end{lem}

Now let us give the functorial reformulation of (\ref{eqn:integral for P L}). 
Consider the diagonal map 
$$
\Delta:G\times X\times X\rightarrow G\times G\times X\times X,
\quad (g,x,y)\mapsto (g,g,x,y),
$$
and the projection
$$
p:G\times X\times X\rightarrow X\times X, \quad (g,x,y) \mapsto (x,y).
$$

\begin{lem}
 \label{lem:P_L functorial}
 The Schwartz kernel $\Pi_L\in \mathcal{D}'(X\times X)$ is given by
 $$\Pi_L = p_* \left( \Delta ^* \big( \overline{\chi} _L\boxtimes \widehat{\Pi} \big)\right).$$
\end{lem}

Let $\sigma:T^*G\rightarrow T^*G$ be given by
$(g,f)\mapsto (g,-f)$. Then
\begin{eqnarray}
 \label{eqn:wave front product}
 WF (\overline{\chi} _L\boxtimes \widehat{\Pi})
&\subseteq &\Big(\sigma (\Lambda_L)\times (0)\Big)\cup 
 \Big( 
 \sigma (\Lambda_L) \times  \mathrm{WF} \left( \widehat{\Pi} \right)\Big) \cup \Big( (0)\times \mathrm{WF} \left( \widehat{\Pi} \right)\Big)\nonumber\\
& \subset & T^*G\times (T^*G \times T^*X\times T^*X).
 \nonumber
\end{eqnarray}

Therefore, the pull-back $\Delta ^* \big( \overline{\chi} _L\boxtimes \widehat{\mu}\big)$ is well-defined, and
\begin{eqnarray}
 \label{eqn:wave front pull back}
\lefteqn{ WF \left(\Delta ^* \big( \overline{\chi} _L\boxtimes \widehat{\Pi}\big)\right)  \subseteq 
 \mathrm{d}\Delta ^* \big(WF (\overline{\chi} _L\boxtimes \widehat{\Pi})\big) } \\
 & \subseteq & 
 \Big(\sigma (\Lambda_L)\times (0)\Big)\cup 
 \mathrm{d}\Delta ^* \Big( 
 \sigma (\Lambda_L) \times  \mathrm{WF} \left( \widehat{\Pi} \right)   \Big) \cup \mathrm{WF} \left( \widehat{\Pi} \right)
 \nonumber\\
 &\subset & T^*G\times T^*X\times T^*X. \nonumber
\end{eqnarray}
Explicitly, we have 
\begin{eqnarray}
 \label{eqn:pull back product wf}
 \lefteqn{ \mathrm{d}\Delta ^* \Big( 
 \sigma (\Lambda_L) \times  \mathrm{WF} \left( \widehat{\Pi} \right)\Big)} \\
 & = & \Big\{ 
 \Big( 
   \big ( g , -f+ r \, \Phi _G (m_x ) \big), ( x, r \, \alpha_x ) , 
 ( y , - r \, \alpha _y ) \Big)\nonumber\\
 && : f\in \mathcal{C}(\mathcal{O}),\,
 g \in H_f, \, x \in X, \, r>0 , \,  y = \widetilde{\mu} _{g ^{-1} } ( x ) 
 \Big\}.\nonumber
\end{eqnarray}

Using that $\Phi_G $ is nowhere vanishing, we can now apply Proposition 1.3.4 of \cite{duist} to conclude the following.

\begin{cor}
 \label{cor:wave front of PL}
 The wave front $WF (\Pi_L)\subseteq \left(T^*X\setminus (0)\right)\times \left(T^*X\setminus (0)\right)$ 
 of the distributional kernel $\Pi_L$ satisfies 
\begin{eqnarray*}
\lefteqn{ WF \big(\Pi_L\big) }	\\
& \subseteq & \Big\{ 
 \Big(  ( x, r \, \alpha_x ) , 
 ( y , - r \, \alpha _y )  \Big): f:=\Phi_G ( x )\in \mathcal{C}(\mathcal{O}),\,
y \in H_f\cdot x  \Big\},
\end{eqnarray*}
where $H_f\cdot x$ is the $H_f$-orbit of $x$.
\end{cor}

\begin{cor}
 \label{cor:singular support of PL}
Let $SS (\Pi_L)\subseteq X\times X$ 
be the singular support of the distributional kernel $\Pi_L$. Then
%
 $SS \big(\Pi_L\big) \subseteq \mathcal{Z}_{\boldsymbol{ \nu }}$.
%
\end{cor}

\subsection{The proof}

\begin{proof}[Proof of Theorem \ref{thm:rapid decrease fixed}]
For every $\boldsymbol{\mu}=(\mu_1,\mu_2)\in \mathbb{Z}^2$ with $\mu_1>\mu_2$, let
$P_{ \boldsymbol{\mu} } : L^2( X ) \rightarrow L^2( X )_{ \boldsymbol{\mu} }$
be the orthogonal projector.
Clearly 
\begin{equation}
 \label{eqn:expression for piknu}
 \Pi_{k \boldsymbol{\nu} } = P_{k \boldsymbol{\nu} } \circ \Pi_L.
\end{equation}
In terms of Schwartz kernels, (\ref{eqn:expression for piknu}) can be reformulated as follows:
\begin{eqnarray}
 \label{eqn:expressione for piknu kernel}
 \Pi_{k \boldsymbol{\nu} } (x,y) & = & 
 d_{k \boldsymbol{\nu} }\, \int_G \, \mathrm{d}V_G (g)\left[\overline{\chi_{k \boldsymbol{\nu} }(g)} \,
 \Pi_L \left( \widetilde{\mu}_{g^{-1}} (x) , y \right) \right].
\end{eqnarray}
Using the Weyl integration, character and dimension formulae, 
(\ref{eqn:expressione for piknu kernel}) can in turn be rewritten as follows:
\begin{eqnarray}
 \label{eqn:expressione for piknu kernel weyl}
\lefteqn{ 
\Pi_{k \boldsymbol{\nu} } (x,y)} \\
& = & 
 \frac{k \, (\nu_1 -\nu_2 )}{(2\pi)^2}\, \int_{(-\pi,\pi)^2}\,\mathrm{d}\boldsymbol{\vartheta}\, 
 \left[e^{-\imath\,k \langle \boldsymbol{ \nu },\boldsymbol{ \vartheta} \rangle} \, 
 \left(e^{\imath\,\vartheta_1}
 -e^{\imath\,\vartheta_2}\right)\,F_{L} \left(x,y;e^{\imath\,\boldsymbol{\vartheta}}\right) \right],
 \nonumber
\end{eqnarray}
where for $t \in T$ we set 
\begin{equation}
 \label{eqn:defn of FL}
 F_{L} (x,y;t) : = \int_{G/T}\,\mathrm{d}V_{G/T} (gT)\, \left[
\Pi_L \left( \widetilde{\mu}_{g\,t^{-1} \, g^{-1}} (x) , y \right) \right].
\end{equation}

Now suppose $K\Subset (X\times X) \setminus \mathcal{Z}_{\boldsymbol{ \nu }}$. 
We may assume without loss that $K$ is $G\times G$-invariant.
There exist $G\times G$-invariant open subsets $A,B \subset X\times X$ such that 
$$
K\subset A\Subset (X\times X) \setminus \mathcal{Z}_{\boldsymbol{ \nu }}, \quad
\mathcal{Z}_{\boldsymbol{ \nu }}\subset B\Subset (X\times X)\setminus K, \quad X\times X= A\cup B.
$$
Hence $A$ is a $G\times G$-invariant open neighborhood of $K$ in $X\times X$, and the restriction of $\Pi_L$ to $A$ is
$\mathcal{C}^\infty$. 

Therefore, we get a $\mathcal{C}^\infty$ function 
$$
R:T\times G/T\times A\rightarrow \mathbb{C},
\quad 
\big(t, gT, (x,y) \big ) \mapsto \Pi_L \left( \widetilde{\mu}_{g\,t^{-1} \, g^{-1}} (x) , y \right).
$$
With $F_L$ as in (\ref{eqn:defn of FL}), we obtain a $\mathcal{C}^\infty$ function on $T\times A$
by setting
$$
\beta: \big(t,(x,y) \big)\mapsto  \Delta (t) \,F_{L} (x,y;t).
$$
Let us denote by $\mathcal{F}_T$ the Fourier transform with respect to $t\in T$ of
a function on $T\times A$, viewed as a function on 
$\mathbb{Z}^2\times A$; then (\ref{eqn:expressione for piknu kernel weyl}) may be rewritten
\begin{eqnarray}
 \label{eqn:expressione fourier for piknu kernel weyl}
\Pi_{k \boldsymbol{\nu} } (x,y)
& = & 
 \frac{k }{2}\, (\nu_1 -\nu_2 )\cdot 
 \mathcal{F}_T(\beta) (k\,\boldsymbol{\nu}; x,y).
\end{eqnarray}
The statement of Theorem \ref{thm:rapid decrease fixed} follows from (\ref{eqn:expressione fourier for piknu kernel weyl}) and the previous considerations.

\end{proof}

\section{Proof of Theorem \ref{thm:geometry locus}}

We shall assume throughout this section that 
the assumptions of Theorem \ref{thm:geometry locus} hold.

\subsection{Preliminaries}

Before attacking the proof, it is in order to list some useful preliminaries (see also the discussion in \S 2 of \cite{pao-IJM}).

For any $m\in M$, let 
$\mathrm{val}_m:\mathfrak{g}\rightarrow T_mM$ be the evaluation map $\xi\mapsto \xi_M(m)$; 
similarly, for any $x\in X$ let $\mathrm{val}_x:\mathfrak{g}\rightarrow T_xX$ be the evaluation map $\xi\mapsto \xi_X(x)$.

\subsubsection{Ray transversality and locally free actions}

Since $\widetilde{\mu}$  preserves the connection 1-form,
the induced cotangent action of $G$ on $T^*X$ leaves the 
symplectic cone $\Sigma$ in (\ref{eqn:defn of Sigma}) invariant. 
The restricted action is of course still Hamiltonian, and
its moment map $\widetilde{\Phi}_G:\Sigma\rightarrow \mathfrak{g}$
is the restriction to $\Sigma$ of the cotangent Hamiltonian map on $T^*X$.

If $m\in M^G_{\mathcal{O}}$, then by equivariance $\Phi_G$ is transverse to $\mathbb{R}_+\cdot \Phi_G(m)$.
Hence, 
\begin{equation}
 \label{eqn:transversality on M}
\mathrm{d}_m\Phi_G(T_mM)+\mathrm{span}\big(\Phi_G(m)\big)=
\mathfrak{g}.
\end{equation}

Suppose $x\in \pi^{-1}(m)\subset X$ and $r>0$, and consider $\sigma=(x,r\alpha_x)\in \Sigma$.
Then it follows from (\ref{eqn:transversality on M}) that 
\begin{equation}
 \label{eqn:transversality on X}
\mathrm{d}_\sigma\widetilde{\Phi}_G(T_\sigma\Sigma)=\mathrm{d}_m\Phi_G(T_mM)+\mathrm{span}\big(\Phi_G(m)\big)=
\mathfrak{g}.
\end{equation}
Thus $\widetilde{\Phi}_G$ is submersive at any $(x,r\alpha_x)$ with $x\in X^G_{\mathcal{O}}$.
If we let $\Sigma^G_{\mathcal{O}}\cong X^G_{\mathcal{O}}\times \mathbb{R}_+$ denote the inverse image of 
$X^G_{\mathcal{O}}$ in $\Sigma$, we conclude therefore that $G$ acts locally freely on $\Sigma^G_{\mathcal{O}}$,
and this clearly implies that it acts locally freely on $X^G_{\mathcal{O}}$.

The previous implications may obviously be reversed, and we obtain the following.

\begin{lem}
 \label{lem:transversality and locally free action}
The following conditions are equivalent:
\begin{enumerate}
 \item $\Phi_G$ is transverse to $\mathbb{R}_+\cdot \imath \,\boldsymbol{\nu}$;
\item $\widetilde{\mu}$ is locally free on $X^G_{\mathcal{O}}$;
\item for every $x\in X^G_{\mathcal{O}}$, $\mathrm{val}_x$ is injective;
\item for every $m\in M^G_{\mathcal{O}}$, $\mathrm{val}_m$ is injective on $\Phi_G(m)^{\perp_{\mathfrak{g}}}$.
\end{enumerate}
\end{lem}

\subsubsection{The vector field $\Upsilon=\Upsilon_{ \mu,\boldsymbol{ \nu }}$}

\label{sctn:construction of Upsilon}

Let us construct the normal vector field $\Upsilon=\Upsilon_{ \mu,\boldsymbol{ \nu }}$ 
to $M^G_{\mathcal{O}}$ appearing in the statement of Theorem \ref{thm:geometry locus}.

By definition, $m\in M^G_{\mathcal{ O }_{ \boldsymbol{ \nu } }}$ if and only if $\Phi_G(m)$ is similar to
 $\imath\,\lambda_{\boldsymbol{ \nu } } ( m )\,D_{\boldsymbol{\nu}} $,
for some $\lambda_{\boldsymbol{ \nu } } ( m )>0$ (Definition \ref{defn:definition of Dnu}). Equating norms and traces, we obtain
\begin{equation}
 \label{eqn:defn di lambda}
\lambda_{\boldsymbol{ \nu } } ( m ) = \frac{\| \Phi_G ( m ) \|}{\| \boldsymbol{ \nu } \|}
= -\imath \,
\frac{\mathrm{trace}\big(\Phi_G(m)\big)}{\nu_1+\nu_2}
\quad \left( m\in M^G_{\mathcal{ O }_{ \boldsymbol{ \nu } }} \right).
\end{equation}
Since $\nu_1>\nu_2$, there exists a unique coset $h_m\,T\in G/T$ such that

\begin{equation}
 \label{eqn:similitude phinu}
\Phi_G (m) = \imath\,\lambda_{\boldsymbol{ \nu } } ( m ) \,
h_m \, D_{\boldsymbol{\nu}} \, h_m^{-1}.
\end{equation}
Let us set $\boldsymbol{ \nu }_\perp  : = \begin{pmatrix} -\nu_2 & \nu_1 \end{pmatrix}^t$,
and define $\boldsymbol{ \rho }= \boldsymbol{ \rho }_{ \boldsymbol{ \nu }}: M^G_{\mathcal{ O }_{ \boldsymbol{ \nu } }}
\rightarrow \mathfrak{g}$ by setting 
\begin{equation}
 \label{eqn:defn di rhonu}
\boldsymbol{ \rho } (m ) := \imath \, h _m \, D_{\boldsymbol{ \nu }_\perp } \,h_m^{-1} \quad \left(
m \in M^G_{\mathcal{ O }_{ \boldsymbol{ \nu } }} \right).
\end{equation}
Then $\boldsymbol{ \rho } (m )_M\in \mathfrak{X}(M)$ is the 
vector field on $M$ induced by $\boldsymbol{ \rho } (m )\in \mathfrak{g}$;
its evaluation at $m'\in M$ is $\boldsymbol{ \rho } (m )_M(m')$ (and similarly for $X$). 

\begin{defn}
 \label{defn:defn di Upsilon}
The vector field $\Upsilon=\Upsilon_{\mu, \boldsymbol{ \nu } }$ along $M^G_{\mathcal{ O }_{ \boldsymbol{ \nu } }}$
is 
\begin{equation*}
\Upsilon (m) := J_m \big (\boldsymbol{ \rho } (m )_M (m)\big)  \quad \left( m\in M^G_{\mathcal{ O }_{ \boldsymbol{ \nu } }}\right).
\end{equation*}
With abuse of notation, recalling (\ref{eqn:JonX}) 
we shall also denote by $\Upsilon$ the vector field along $X^G_{\mathcal{ O }_{ \boldsymbol{ \nu } }}$
given by
\begin{equation*}
\Upsilon (x) := J_x \big (\boldsymbol{ \rho } (m_x )_X (x)\big), \quad m_x:=\pi(x).
\end{equation*}
\end{defn}

Notice that 
\begin{equation}
 \label{eqn:paring rhophi}
\langle \Phi _G(m), \boldsymbol{\rho} (m)\rangle =\lambda_{\boldsymbol{ \nu } } ( m ) \,
\left\langle \boldsymbol{\nu}, \boldsymbol{\nu}_\perp\right\rangle =0  
\quad \left( m\in M^G_{\mathcal{ O }_{ \boldsymbol{ \nu } }}\right).
\end{equation}
Therefore, in view of (\ref{eqn:contact lift xi}) 
for any 
$x\in \pi^{-1}(m)$ we have 
\begin{equation}
 \label{eqn:rhoMhorizontal}
\boldsymbol{ \rho } (m )_X (x)=\boldsymbol{ \rho } (m )_M^\sharp(x)=\boldsymbol{ \rho } (m )_M(m)^\sharp(x).
\end{equation}
Hence $\Upsilon (x)=\Upsilon(m)^\sharp$ if $m=\pi(x)$.

\subsubsection{A spectral characterization of $G\cdot M^T_{\boldsymbol{\nu}}$}

Suppose that $-\imath\,\Phi_G(m)$ has eigenvalues $\lambda_1(m) \ge \lambda_2(m)$.
Then $m\in M^G_{\mathcal{O}}$ if and only if $\lambda_1(m)\nu_2- \lambda_2(m)\,\nu_1=0$.
We shall give a similar spectral characterization of $G\cdot M^T_{\boldsymbol{\nu}}$.
Notice that if $\lambda_1(m) = \lambda_2(m)$, then $\Phi_G(m)$ is a multiple of the identity,
hence certainly $m\not\in G\cdot M^T_{\boldsymbol{\nu}}$.
Thus we may as well assume that $\lambda_1(m) > \lambda_2(m)$.

\begin{prop}
 \label{prop:spectral characterization}
Suppose $m\in M$, and let the eigenvalues of $-\imath\,\Phi_G(m)$ be 
$\lambda_1(m) > \lambda_2(m)$. Then $m\in G\cdot M^T_{\boldsymbol{\nu}}$
if and only if
\begin{equation}
 \label{eqn:spectral inequality}
t(m,\boldsymbol{\nu}):=
\frac{\lambda_1(m) \,\nu_2- \lambda_2(m)\,\nu_1}{(\nu_1+\nu_2) \,\big(\lambda_1(m) - \lambda_2(m)\big)} \in [0,1/2).
\end{equation}

\end{prop}

\begin{proof}
 [Proof of Proposition \ref{prop:spectral characterization}]
Let us set $\boldsymbol{\lambda}(m):=\big( \lambda_1(m),\,\lambda_2(m) \big)$,
and let $D_{\boldsymbol{\lambda}}$ be the corresponding diagonal matrix.
By definition, $m\in G\cdot M^T_{\boldsymbol{\nu}}$ if and only if there exists
$g\in SU(2)\leqslant G$ such that $\mathrm{diag}\left( g \,D_{\boldsymbol{\lambda}} \,g^{-1}\right)
\in \mathbb{R}_+\cdot \boldsymbol{\nu}$. This is equivalent to the condition that there exist 
$u,w\in \mathbb{C}$ such that 
\begin{eqnarray}
\label{eqn:ginSU(2)}
\begin{pmatrix}
 u & -\overline{w}\\
w & \overline{u}
\end{pmatrix}
\,
D_{\boldsymbol{\lambda}}\,
\begin{pmatrix}
 \overline{u}& \overline{w}\\
-w &u
\end{pmatrix}=
c \,
\begin{pmatrix}
 \nu_1& a\\
\overline{a}&\nu_2
\end{pmatrix},
\end{eqnarray}
for some $c>0$ and $a\in \mathbb{C}$. If we set $t=|w|^2$, we conclude that $m\in M^G_{\mathcal{O}}$ if and only if there exists
$t\in [0,1]$ such that 
\begin{equation}
\label{eqn:lambdat}
\boldsymbol{\lambda}_t(m):= 
\begin{pmatrix}
  (1-t)\,\lambda_1(m) + t \,\lambda_2(m) \\
 t\,\lambda_1(m) +  (1-t)\,\lambda_2(m)
 \end{pmatrix}\in \mathbb{R}_+\,
\begin{pmatrix}
 \nu_1\\
\nu_2
\end{pmatrix}.
\end{equation}

The condition $\boldsymbol{\lambda}_t(m)\wedge \boldsymbol{\nu}=\mathbf{0}$ translates into the equality
$t = t(m,\boldsymbol{\nu})$. Hence we need to have $t(m,\boldsymbol{\nu})\in [0,1]$. 
Given this, $\boldsymbol{\lambda}_t(m)$ is a \textit{positive} multiple of
$\boldsymbol{\nu}$ if and only if 
$$\big(1-t(m,\boldsymbol{\nu})\big)\,\lambda_1(m) + t(m,\boldsymbol{\nu}) \,\lambda_2(m) >
 t(m,\boldsymbol{\nu})\,\lambda_1(m) +  \big(1-t(m,\boldsymbol{\nu})\big)\,\lambda_2(m),
$$ 
and this is equivalent to $t(m,\boldsymbol{\nu})<1/2$.

Conversely, suppose that $t(m,\boldsymbol{\nu})\in [0,1/2)$, and define
$$
g:=
\begin{pmatrix}
 \sqrt{1-t(m,\boldsymbol{\nu})}&-\sqrt{t(m,\boldsymbol{\nu})}\\
\sqrt{t(m,\boldsymbol{\nu})}& \sqrt{1-t(m,\boldsymbol{\nu})}
\end{pmatrix}.
$$

\end{proof}

\subsection{The Proof}

\begin{proof}[Proof of Theorem \ref{thm:geometry locus}]
 
As $\Phi_G$ is equivariant, it is transverse to $\mathbb{R}_+\cdot \imath\,D_{\boldsymbol{\nu}}$
if and only if it is transverse to $\mathbb{R}_+\cdot \mathcal{O}$. Given that $\nu_1>\nu_2$,
$\mathcal{O}$ is 2-dimensional (and diffeomorphic to $S^2$); therefore, $\mathbb{R}_+\cdot \mathcal{O}$
has codimension $1$ in $\mathfrak{g}$. Similarly, $\mathbb{R}_+\cdot \imath\,D_{\boldsymbol{\nu}}$
has codimension $1$ in $\mathfrak{t}^\vee$. 
Given that $\mathbf{0}\not\in \Phi_T(M)$, we conclude the following.

\begin{step}
 \label{step:MGOsmooth}
$M^G_{\boldsymbol{\nu}}$, $M^G_{\mathcal{O}}$ and $M^T_{\boldsymbol{\nu}}$ are compact and smooth 
(real) submanifolds of $M$. $M^G_{\boldsymbol{\nu}}$ has codimension $3$, and 
$M^G_{\mathcal{O}}$ and $M^T_{\boldsymbol{\nu}}$ are hypersurfaces.
\end{step}

The Weyl chambers in $\mathfrak{t}$ are the half-planes 
$$
\mathfrak{t}_+:=\big\{ \boldsymbol{\mu} \,: \,
\mu_1>\mu_2\big\},\quad 
\mathfrak{t}_-:=\big\{ \boldsymbol{\mu} \,: \,
\mu_1<\mu_2\big\},
$$
and clearly with our identifications $\imath\,D_{\boldsymbol{\nu}}\leftrightarrow\boldsymbol{\nu}\in \mathfrak{t}_+$.
Since $\Phi_G(M)\cap \mathfrak{t}_+$ is a convex polytope (\cite{guillemin-sternberg cp1}, \cite{guillemin-sternberg cp2}, \cite{kirwan}),
$\Phi_G(M)\cap \mathbb{R}_+\cdot\imath\,D_{\boldsymbol{\nu}}$ is a closed segment $J$.
Furthermore, for any $a\in J$, the inverse image $\Phi_G^{-1}(a)\subseteq M$ is also connected 
(\cite{kirwan-cohomology}, \cite{lerman}).
Thus we obtain the following conclusion.

\begin{step}
 \label{step:MGOconnected}
$M^G_{\boldsymbol{\nu}}$, $M^G_{\mathcal{O}}$ and $M^T_{\boldsymbol{\nu}}$ are connected.
\end{step}

\begin{proof}[Proof of Step \ref{step:MGOconnected}]
The previous considerations immediately imply that $M^G_{\boldsymbol{\nu}}$ is connected. Given this,
since $M^G_{\mathcal{O}}=G\cdot M^G_{\boldsymbol{\nu}}$ the connectedness of $G$ implies the one of 
$M^G_{\mathcal{O}}$. Let us consider $M^T_{\boldsymbol{\nu}}$. Since $\Phi_T(M)$ is a convex polytope
(\cite{guillemin-sternberg cp1}, \cite{atiyah}), $\Phi_T(M)\cap \mathbb{R}_+\cdot\imath\,D_{\boldsymbol{\nu}}$ 
is also a closed segment $J'$. The statement follows since the fibers of $\Phi_T$ 
are connected again by \cite{kirwan-cohomology}, \cite{lerman}.
\end{proof}

 For any $m\in M^G_{\mathcal{O}}$, let us set 
\begin{equation*}
 M^G_{\Phi_G(m)}:=\Phi_G^{-1}\big(\mathbb{R}_+\cdot \Phi_G(m)\big).
\end{equation*}

Since $\Phi_G$ is transverse to $\mathbb{R}_+\cdot \boldsymbol{\nu}$, by equivariance it is also transverse to
$\mathbb{R}_+\cdot \Phi_G(m)$; hence $M^G_{\Phi_G(m)}$ is also a connected real submanifold of $M$, of real codimension
$3$ and contained in $M^G_{\mathcal{O}}$.

Let us consider the normal bundle $N\big(M^G_{\Phi_G(m)}\big)$ to $M^G_{\Phi_G(m)}\subset M$. For any $\xi \in \mathfrak{g}$, 
let $\xi^\perp\subset \mathfrak{g}$
be the orthocomplement to $\xi$. Under the equivariant identification $\mathfrak{g}\cong \mathfrak{g}^\vee$, 
$\xi^\perp$ corresponds to $\xi^0$.

For any subset $L\subseteq \mathfrak{g}$, let $L^{\perp_{\mathfrak{g}}}$ denote the orthocomplement
of $L$ (that is, of the linear span of $L$) under the pairing $\langle \cdot , \cdot \rangle_{\mathfrak{g}}$.

\begin{lem}
 \label{lem:normal to MGnu}
For any $m\in M^G_{\mathcal{O}}$, we have 
$$
N_m\left(M^G_{\Phi_G(m)}\right)= J_m\circ \mathrm{val}_m\left(\Phi_G(m)^{\perp_{\mathfrak{g}}} \right).
$$
Simlarly, for any $m\in M^T_{\boldsymbol{\nu}}$, we have 
$$
N_m\left(M^T_{\boldsymbol{\nu}}\right)= J_m\circ \mathrm{val}_m\left((\imath\,\boldsymbol{\nu})^{\perp_{\mathfrak{t}}} \right)
.
$$
\end{lem}

\begin{proof}[Proof of Lemma \ref{lem:normal to MGnu}]
If $v\in T_mM^G_{\Phi_G(m)}$, then 
$\mathrm{d}_m\Phi_G(v)= a\,\Phi_G(m)
$ for some $a\in \mathbb{R}$.
Given $\eta\in \Phi_G(m)^{\perp_{\mathfrak{g}}}$, and with $\rho$ as in (\ref{eqn:riemannian structure}), we have 
\begin{eqnarray*}
 \rho_m \Big(J_m \big(\eta_M(m)\big), v\Big)& = & \omega_m \big(\eta _M(m), v \big) = \mathrm{d}_m\Phi^\eta (v)\\
& = & \langle \mathrm{d}_m\Phi(v), \eta \rangle_{\mathfrak{g}} 
= a \langle \Phi_G(m), \eta\rangle _{\mathfrak{g}} =0.
\end{eqnarray*}
Therefore, $J_m\circ \mathrm{val}_m\left(\Phi_G(m)^{\perp_{\mathfrak{g}}} \right)\subseteq N_m(M^G_{\Phi_G(m)})$. Since both 
$\Phi_G(m)^{\perp_{\mathfrak{g}}}$ and $N_m(M^G_{\Phi_G(m)})$ are $3$-dimensional, it suffices to recall that by Lemma 
\ref{lem:transversality and locally free action} 
$\mathrm{val}_m$ is injective when restricted to $\Phi_G(m)^{\perp_{\mathfrak{g}}}$.

The proof of the second statement is similar.

\end{proof}

For any vector subspace $L\subseteq \mathfrak{g}$, let us set 
$L_M(m):=\mathrm{val}_m \big(L)\subseteq T_mM$ ($m\in M$).
For any $m\in M^G_{\mathcal{O}}$, given that $M^G_{\mathcal{O}}$ is the $G$-saturation of 
$M^G_{\Phi_G(m)}$, we have 
\begin{equation}
 \label{eqn:tangent space to MGO}
T_mM^G_{\mathcal{O}} = T_mM^G_{\Phi_G(m)} + \mathfrak{g}_M(m). 
\end{equation}
Therefore, passing to $\rho_m$-orthocomplements
\begin{equation}
 \label{eqn:normal space to MGO}
N_m\left(M^G_{\mathcal{O}}\right) = N_m\left(M^G_{\Phi_G(m)}\right)\cap \mathfrak{g}_M(m)^{\perp_{\rho_m}}.
\end{equation}

We conclude from Lemma \ref{lem:normal to MGnu} and (\ref{eqn:tangent space to MGO}) that $N_m\left(M^G_{\mathcal{O}}\right)$
is the set of all vectors 
$J_m\big(\eta_M(m)\big)\in T_mM$ where 
$\eta\in \Phi_G(m)^{\perp_{\mathfrak{g}}}$ and $\rho_m\big(J_m\big(\eta_M(m)\big), \xi_M(m)\big)=0$
for every $\xi\in \mathfrak{g}$. From this remark we can draw the following conclusion.

\begin{step}
 \label{step:normal to MGO}
Let $\Upsilon=\Upsilon_{\mu,\boldsymbol{\nu}}$ be as in \S \ref{sctn:construction of Upsilon}.
Then for any $m\in M^G_{\mathcal{O}}$ we have 
$$
N_m\left(M^G_{\mathcal{O}}\right)= \mathrm{span}\big(\Upsilon(m)\big).
$$
In particular, $M^G_{\mathcal{O}}$ is orientable.
\end{step}

\begin{proof}
[Proof of Step \ref{step:normal to MGO}]
By the above,
\begin{eqnarray}
\label{eqn:normal space explicit}
\lefteqn{ N_m\left(M^G_{\mathcal{O}}\right)} \\
&=& \left\{ J_m\big(\eta_M(m)\big) : 
\eta\in \Phi_G(m)^{\perp_{\mathfrak{g}}}\,\wedge \, \rho_m\Big(J_m\big(\eta_M(m)\big), \xi_M(m)\Big) =0 \, \forall \xi\in \mathfrak{g}\right\}\nonumber \\
& = &
\left\{ J_m\big(\eta_M(m)\big) : 
\eta\in \Phi_G(m)^{\perp_{\mathfrak{g}}}\,\wedge \, \omega_m\big(\eta_M(m), \xi_M(m)\big) =0 \, \forall \xi\in \mathfrak{g}\right\}\nonumber\\
& = &
\left\{ J_m\big(\eta_M(m)\big) : 
\eta\in \Phi_G(m)^{\perp_{\mathfrak{g}}}\,\wedge \, \eta_M(m)\in \ker (\mathrm{d}_m\Phi_G)\right\}\nonumber\\
&=&\left\{ J_m\big(\eta_M(m)\big) : 
\eta\in \Phi_G(m)^{\perp_{\mathfrak{g}}}\,\wedge \, \big[\eta, \Phi_G(m)\big]=0\right\}.\nonumber
\end{eqnarray}
The latter equality holds because, by the equivariance of $\Phi_G$, we have
\begin{eqnarray*}
\mathrm{d}_m\Phi_G\big(\eta_M(m)\big)&=&\left.\frac{\mathrm{d}}{\mathrm{d}t} \Phi_G\left(\mu_{e^{t\eta}}(m)\right)\right|_{t=0}
=\left.\frac{\mathrm{d}}{\mathrm{d}t} \mathrm{Ad}_{e^{t\eta}}\Phi_G\left(m\right)\right|_{t=0}\nonumber\\
&=& \big[\eta, \Phi_G(m)\big].
\end{eqnarray*}

There exists a unique $h_m\,T\in G/T$ such that 
$\Phi_G(m)=\imath\,\lambda_{\boldsymbol{\nu}}(m)\,h_m\,D_{\boldsymbol{\nu}}\,h_m^{-1}$.
It is then clear that $\langle\Phi_G(m),\eta\rangle_{\mathfrak{g}}=0$ and $ \big[\eta, \Phi_G(m)\big]=0$ if and only if
$$\eta \in 
\mathrm{span}\left( \imath\,h_m\,D_{\boldsymbol{\nu}_\perp}\,h_m^{-1}\right)=\mathrm{span}\big(\boldsymbol{\rho}(m)\big),$$
where $\boldsymbol{\rho}(m)$ is as in (\ref{eqn:defn di rhonu}).
This completes the proof of Step \ref{step:normal to MGO}.

\end{proof}

\begin{step}
 \label{step:intersection of hypersurfaces}
$M^G_{\mathcal{O}}\cap M^T_{\boldsymbol{\nu}}=M^G_{\boldsymbol{\nu}}$.
\end{step}

\begin{proof}
[Proof of Step \ref{step:intersection of hypersurfaces}]
Obviously, $M^G_{\mathcal{O}}\cap M^T_{\boldsymbol{\nu}}\supseteq M^G_{\boldsymbol{\nu}}$.
Conversely, suppose $m\in M^G_{\mathcal{O}}\cap M^T_{\boldsymbol{\nu}}$.
Then on the one hand $\Phi_G(m)$ is similar to a positive multiple of $\imath\,D_{\boldsymbol{\nu}}$:
for a unique $h_m\,T\in G/T$,
\begin{equation}
\label{eqn:similar to Dnu}
 \Phi_G(m)=\imath\,\lambda_{\boldsymbol{\nu}}(m)\,h_m\,D_{\boldsymbol{\nu}}\,h_m^{-1}.
\end{equation}
We can assume without loss that $h_m\in SU(2)$.
On the other $\mathrm{diag}\big(\Phi_G(m)\big)$ is a positive multiple of $\imath\,\boldsymbol{\nu}$. Hence 
the diagonal of $h_m\,D_{\boldsymbol{\nu}}\,h_m^{-1}$ is a positive multiple of $\boldsymbol{\nu}$.
Let us write $h_m$ as in (\ref{eqn:ginSU(2)}), and argue as in the proof of Proposition \ref{prop:spectral characterization};
using that $\nu_1^2\neq \nu_2^2$, one concludes readily that
$h_m$ is diagonal. Hence $h_m\,D_{\boldsymbol{\nu}}\,h_m^{-1}=D_{\boldsymbol{\nu}}$, 
and so $\Phi_G(m)\in \mathbb{R}_+\cdot \imath\boldsymbol{\nu}$. Thus $m\in M^G_{\boldsymbol{\nu}}$. 
 
\end{proof}

\begin{step}
 \label{step:tangent intersection}
For any $m\in M^G_{\boldsymbol{\nu}}$,
$T_mM^G_{\mathcal{O}} = T_mM^T_{\boldsymbol{\nu}}$.
\end{step}

\begin{proof}[Proof of Step \ref{step:tangent intersection}]
If $m\in M^G_{\boldsymbol{\nu}}$, then $h_m=I_2$ in (\ref{eqn:similitude phinu}) and 
(\ref{eqn:defn di rhonu}); therefore,
$\Upsilon(m)=J_M\left((\imath\,\boldsymbol{\nu}_\perp) (m)\right)$.
Hence $N_m\left(M^G_{\mathcal{O}}\right)= \mathrm{span}\left(J_m\left((\imath\,\boldsymbol{\nu}_\perp) (m)\right)\right)$.
The claim follows from this and Lemma \ref{lem:normal to MGnu}.

\end{proof}

\begin{step}
 \label{step:boundary description}
$M^G_{\mathcal{O}}= \partial \left(G\cdot M^T_{\boldsymbol{\nu}}\right)$.
\end{step}

\begin{proof}
[Proof of Step \ref{step:boundary description}]
Suppose $m\in M^G_{\mathcal{O}}$. Thus 
$\Phi_G(m)=\imath\,\lambda_{\boldsymbol{\nu}}(m)\,h_m\,D_{\boldsymbol{\nu}}\,h_m^{-1}$for a unique $h_m\,T\in G/T$.
Let us choose $\delta>0$ arbitrarily small, and let $M(m,\delta)\subseteq M$ be the
open ball centered at $m$ and radius $\delta$ in the Riemannian distance on $M$. 
Since $\Phi_G$ is transverse to $\mathbb{R}_+\cdot \imath\,\boldsymbol{\nu}$, 
there exists $\epsilon_1>0$ such that the following holds. For every $\epsilon\in (-\epsilon_1,\epsilon_1)$
there exists $m'\in M(m,\delta)$ with 
\begin{equation}
 \label{eqn:PhiG m'}
\Phi_G(m')= \imath\, \lambda(m')\,h_m\,D_{\boldsymbol{\nu}+\epsilon\,\boldsymbol{\nu}_\perp}\,h_m^{-1}
\end{equation}
for some $\lambda(m')>0$ (see \S 2 of \cite{pao-jsg}). This implies that the eigenvalues of $-\imath\,\Phi_G(m')$ are 
$$
\lambda_1(m'):=\lambda (m')\,(\nu_1-\epsilon\,\nu_2),
\quad \lambda_2(m'):=\lambda (m')\,(\nu_2+\epsilon\,\nu_1).
$$
Therefore, the invariant defined in (\ref{eqn:spectral inequality}) takes the following value at $m'$:
\begin{eqnarray}
 \label{eqn:spectral inequality m'}
t(m',\boldsymbol{\nu})=-\frac{\epsilon}{\nu_1+\nu_2}\, \frac{\nu_1^2+\nu_2^2}{(\nu_1-\nu_2)-\epsilon\,(\nu_1+\nu_2)}.
\end{eqnarray}
Therefore, if $\epsilon\,(\nu_1+\nu_2)>0$ (and $\epsilon$ is sufficiently small) 
then $m'\not\in G\cdot M^T_{\boldsymbol{\nu}}$ 
by Proposition \ref{prop:spectral characterization}.
This implies $M^G_{\mathcal{O}}\subseteq  \partial \left(G\cdot M^T_{\boldsymbol{\nu}}\right)$.

To prove the reverse inclusion, assume that $m\in G\cdot M^T_{\boldsymbol{\nu}}\setminus M^G_{\mathcal{O}}$.
Then $t(m,\boldsymbol{\nu})\in [0,1/2)$ by Proposition \ref{prop:spectral characterization}.
Furthermore, $t(m,\boldsymbol{\nu})\neq 0$, for otherwise
$m\in M^G_{\mathcal{O}}$.  Hence
$t(m,\boldsymbol{\nu})\in (0,1/2)$; by continuity, then, 
$t(m',\boldsymbol{\nu})\in (0,1/2)$ for every $m'$ in a sufficiently small open neighborhood of $m$.
Hence Proposition \ref{prop:spectral characterization} implies that 
$G\cdot M^T_{\boldsymbol{\nu}}\setminus M^G_{\mathcal{O}}$
contains an open neighborhood of $m$ in $M$. Thus $G\cdot M^T_{\boldsymbol{\nu}}\setminus M^G_{\mathcal{O}}$
is open, and in particular $m\not\in \partial \left(G\cdot M^T_{\boldsymbol{\nu}}\right)$.
\end{proof}

\begin{step}
 \label{step:Upsilon outer or inner}
$\Upsilon$ is outer oriented if $\nu_1+\nu_2>0$ and inner oriented if $\nu_1+\nu_2<0$.
\end{step}

\begin{proof}
[Proof of Step \ref{step:Upsilon outer or inner}]
Let denote by $\mathcal{B}_{\boldsymbol{\nu}}$ the collection of all $B\in \mathfrak{g}$ 
such that $\mathrm{diag}\left(g\,B\,g^{-1}\right)\in \mathbb{R}_+\,\imath\,\boldsymbol{\nu}$
for some $g\in G$.
Thus $\mathcal{B}_{\boldsymbol{\nu}}$ is a conic and invariant closed subset of
$\mathfrak{g}\setminus \{0\}$; in addition,
$m\in G\cdot M^T_{\boldsymbol{\nu}}$ if and only if $\Phi_G(m)\in \mathcal{B}_{\boldsymbol{\nu}}$.

If $\lambda_1(B)\ge \lambda_2(B)$ are the eigenvalues of $-\imath\,B$, then 
Proposition \ref{prop:spectral characterization} implies that $B\in \mathcal{B}_{\boldsymbol{\nu}}$
if and only if $\lambda_1(B)> \lambda_2(B)$ and
$$
t(B,\boldsymbol{\nu}):=
\frac{\lambda_1(B) \,\nu_2- \lambda_2(B)\,\nu_1}{(\nu_1+\nu_2) \,\big(\lambda_1(B) - \lambda_2(B)\big)} \in [0,1/2).
$$
In particular, if $t(B,\boldsymbol{\nu})\in (0,1/2)$ then $B$ belongs to the interior of 
$\mathcal{B}_{\boldsymbol{\nu}}$.

Suppose $m\in M^G_{\boldsymbol{\nu}}$ and consider the path 
$$
\gamma_1:\tau\in (-\epsilon,\epsilon)\mapsto \Phi_G\big(m+\tau\,\Upsilon(m)\big)\in \mathfrak{g},
$$
defined for sufficiently small $\epsilon>0$; 
the expression $m+\tau\,\Upsilon(m)\in M$
is meant in an adapted coordinate system on $M$ centered at $m$. Then 
\begin{eqnarray}
 \label{eqn:initial conds gamma1}
\gamma_1(0)&=&\Phi_G(m)=\imath\,\lambda_{\boldsymbol{\nu}}(m) \, D_{\boldsymbol{\nu}},
\\ 
\dot{\gamma}_1(0)&=& \omega_m \big(\cdot, \Upsilon (m)\big)=\rho_m\big(\cdot , (\imath\,\boldsymbol{\nu}_\perp)_M(m)\big).
\end{eqnarray}

Let us consider a smooth positive function, 
$y:(-\epsilon,\epsilon) \rightarrow \mathbb{R}_+$, to be determined but 
subject to the condition $y(0)=\lambda_{\boldsymbol{\nu}}(m)$.
Let us define a second smooth path of the form
\begin{equation}
 \label{eqn:second path}
\gamma_2(\tau):=
\imath\,y(\tau)\,\mathrm{Ad}_{e^{\tau\, \boldsymbol{\xi}}}\left (D_{\boldsymbol{\nu}+a\,\tau\,\boldsymbol{\nu}_\perp}\right),
\end{equation}
where $a>0$ is a constant also to be determined.

Then 

\begin{eqnarray}
\gamma_1(0)&=&\gamma_2(0) \nonumber\\
 \dot{\gamma}_2(0)&=&\imath\, \left [ \dot{y}(0)\, D_{\boldsymbol{\nu}}+
\lambda_{\boldsymbol{\nu}}(m)\, [\boldsymbol{\xi},\boldsymbol{\nu}] + a\,\lambda_{\boldsymbol{\nu}}(m)\,
D_{\boldsymbol{\nu}_\perp}\right].
\end{eqnarray}
Clearly, we can choose $a>0$ uniquely so that 
\begin{eqnarray}
 \label{eqn:gamma su hatnu}
 a\,\lambda_{\boldsymbol{\nu}}(m)\,\|\boldsymbol{\nu}\|^2=
\rho_m\big((\imath\,\boldsymbol{\nu}_\perp)_M(m) , (\imath\,\boldsymbol{\nu}_\perp)_M(m)\big),
\end{eqnarray}
so that 
$\left\langle \dot{\gamma}_2(0),\boldsymbol{\nu}_\perp\right\rangle 
=
\left\langle \dot{\gamma}_1(0),\boldsymbol{\nu}_\perp\right\rangle$.
Having fixed $a$, we can then choose $\dot{y}(0)$ uniquely so that
\begin{equation}
 \label{eqn:gamma su nu}
\dot{y}(0)\, \|\boldsymbol{\nu}\|^2=\rho_m
\big((\imath\,\boldsymbol{\nu})_M(m), (\imath\,\boldsymbol{\nu}_\perp)_M(m)\big),
\end{equation}
so that we also have
$\left\langle \dot{\gamma}_2(0),\boldsymbol{\nu}\right\rangle =
\left\langle \dot{\gamma}_1(0),\boldsymbol{\nu}\right\rangle$.
Finally, if we set 
$$
\boldsymbol{\upsilon}_1:=
\begin{pmatrix}
 0&1\\
-1&0
\end{pmatrix},
\quad
\boldsymbol{\upsilon}_2:=
\begin{pmatrix}
 0&\imath\\
\imath&0
\end{pmatrix}
$$
we can choose $\boldsymbol{\xi}\in \mathrm{span}_{\mathbb{R}}\big\{\boldsymbol{\upsilon}_1,\,\boldsymbol{\upsilon}_2\big\}$
uniquely so that 
\begin{equation}
 \label{eqn:gamma su upsilonj}
\lambda_{\boldsymbol{\nu}}(m)\,\langle [\boldsymbol{\xi},\boldsymbol{\nu}], \boldsymbol{\upsilon}_j\rangle
=\rho_m\big(\boldsymbol{\upsilon}_{jM}(m), 
(\imath\,\boldsymbol{\nu}_\perp)_M(m)\big),
\end{equation}
so that in addition
$\left\langle \dot{\gamma}_2(0),\boldsymbol{\upsilon}_j\right\rangle =
\left\langle \dot{\gamma}_1(0),\boldsymbol{\upsilon}_j\right\rangle$
for $j=1,2$.
With these choices, 
$\gamma_1$ and $\gamma_2$ agree to first order at $0$.

Let us remark that when $\tau$ is sufficiently small $\gamma_2(\tau)$ has eigenvalues 
$$
\lambda_1\big(\gamma_2(\tau)\big)= y(\tau)\, (\nu_1-a\,\tau\,\nu_2)>
\lambda_2\big(\gamma_2(\tau)\big)= y(\tau)\, (\nu_2+a\,\tau\,\nu_1).
$$
Hence
\begin{equation}
 \label{eqn:tBnu}
t(B,\boldsymbol{\nu})=-\frac{a\,\tau}{\nu_1+\nu_2}\,\frac{\nu_1^2+\nu_2^2}{\nu_1-\nu_2+a\tau\,(\nu_1+\nu_2)}.
\end{equation}

Thus if $\nu_1+\nu_2>0$ then $\gamma_2(\tau)\not\in \mathcal{B}_{\boldsymbol{\nu}}$
when $\tau\in (0,\epsilon)$; since $\gamma_1$ and $\gamma_2$ agree to second order at $0$,
we also have $\Phi_G\big(m+\tau\,\Upsilon(m)\big)\not\in \mathcal{B}_{\boldsymbol{\nu}}$
when $\tau\sim 0^+$. Hence $\Upsilon$ is outer oriented at $m$, and thus everywhere on
$M^G_{\mathcal{O}}$.

The argument when $\nu_1+\nu_2<0$ is similar.
\end{proof}

The proof of Theorem \ref{thm:geometry locus} is complete.
\end{proof}

\section{Proof of Theorem \ref{thm:rapid decrease slow}}

\subsection{Preliminaries}

\subsubsection{Recalls on Szeg\"{o} kernels}


Let $\Pi$, $\Pi(\cdot,\cdot)$ and 
$\Pi_{\boldsymbol{\nu}}$, $\Pi_{\boldsymbol{\nu}}(\cdot,\cdot)$ be as in (\ref{eqn:Pi and kernel})
and (\ref{eqn:equivariant szego kernel component}). For any $x,y\in X$, we have
\begin{equation}
 \label{eqn:equivariant szego kernel explicit}
 \Pi_{\boldsymbol{\nu}}(x,y) =
 d_{\boldsymbol{\nu}}\,\int_G \overline{\chi_{\boldsymbol{\nu}}(g)}\,\Pi\left(\widetilde{\mu}_{g^{-1}}(x),y\right)\,
 \mathrm{d}V_G(g).
\end{equation}

In view of (\ref{explicit character}) and the Weyl integration formula (\ref{sctn:weyl}), 
(\ref{eqn:equivariant szego kernel explicit}) can be rewritten
\begin{eqnarray}
 \label{eqn:equivariant szego kernel weyl}
\Pi_{\boldsymbol{\nu}}(x,y)
&=& d_{\boldsymbol{\nu}}\,\int_T t^{-\boldsymbol{\nu}}\,\Delta(t)\,F(t;x,y)\,\mathrm{d}V_T(t),
\end{eqnarray}
where $t^{-\boldsymbol{\nu}}=t_1^{-\nu_1}\,t^{-\nu_2}$, and 
\begin{equation}
 \label{eqn:defn of Fxy}
F(t;x,y):=\int_{G/T}\Pi\left(\widetilde{\mu}_{gt^{-1}g^{-1}}(x),y\right)\,\mathrm{d}V_{G/T}(g\,T).
\end{equation}

We have already used the structure
of the wave front of $\Pi$ 
in the proof of Theorem \ref{thm:rapid decrease fixed} (see (\ref{eqn:wave front of Pi})).
In the proof of Theorem \ref{thm:rapid decrease slow}, we need to exploit the explicit description
of $\Pi$ as an FIO 
developed in \cite{boutet-sjostraend} (see also the discussions in \cite{zelditch-theorem-of-Tian}, \cite{bsz}, \cite{sz}). 

Namely, up to
a smoothing contribution, we have
\begin{equation}
 \label{eqn:PiasFIO}
 \Pi(x,y) \sim \int_0^{+\infty}
 e^{\imath\,u\,\psi (x,y)}\,s(x,y,u)\,\mathrm{d}u,
\end{equation}
where $\psi$ is essentially determined by the Taylor expansion of the metric along the diagonal,
and $s$ is a semiclassical symbol admitting an asymptotic expansion 
$s(x,y,u)\sim \sum_{j\ge 0}u^{d-j}\,s_j(x,y)$. The differential of $\psi$ along the diagonal is
\begin{equation}
 \label{eqn:differential of psi diagonal}
 \mathrm{d}_{\left(x,x\right)}\psi=
 (\alpha_x,-\alpha_{x})\quad (x\in X).
\end{equation}

\subsubsection{An \textit{a priori} polynomial bound}
Let us record the following rough \textit{a priori} polynomial bound.

\begin{lem}
\label{lem:a priori bound}
There is a constant $C_{\boldsymbol{\nu}}>0$ such that for any $x\in X$ one has
$$
|\Pi_{k\boldsymbol{\nu}}(x,x)|\le C_{\boldsymbol{\nu}}\,k^{d+1}
$$
for $k\gg 0$.
\end{lem}
					      
\begin{proof}
Let $r:S^1\times X\rightarrow X$ be the standard structure action on the unit circle bundle $X$.
As in \ref{scnt:esempio1}, let 
$$
H(X)=\bigoplus_{l=0}^{+\infty}H(X)_l
$$
be the decomposition of $H(X)$ as a direct sum of isotypes for the $S^1$-action.

Since $\widetilde{\mu}$ commutes with the structure action of $S^1$ on $X$, we have 
$$
H(X)_{k\boldsymbol{\nu}}=\bigoplus_{l=0}^{+\infty}H(X)_{k\boldsymbol{\nu}}\cap H(X)_l.
$$
On the other hand, by the theory of \cite{guillemin-sternberg gq} we have
$H(X)_{k\boldsymbol{\nu}}\cap H(X)_l\neq (0)$ only if the highest weight vector $\mathbf{r}(k\boldsymbol{\nu})$ of the representation
indexed by $k\,\boldsymbol{\nu}$ satisfies
\begin{equation}
 \label{eqn:highest weight moment map}
\mathbf{r}(k\boldsymbol{\nu})=(k\,\nu_1-1,k\nu_2)=
k\,\boldsymbol{\nu}+(-1,0)\in l\,\Phi_G(M)\subseteq \mathfrak{g}.
\end{equation}
Let us define
$$
a_G:=\min\|\Phi_G\|,\,\,\,\,A_G:=\max\|\Phi_G\|.
$$
Thus $A_G\ge a_G>0$. 
Therefore, 
we need to have
\begin{equation}
\label{eqn:lower estimate on l}
l\, a_G\le \|\mathbf{r}(k\boldsymbol{\nu})\|\le k\,\|\boldsymbol{\nu}\|+1\,\Rightarrow\,
l\le L_1(k):=\left\lceil\frac{\|\boldsymbol{\nu}\|}{a_G}\,k+\frac{1}{a_G}\right\rceil.
\end{equation}
Similarly,
\begin{equation}
\label{eqn:upper estimate on l}
k\,\|\boldsymbol{\nu}\|-1\le \|\mathbf{r}(k\boldsymbol{\nu})\|\le l\,A_G\,\Rightarrow\,
L_2(k):=\left\lfloor\frac{\|\boldsymbol{\nu}\|}{A_G}\,k-\frac{1}{A_G}\right\rfloor\le l.
\end{equation}
On the other hand, in view of the asymptotic expansion of $\Pi_k(x,x)$ from \cite{catlin}, \cite{tian}, \cite{zelditch-theorem-of-Tian}
we also have
$
\Pi_l(x,x)\le 2\,\left(l/\pi\right)^d$ for $l\gg 0$. 
We conclude that
\begin{eqnarray}
\label{eqn:first etimate}
\Pi_{k\boldsymbol{\nu}}(x,x)&\le&\sum_{l=L_1(k)}
^{L_2(k)}\Pi_l(x,x)\le \frac{2}{\pi^d}\,\sum_{l=L_1(k)}
^{L_2(k)}l^d\le C_{\boldsymbol{\nu}}\,k^{d+1}
\end{eqnarray}
for some constant $C_{\boldsymbol{\nu}}>0$.
\end{proof}

\subsection{The proof}

We shall use the following notational short-hand.
If $x\in X$, $g\in G$, $t\in T$, let us set
$$
x(g,t):=\widetilde{\mu}_{g\,t^{-1}\,g^{-1}}(x);
$$
similarly, if $m\in M$
$$
m(g,t):=\mu_{g\,t^{-1}\,g^{-1}}(m).$$ 
If $t=e^{i\boldsymbol{\vartheta}}:=\left(e^{i\vartheta_1},e^{i\vartheta_2}\right)$, we shall write 
$x(g,t)=x(g,\boldsymbol{\vartheta})$, $m(g,t)=m(g,\boldsymbol{\vartheta})$. 
Since $\widetilde{\mu}$ is a lifting
of $\mu$, if $m=\pi(x)$ then 
$$
m(g,\boldsymbol{\vartheta})= \pi\big(x(g,\boldsymbol{\vartheta})\big).
$$

\begin{proof}[Proof of Theorem \ref{thm:rapid decrease slow}]
If we replace $\boldsymbol{\nu}$ by $k\,\boldsymbol{\nu}$ in (\ref{eqn:equivariant szego kernel weyl}),
and use angular coordinates on $T$, we obtain
\begin{eqnarray}
  \label{eqn:equivariant szego kernel weyl knu}
  \lefteqn{
\Pi_{k\boldsymbol{\nu}}(x,y)}\\
 &=&\frac{k\,(\nu_1-\nu_2)}{(2\pi)^2}\,
  \int_{-\pi}^\pi\,\int_{-\pi}^\pi\,e^{-ik\langle\boldsymbol{\nu},\boldsymbol{\vartheta}\rangle}\,\Delta\left(e^{i\boldsymbol{\vartheta}}\right)\,
F\left(e^{i\boldsymbol{\vartheta}};x,y\right)\,\mathrm{d}\boldsymbol{\vartheta};\nonumber
 \end{eqnarray}
here $e^{i\boldsymbol{\vartheta}}=\left(e^{\imath\,\vartheta_1},\,e^{\imath\,\vartheta_2}\right)$.

For $\delta>0$, let us define 
\begin{equation}
 \label{eqn:positive distance}
V_{\delta}:=\left\{(x,y)\in X\,:\,\mathrm{dist}_X\big(x,G\cdot y)\ge \delta\right\}.
\end{equation}

\begin{prop}
 \label{prop:rapid decrease delta}
For any $\delta>0$, we have $\Pi_{k\boldsymbol{\nu}}(x,y)=O\left(k^{-\infty}\right)$ uniformly on $V_\delta$.
\end{prop}

\begin{proof}
 [Proof of Proposition \ref{prop:rapid decrease delta}]
By (\ref{eqn:wave front of Pi}), the singular support of $\Pi$ is the diagonal in $X\times X$.
Therefore,  
\begin{equation}
 \label{eqn:non-singular}
\beta:\big((x,y),\,gT,\,t\big)\in V_\delta\times G/T\times T\mapsto \Pi\,\big(x(g,t),y\big)\in \mathbb{C}
\end{equation}
is $\mathcal{C}^\infty$.
The same then holds of $\big((x,y),t\big)\in V_{\delta}\times T\mapsto \Delta(t)\,F(t;x,y)$.
Hence its Fourier transform (\ref{eqn:equivariant szego kernel weyl knu}) is rapidly decreasing for 
$k\rightarrow+\infty$.
\end{proof}

We are thus reduced to assuming that 
$\mathrm{dist}_X\big(x,G\cdot y)< \delta$ for some fixed and arbitrarily small 
$\delta>0$. Let $\varrho\in \mathcal{C}_0^\infty(\mathbb{R})$ be $\equiv 1$ on $[-1,1]$ and 
$\equiv 0$ on $\mathbb{R}\setminus (-2,2)$. We can write 
$$
\Pi_{\boldsymbol{\nu}}(x,y)=\Pi_{\boldsymbol{\nu}}(x,y)_1+\Pi_{\boldsymbol{\nu}}(x,y)_2,
$$
where the two summands on the right are defined by setting
\begin{eqnarray}
 \label{eqn:equivariant szego kernel weyl'}
\Pi_{\boldsymbol{\nu}}(x,y)_j
&:=& d_{\boldsymbol{\nu}}\,\int_T t^{-\boldsymbol{\nu}}\,\Delta(t)\,F(t;x,y)_j\,\mathrm{d}V_T(t),
\end{eqnarray}
and $F(t;x,y)_1$ is defined as in (\ref{eqn:defn of Fxy}), but with the integrand multiplied by
$\varrho \left(\delta^{-1}\,\mathrm{dist}_X\big(x(g,\boldsymbol{\vartheta}),y\big)\right)$;
similarly, $F(t;x,y)_2$  is defined as in (\ref{eqn:defn of Fxy}), but with the integrand multiplied by
$1-\varrho \left(\delta^{-1}\,\mathrm{dist}_X\big(x(g,\boldsymbol{\vartheta}),y\big)\right)$.

\begin{lem}
\label{lem:only first summand}
 $\Pi_{k\boldsymbol{\nu}}(x,y)_2=O\left(k^{-\infty}\right)$ for $k\rightarrow +\infty$.
\end{lem}

\begin{proof}
 [Proof of Lemma \ref{lem:only first summand}]
On the support of the integrand in $\Pi_{k\boldsymbol{\nu}}(x,y)_2$, we have 
$\mathrm{dist}_X\big(x(g,t),y\big)\ge \delta$.
We can then apply with minor changes the argument in the proof of Proposition \ref{prop:rapid decrease delta}.
\end{proof}

On the support of the integrand in $\Pi_{k\boldsymbol{\nu}}(x,y)_1$, 
$\mathrm{dist}_X\big(x(g,t),y\big)\le 2\,\delta$;
therefore, perhaps after discarding a smoothing term contributing negligibly to the asymptotics,
we can apply (\ref{eqn:PiasFIO}).
With some passages, we obtain in place of (\ref{eqn:equivariant szego kernel weyl knu}):
\begin{eqnarray}
 \label{eqn:integrale per Pi'}
\lefteqn{\Pi_{k\boldsymbol{\nu}}(x,y)\sim \Pi_{k\boldsymbol{\nu}}(x,y)_1}\\
&\sim& \frac{k^2\,(\nu_1-\nu_2)}{(2\pi)^2}\,
  \int_{-\pi}^\pi\,\int_{-\pi}^\pi\,\int _{G/T}\,\int_0^{+\infty}
  e^{\imath \,k\,\Psi_{x,y}}\,\mathcal{A}_{x,y}\,\mathrm{d}u\,
  \mathrm{d}V_{G/T}(gT)\,\mathrm{d}\boldsymbol{\vartheta};\nonumber
\end{eqnarray}
we have applied the rescaling $u\mapsto k\,u$ to the parameter in
(\ref{eqn:PiasFIO}), and set
\begin{equation}
 \label{eqn:phase Psi}
 \Psi_{x,y}=\Psi_{x,y}(u,\boldsymbol{\vartheta},gT):=
 u\,\psi\left(\widetilde{\mu}_{g\,e^{-\imath\,\boldsymbol{\vartheta}}\,g^{-1}}(x),y\right)
 -\langle\boldsymbol{\nu},\boldsymbol{\vartheta}\rangle,
\end{equation}
\begin{equation}
 \label{eqn:amplitude A}
 \mathcal{A}_{x,y}=\mathcal{A}_{x,y}(u,\boldsymbol{\vartheta},gT):=
 \Delta\left(e^{i\boldsymbol{\vartheta}}\right)\,s'\left(\widetilde{\mu}_{g\,e^{-\imath\,\boldsymbol{\vartheta}}\,g^{-1}}(x),y,k\,u\right),
\end{equation}
with
\begin{eqnarray}\label{eqn:defn di s'}
 s'\left(\widetilde{\mu}_{g\,e^{-\imath\,\boldsymbol{\vartheta}}\,g^{-1}}(x),y,k\,u\right)
&:=&s\left(\widetilde{\mu}_{g\,e^{-\imath\,\boldsymbol{\vartheta}}\,g^{-1}}(x),y,k\,u\right)
\nonumber\\
&&\cdot\varrho \left(\delta^{-1}\,\mathrm{dist}_X\left(\widetilde{\mu}_{g\,t^{-1}\,g^{-1}}(x),y\right)\right).
\end{eqnarray}

\begin{lem}
 \label{lem:reduction in u}
 Only a rapidly decreasing contribution to the asymptotics is lost, if in (\ref{eqn:integrale per Pi'}) integration
 in $\mathrm{d}u$ is restricted to an interval of the form $(1/D,D)$ for some $D\gg 0$.
\end{lem}

\begin{proof}[Proof of Lemma \ref{lem:reduction in u}]
Suppose that $x,y\in X$, $\left(g_0\,T,e^{\imath\boldsymbol{\vartheta}_0}\right)\in (G/T)\times T$ and
\begin{equation}
 \label{eqn:limite su distanza}
 \mathrm{dist}_X\left(x(g_0,\boldsymbol{\vartheta}_0),y\right)< \delta.
\end{equation}
In view of (\ref{eqn:differential of psi diagonal}), in any system of local
coordinates we have
\begin{equation}
 \label{eqn:differential of psi}
 \mathrm{d}_{\left(x(g_0,\boldsymbol{\vartheta}_0),y\right)}\psi=
 (\alpha_{x(g_0,\boldsymbol{\vartheta}_0)},-\alpha_{y})+O(\delta).
\end{equation}

Let $\mathrm{d}^{(\boldsymbol{\vartheta})}$ denote the differential with respect to the
variable $\boldsymbol{\vartheta}$. If $\imath\,\boldsymbol{\eta}\in \mathfrak{t}$, we obtain
with $m_x:=\pi(x)$:
\begin{eqnarray}
 \label{eqn:directional derivative x}
 \lefteqn{\left.\frac{d}{\mathrm{d}\tau}\,
 x(g_0,\boldsymbol{\vartheta}_0+\tau\,\boldsymbol{\eta})\right|_{\tau=0}}\nonumber\\
 &=&-\mathrm{Ad}_{g_0}(\imath\,\boldsymbol{\eta})_X\big(x(g_0,\boldsymbol{\vartheta}_0)\big)\nonumber\\
 &=&-\mathrm{Ad}_{g_0}(\imath\,\boldsymbol{\eta})_M\big(m_x(g_0,\boldsymbol{\vartheta}_0)\big)^\sharp+
 \Big\langle\Phi_G\big(m_x(g_0,\boldsymbol{\vartheta}_0)\big),\mathrm{Ad}_{g_0}(\imath\,\boldsymbol{\eta})\Big\rangle\,
 \partial_\theta.
\end{eqnarray}

On the other hand, as $\Phi_G$ is $G$-equivariant we get
\begin{eqnarray}
 \label{eqn:PhiGangular part}
 \lefteqn{\langle\Phi_G\big(m_x(g_0,\boldsymbol{\vartheta}_0)\big),\mathrm{Ad}_{g_0}(\imath\,\boldsymbol{\eta})\rangle=
 \left\langle\mathrm{Ad}_{g_0^{-1}}\Big(\Phi_G\big(m_x(g_0,\boldsymbol{\vartheta}_0)\big)\Big),\imath\,\boldsymbol{\eta}\right\rangle}
 \\
 &=&\left\langle\Phi_G\Big(\mu_{g_0^{-1}}\big(m_x(g_0,\boldsymbol{\vartheta}_0)\big)\Big),\imath\,\boldsymbol{\eta}\right\rangle
 =\left\langle\Phi_T\Big(\mu_{g_0^{-1}}\big(m_x(g_0,\boldsymbol{\vartheta}_0)\big)\Big),\imath\,\boldsymbol{\eta}\right\rangle.
 \nonumber
\end{eqnarray}

Now, (\ref{eqn:differential of psi}), (\ref{eqn:directional derivative x}) and (\ref{eqn:PhiGangular part}) imply
\begin{eqnarray}
 \label{eqn:directional derivative psi}
\lefteqn{ \left.\frac{d}{\mathrm{d}\tau}\,\psi\Big(x(g_0,\boldsymbol{\vartheta}_0+\tau\,\boldsymbol{\eta}),y\Big)
 \right|_{\tau=0}}\\
 &=&-\mathrm{d}_{\left(x(g_0,\boldsymbol{\vartheta}_0),y\right)}\psi
 \Big (\mathrm{Ad}_{g_0}(\imath\,\boldsymbol{\eta})_X\big(x(g_0,\boldsymbol{\vartheta}_0)\big),0\Big)\nonumber\\
 &=&-\alpha_{x(g_0,\boldsymbol{\vartheta}_0)}
 \Big (\mathrm{Ad}_{g_0}(\imath\,\boldsymbol{\eta})_X\big(x(g_0,\boldsymbol{\vartheta}_0)\big)\Big)+
 \langle O(\delta),\boldsymbol{\eta}\rangle\nonumber\\
 &=&\left\langle\frac{1}{\imath}\,\Phi_T\Big(\mu_{g_0^{-1}}\big(m_x(g_0,\boldsymbol{\vartheta}_0)\big)\Big) +
 O(\delta),\boldsymbol{\eta}\right\rangle.\nonumber
\end{eqnarray}

Let $\mathrm{d}^{(\boldsymbol{\vartheta})}$ denote the differential with respect to  
$\boldsymbol{\vartheta}$.
Recalling (\ref{eqn:phase Psi}), we obtain
\begin{eqnarray}
\label{eqn:vartheta derivative}
 \mathrm{d}^{(\boldsymbol{\vartheta})}_{(u,g_0T,\boldsymbol{\vartheta}_0)}\Psi_{x,y}
 &=&\frac{u}{\imath}\,\Phi_T\Big(\mu_{g_0^{-1}}(m_x)\Big)-\boldsymbol{\nu}+
 O(\delta).
\end{eqnarray}

By assumption, $\mathbf{0}\not\in \Phi_T(M)$. Let us set
$$
a_T:=\min \|\Phi_T\|,\,\,\,\,A_T:=\max \|\Phi_T\|.
$$
Then $A_T\ge a_T>0$, and (\ref{eqn:vartheta derivative}) implies
\begin{eqnarray}
 \label{eqn:vartheta derivative bound}
 \lefteqn{\left\|\mathrm{d}^{(\boldsymbol{\vartheta})}_{(u,g_0T,\boldsymbol{\vartheta}_0)}\Psi_{x,y}\right\|}
\nonumber \\
 &\ge&
\max\big\{ u\,a_T-\|\boldsymbol{\nu}\|+O(\delta),\,\|\boldsymbol{\nu}\|-u\,A_T+O(\delta)\big\}.
\end{eqnarray}
Thus if $D\gg 0$ and $u\ge D$ we have
\begin{equation}
 \label{eqn:bound large}
 \left\|\mathrm{d}^{(\boldsymbol{\vartheta})}_{(u,g_0T,\boldsymbol{\vartheta}_0)}\Psi_{x,y}\right\|\ge
 \frac{a_T}{2}\,u+1,
\end{equation}
while for $0<u<1/D$ 
\begin{equation}
 \label{eqn:bound small}
 \left\|\mathrm{d}^{(\boldsymbol{\vartheta})}_{(u,g_0T,\boldsymbol{\vartheta}_0)}\Psi_{x,y}\right\|\ge
 \frac{\|\boldsymbol{\nu}\|}{2}.
\end{equation}

The Lemma then follows from (\ref{eqn:bound large}) and (\ref{eqn:bound small}) by a standard 
iterated integration by parts
in $\boldsymbol{\vartheta}$ (in view of the compactness of $T$).
\end{proof}

Suppose that $\rho\in \mathcal{C}^\infty_0\big((0,+\infty)\big)$ is $\equiv 1$ on $(1/D,D)$ and is supported on
$(1/(2D),2D)$.
By Lemma \ref{lem:reduction in u}, the asymptotics of (\ref{eqn:integrale per Pi'}) are unaltered, if the integrand
is multiplied by $\rho(u)$.
Thus we obtain
\begin{eqnarray}
 \label{eqn:integrale per Pi'1}
\lefteqn{\Pi_{k\boldsymbol{\nu}}(x,y)}\\
&\sim& \frac{k^2\,(\nu_1-\nu_2)}{(2\pi)^2}\,
  \int_{-\pi}^\pi\,\int_{-\pi}^\pi\,\int _{G/T}\,\int_{1/(2D)}^{2D}
  e^{\imath \,k\,\Psi_{x,y}}\,\mathcal{A}'_{x,y}\,\mathrm{d}u\,
  \mathrm{d}V_{G/T}(gT)\,\mathrm{d}\boldsymbol{\vartheta};\nonumber
\end{eqnarray}
with $\mathcal{A}_{x,y}$ as in (\ref{eqn:amplitude A}), we have set
\begin{equation}
 \label{eqn:amplitude A'}
 \mathcal{A}'_{x,y}(u,\boldsymbol{\vartheta},gT):=\rho(u)\,\mathcal{A}_{x,y}(u,\boldsymbol{\vartheta},gT).
\end{equation}
Integration in $\mathrm{d}u$ is now over a compact interval

Let $\Im (z)$ denote the imaginary part of $z\in \mathbb{C}$. In view of Corollary 1.3 of \cite{boutet-sjostraend}, there exists a
fixed constant $D$, depending only on $X$,
such that 
\begin{equation}
 \label{eqn:bound on Im psi}
\Im \Big(\psi\left(x',x''\right)\Big)
\ge D\,\mathrm{dist}_X\left(x',x''\right)^2\quad (x',x''\in X).
\end{equation}

\begin{prop}
 \label{prop:rapid decay orbit case}
Uniformly for
\begin{equation}
 \label{eqn:orbital distance lower bound}
\mathrm{dist}_X(x,G\cdot y)\ge C\,k^{\epsilon-1/2},
\end{equation}
we have
$\Pi_{k\boldsymbol{\nu}}(x,y)=O\left(k^{-\infty}\right)$.
\end{prop}

\begin{proof}
[Proof of Proposition \ref{prop:rapid decay orbit case}]
In the range (\ref{eqn:orbital distance lower bound}), 
we have
\begin{equation}
 \label{eqn:bound on distance xy k}
\mathrm{dist}_X\big(x(g,\boldsymbol{\vartheta}),y\big)\ge C\,k^{\epsilon-1/2}
\end{equation}
for every $g\,T\in G/T$ and $e^{\imath\,\boldsymbol{\vartheta}}\in T$.
In view of (\ref{eqn:phase Psi}) and (\ref{eqn:bound on Im psi}),
\begin{eqnarray}
\label{eqn:partial_u bound}
 \left|\partial_u\Psi_{x,y}(u,\boldsymbol{\vartheta},gT)\right|&=&
\left|\psi\left(x(g,\boldsymbol{\vartheta}),y\right)\right|
\ge\Im\left(\psi\big(x(g,\boldsymbol{\vartheta}),y\big)\right)
\nonumber\\
&\ge&
D\,\mathrm{dist}_X\big(x(g,\boldsymbol{\vartheta}),y\big)^2\ge D\,C^2\,k^{2\epsilon-1}.
\end{eqnarray}
Let us use the identity
\begin{equation}
 \label{eqn:key identity du}
-\frac{\imath}{k}\,\psi\big(x(g,\boldsymbol{\vartheta}),y\big)^{-1}\,
\frac{\mathrm{d}}{\mathrm{d}u} e^{\imath \,k\,\Psi_{x,y}}=e^{\imath \,k\,\Psi_{x,y}}
\end{equation}
to iteratively integrate by parts in $\mathrm{d}u$ in (\ref{eqn:integrale per Pi'1}); then by (\ref{eqn:partial_u bound}) at each step we introduce 
a factor $O\left(k^{-2\,\epsilon}\right)$.
The claim follows.
\end{proof}

To complete the proof of Theorem \ref{thm:rapid decrease slow}, we need to establish the following.

\begin{prop}
 \label{prop:diagonal case distance MGO}
 Uniformly for 
\begin{equation}
 \label{eqn:distance to GMTnu}
\mathrm{dist}_X\left(x,G\cdot X^T_{\boldsymbol{\nu}}\right)\ge C\,k^{\epsilon-1/2},
\end{equation} 
we have 
$\Pi_{k\boldsymbol{\nu}}(x,x)=O\left(k^{-\infty}\right)
$
as $k\rightarrow +\infty$.
\end{prop}

\begin{rem}
Let $\mathrm{dist}_M$ denote the distance function on $M$; 
if $m=\pi(x)$, then
$\mathrm{dist}_X\left(x,G\cdot X^T_{\boldsymbol{\nu}}\right)=\mathrm{dist}_M\left(m,G\cdot M^T_{\boldsymbol{\nu}}\right)$.
\end{rem}

\begin{proof}
 [Proof of Proposition \ref{prop:diagonal case distance MGO}]
Since $G$ acts on $M$ as a group of Riemannian isometries, 
(\ref{eqn:distance to GMTnu}) means that for any $g\in G$ we have
\begin{equation}
 \label{eqn:distance to gMTnu}
C\,k^{\epsilon-1/2}\le \mathrm{dist}_M\left(m,\mu_{g}\left(M^T_{\boldsymbol{\nu}}\right)\right)=
\mathrm{dist}_M\left(\mu_{g^{-1}}\left(m\right),M^T_{\boldsymbol{\nu}}\right).
\end{equation}

On the other hand, as
$-\imath\,\Phi^T$ is transverse to $\mathbb{R}_+\,\boldsymbol{\nu}$, by the discussion in \S 2.1.3 of \cite{pao-jsg}
there is a constant $b_{\boldsymbol{\nu}}>0$ such that every $u\in [1/(2D),2D]$ we have
\begin{equation}
 \label{eqn:distance_regular_fcnt g u}
 \left\|-\imath\,u\,\Phi^T\left(\mu_{g^{-1}}\left(m\right)\right)-\boldsymbol{\nu}\right\|\ge
b_{\boldsymbol{\nu}}\,C\,k^{\epsilon-1/2}.
\end{equation}

Let us consider (\ref{eqn:integrale per Pi'1}) with $x=y$:
\begin{eqnarray}
 \label{eqn:integrale per Pi'2}
\lefteqn{\Pi_{k\boldsymbol{\nu}}(x,x)}\\
&\sim& \frac{k^2\,(\nu_1-\nu_2)}{(2\pi)^2}\,
  \int_{-\pi}^\pi\,\int_{-\pi}^\pi\,\int _{G/T}\,\int_{1/(2D)}^{2D}
  e^{\imath \,k\,\Psi_{x,x}}\,\mathcal{A}'_{x,x}\,\mathrm{d}u\,
  \mathrm{d}V_{G/T}(gT)\,\mathrm{d}\boldsymbol{\vartheta}.\nonumber
\end{eqnarray}
Let us choose $\epsilon'\in (0,\epsilon)$ and multiply the integrand in (\ref{eqn:integrale per Pi'2}) by the
identity
\begin{equation*}
\varrho \left(k^{1/2-\epsilon'}\,\mathrm{dist}_X\big(x(g,\boldsymbol{\vartheta}),x\big)\right)+
\left[1-\varrho \left(k^{1/2-\epsilon'}\,\mathrm{dist}_X\big(x(g,\boldsymbol{\vartheta}),x\big)\right)\right]=1.
\end{equation*}
Here $\varrho$ is as in the discussion preceding Lemma \ref{lem:only first summand}.
We obtain a further splitting 
\begin{eqnarray}
 \label{eqn:integrale per Pi'3}
\Pi_{k\boldsymbol{\nu}}(x,x)\sim \Pi_{k\boldsymbol{\nu}}(x,x)_a+\Pi_{k\boldsymbol{\nu}}(x,x)_b,
\end{eqnarray}
where $\Pi_{k\boldsymbol{\nu}}(x,x)_a$ is given by
(\ref{eqn:integrale per Pi'2}) with the amplitude $\mathcal{A}_{x,x}'$ replaced by 
\begin{equation}
 \label{eqn:defn di B'}
\mathcal{B}_{x,x}':=\varrho \left(k^{1/2-\epsilon'}\,
\mathrm{dist}_X\big(x(g,\boldsymbol{\vartheta}),x\big)\right)\,\mathcal{A}_{x,x}';
\end{equation}
similarly, $\Pi_{k\boldsymbol{\nu}}(x,x)_b$ is given by
(\ref{eqn:integrale per Pi'2}) with the amplitude $\mathcal{A}_{x,x}'$ replaced by 
$$
 \mathcal{B}_{x,x}'':=\left[1-\varrho \left(k^{1/2-\epsilon'}\,\mathrm{dist}_X\big(x(g,\boldsymbol{\vartheta}),x\big)\right)\right]\,
\mathcal{A}_{x,x}'.
$$

\begin{lem}
\label{lem:Pikb}
 $\Pi_{k\boldsymbol{\nu}}(x,x)_b=O\left(k^{-\infty}\right)$ as $k\rightarrow+\infty$.
\end{lem}

\begin{proof}
[Proof of Lemma \ref{lem:Pikb}.]
On the support of $\mathcal{B}_{x,x}''$, we have
\begin{equation}
 \label{eqn:distance k xx}
\mathrm{dist}_X\big(x(g,\boldsymbol{\vartheta}),x\big)\ge k^{\epsilon'-1/2}.
\end{equation}
Thus we may again appeal to (\ref{eqn:key identity du}) and iteratively integrate by parts in $\mathrm{d}u$, introducing at each step
a factor $O\left(k^{-1}\,k^{1-2\epsilon'}\right)=O\left(k^{-2\epsilon'}\right)$.

\end{proof}


Thus the proof of the Theorem will be complete once we establish the following.

\begin{lem}
\label{lem:Pika}
 $\Pi_{k\boldsymbol{\nu}}(x,x)_a=O\left(k^{-\infty}\right)$ as $k\rightarrow+\infty$.
\end{lem}

Before attacking the proof of Lemma \ref{lem:Pika}, let us prove the following.

\begin{lem}
 \label{lem:theta differential epsilon}
If (\ref{eqn:distance to GMTnu}) holds,
then for any $u\in [1/(2D),2D]$ and $k\gg 0$
\begin{equation}
 \label{eqn:distance_regular_fcnt g u1}
 \left\|\mathrm{d}^{(\boldsymbol{\vartheta})}_{(u,gT,\boldsymbol{\vartheta})}\Psi_{x,x}\right\|\ge
\frac{b_{\boldsymbol{\nu}}}{2}\,C\,k^{\epsilon-1/2}
\end{equation}
on the support of $\mathcal{B}_{x,x}'$.
\end{lem}

\begin{proof}
 [Proof of Lemma \ref{lem:theta differential epsilon}]
On the support of $\mathcal{B}_{x,x}'$, we have
\begin{equation}
 \label{eqn:distance k xx1}
\mathrm{dist}_X\big(x(g,\boldsymbol{\vartheta}),x\big)\le 2\,k^{\epsilon'-1/2}.
\end{equation}
Thus instead of (\ref{eqn:differential of psi}) we have
\begin{equation}
 \label{eqn:differential of psi1}
 \mathrm{d}_{(x(g,\boldsymbol{\vartheta}),x)}\psi=
 (\alpha_{x(g,\boldsymbol{\vartheta})},-\alpha_{x})+O\left(k^{\epsilon'-1/2}\right).
\end{equation}
Therefore, in place of (\ref{eqn:vartheta derivative}) on the support of $\mathcal{B}_{x,x}'$ we have
\begin{eqnarray}
\label{eqn:vartheta derivative1}
 \mathrm{d}^{(\boldsymbol{\vartheta})}_{(u,gT,\boldsymbol{\vartheta})}\Psi_{x,x}
 &=&\frac{u}{\imath}\,\Phi_T\big(\mu_{g^{-1}}(m_x)\big)-\boldsymbol{\nu}+
 O\left(k^{\epsilon'-1/2}\right).
\end{eqnarray}
Thus in view of (\ref{eqn:distance_regular_fcnt g u}) 
the claim follows since $0<\epsilon'<\epsilon$.
\end{proof}

Given Lemma \ref{lem:theta differential epsilon}, we can prove Lemma \ref{lem:Pika} essentially by iteratively integrating by parts 
in $\mathrm{d}\boldsymbol{\vartheta}$.

\begin{proof}
[Proof of Lemma \ref{lem:Pika}.]
Since $\widetilde{\mu}$ is free on $X^G_{\mathcal{O}}$, it is also 
free on a small tubular neighborhood $X'$ of $X^G_{\mathcal{O}}$ in $X$. Without loss, we may restrict our analysis 
to $X'$ in view of Theorem \ref{thm:rapid decrease fixed}. 
%


On the support of $\mathcal{B}_{x,x}'$, therefore, $e^{\imath\,\boldsymbol{\vartheta}}\in T$
varies in a small neighborhood of $I_2$.
Let $f:T\rightarrow [0,+\infty)$ be a bump function compactly supported in a small 
neighborhood $U\subset T$ of $I_2$ (identified with $(1,1)$), and
identically $=1$ near $I_2$.
Then we obtain
\begin{eqnarray}
 \label{eqn:Pi_knua}
\Pi_{k\boldsymbol{\nu}}(x,x)_a
&\sim& \left(\frac{k}{2\pi}\right)^2\,(\nu_1-\nu_2)\\
&&\cdot 
  \int_{U}\,\int _{G/T}\,\int_{1/(2D)}^{2D}
  e^{\imath \,k\,\Psi_{x,x}}\,f(t)\,\mathcal{B}_{x,x}'\,\mathrm{d}u\,
  \mathrm{d}V_{G/T}(gT)\,\mathrm{d}\boldsymbol{\vartheta}.\nonumber
\end{eqnarray}

Let us introduce the differential operator
\begin{equation}
 \label{eqn:operatorP}
 P=\sum_{h=1}^2\frac{\partial_{\vartheta_h}\Psi_{x,x}}{\left(\partial_{\vartheta_1}\Psi_{x,x}\right)^2
 +\left(\partial_{\vartheta_2}\Psi_{x,x}\right)^2}\,\frac{\partial}{\partial\vartheta_h},
\end{equation}
so that
$$
\frac{1}{\imath\,k}\,P\left(e^{ik\Psi_{x,x}}\right)=e^{ik\Psi_{x,x}}.
$$
Thus, 
\begin{eqnarray}
\label{eqn:integration by parts}
\lefteqn{\int_{U}e^{\imath \,k\,\Psi_{x,x}}\,f(t)\,\mathcal{B}_{x,x}'\,\mathrm{d}\boldsymbol{\vartheta}}\\
&=&\frac{1}{\imath\,k}\,\sum_{h=1}^2\,\int_{U}\frac{\partial_{\vartheta_h}\Psi_{x,x}}{\left(\partial_{\vartheta_1}\Psi_{x,x}\right)^2
 +\left(\partial_{\vartheta_2}\Psi_{x,x}\right)^2}\,
 \frac{\partial}{\partial\vartheta_h}\left[e^{\imath \,k\,\Psi_{x,x}}\right]\,
 f\left(e^{\imath \boldsymbol{\vartheta}}\right)\,\mathcal{B}_{x,x}'\,\mathrm{d}\boldsymbol{\vartheta}
 \nonumber\\
 &=&\frac{\imath}{k}\,\sum_{h=1}^2\,\int_{U}e^{\imath \,k\,\Psi_{x,x}}\,\frac{\partial}{\partial\vartheta_h}\left[
 \frac{\partial_{\vartheta_h}\Psi_{x,x}}{\left(\partial_{\vartheta_1}\Psi_{x,x}\right)^2
 +\left(\partial_{\vartheta_2}\Psi_{x,x}\right)^2}\, 
 f\left(e^{\imath \boldsymbol{\vartheta}}\right)\,\mathcal{B}_{x,x}'\,\right]\,\mathrm{d}\boldsymbol{\vartheta}
 \nonumber\\
 &=&\frac{\imath}{k}\,\int_{U}e^{\imath \,k\,\Psi_{x,x}}\,P^t\big(f(t)\,\mathcal{B}_{x,x}'\big)\,\mathrm{d}\boldsymbol{\vartheta},
\end{eqnarray}
where
\begin{equation}
 \label{eqn:defn di P^t}
 P^t(\gamma):=\sum_{h=1}^2\,\frac{\partial}{\partial\vartheta_h}\left[
 \frac{\partial_{\vartheta_h}\Psi_{x,x}}{\left(\partial_{\vartheta_1}\Psi_{x,x}\right)^2
 +\left(\partial_{\vartheta_2}\Psi_{x,x}\right)^2}\, 
 \gamma\,\right].
\end{equation}

Iterating, for any $r\in \mathbb{N}$ we have
\begin{eqnarray}
\label{eqn:integration by parts ktimes}
\int_{U}e^{\imath \,k\,\Psi_{x,x}}\,f(t)\,\mathcal{B}_{x,x}'\,\mathrm{d}\boldsymbol{\vartheta}
&=&\frac{\imath^r}{k^r}\,\int_{U}e^{\imath \,k\,\Psi_{x,x}}
\,\left(P^t\right)^r\big(f(t)\,\mathcal{B}_{x,x}'\big)\,\mathrm{d}\boldsymbol{\vartheta}.
\end{eqnarray}

Let us consider the function
\begin{equation}
 \label{eqn:defn of calD}
\mathcal{D}:\boldsymbol{\vartheta}\mapsto 
\mathrm{dist}_X\big(x(g,\boldsymbol{\vartheta}),x\big)=
\mathrm{dist}_X\left(\widetilde{\mu}_{e^{-\imath\,\boldsymbol{\vartheta}}}\circ\widetilde{\mu}_{g^{-1}}(x),\mu_{g^{-1}}(x)\right).
\end{equation}


We have the following.

\begin{lem}
 \label{lem:bound on distance theta}
For  $\boldsymbol{\vartheta}\sim \mathbf{0}$, we have
$$
\mathrm{dist}_X\big(x(g,\boldsymbol{\vartheta}),x\big)=
F_1(g\,T; \boldsymbol{\vartheta})+F_2(g\,T; \boldsymbol{\vartheta})+\cdots,
$$
where $F_j(g\,T; \boldsymbol{\vartheta})$ is homogeneous of degree $j$ in $\boldsymbol{\vartheta}$,
and $\mathcal{C}^\infty$ for $\boldsymbol{\vartheta}\neq \mathbf{0}$.
In addition, 
$F_1(g\,T; \boldsymbol{\vartheta})=\|\mathrm{Ad}_g(\boldsymbol{\vartheta})_X(x)\|=
\left\|\boldsymbol{\vartheta}_X\left(\widetilde{\mu}_{g^{-1}}(x)\right)\right\|$.
\end{lem}



For any $c\in \mathbb{N}$ let $\mathcal{D}^{(c)}$ denote a generic iterated derivative
of the form 
$$
\frac{\partial^c\,\mathcal{D}}{\partial\vartheta_{i_1}\,\cdots\partial\vartheta_{i_c}};
$$
clearly $\mathcal{D}^{(c)}$ is not uniquely determined by $c$.
By Lemma \ref{lem:bound on distance theta}, as $k\rightarrow+\infty$
$$
\mathcal{D}^{(c)}=O\left(k^{(c-1)(1/2-\epsilon')}\right)
$$
where $\varrho\left(k^{1/2-\epsilon'}\,
\mathcal{D}\right)\not\equiv 1$.  
For any multi-index $\mathbf{C}=(c_1,\ldots,c_s)$ let us denote by $\mathcal{D}^{(\mathbf{C})}$ a generic product
of the form 
$\mathcal{D}^{(c_1)}\cdots\,\mathcal{D}^{(c_s)}$;
then 
\begin{equation}
 \label{eqn:order of growth D}
 \mathcal{D}^{(\mathbf{C})}=O\left(k^{(1/2-\epsilon')\sum_j(c_j-1)}\right).
\end{equation}

\begin{lem}
\label{lem:generic summand}
 For any $r\in \mathbb{N}$, $\left(P^t\right)^r\big(f(t)\,\mathcal{B}_{x,x}'\big)$
 is a linear combination of summands of the form
 \begin{equation}
  \label{eqn:general summand abC}
  \varrho^{ (b)}\left(k^{1/2-\epsilon'}\,
D_k(\boldsymbol{\vartheta})\right)\,
\frac{P_{a_1}(\Psi_{x,x},\partial\Psi_{x,x})}{\left[\left(\partial_{\vartheta_1}\Psi_{x,x}\right)^2
 +\left(\partial_{\vartheta_2}\Psi_{x,x}\right)^2\right]^{a_2}}\,k^{b(1/2-\epsilon')}\,\mathcal{D}^{(\mathbf{C})},
 \end{equation}
times omitted factors bounded in $k$ depending on $f_j$ and its derivatives, where:
\begin{enumerate}
 \item $P_{a_1}$ denotes a generic differential polynomial in $\Psi_{x,x}$, homogeneous of degree $a_1$ in the first derivatives
 $\partial\Psi_{x,x}$;
 \item if $a:=2a_2-a_1$, then $a,b,\mathbf{C}$ are subject to the bound
\begin{equation}
 \label{eqn:bound abC r}
a+b+\sum_{j=1}^r(c_j-1)\le 2\,r
\end{equation}
(the sum is over the $c_j>0$);
\item $\mathbf{C}$ is not zero if and only if $b>0$.
\end{enumerate}

\end{lem}

Here $\varrho^{ (l)}$ is the $l$-th derivative of the one-variable real function $\varrho$.

\begin{proof}
[Proof of Lemma \ref{lem:generic summand}]
Let us set $F:=f_j\left(e^{\imath \boldsymbol{\vartheta}}\right)\,\mathcal{B}_{x,x}'$. For $r=1$, we have
\begin{eqnarray}
 \label{eqn:-2epsilon}
 \lefteqn{\frac{\partial}{\partial\vartheta_h}\left[
 \frac{\partial_{\vartheta_h}\Psi_{x,x}}{\left(\partial_{\vartheta_1}\Psi_{x,x}\right)^2
 +\left(\partial_{\vartheta_2}\Psi_{x,x}\right)^2}\, 
 F\,\right]}\\
 &=&\frac{\partial_{\vartheta_h}\Psi_{x,x}}{\left(\partial_{\vartheta_1}\Psi_{x,x}\right)^2
 +\left(\partial_{\vartheta_2}\Psi_{x,x}\right)^2}\,\frac{\partial\, F}{\partial\vartheta_h}+
  F\,\frac{\partial}{\partial\vartheta_h}\left[
 \frac{\partial_{\vartheta_h}\Psi_{x,x}}{\left(\partial_{\vartheta_1}\Psi_{x,x}\right)^2
 +\left(\partial_{\vartheta_2}\Psi_{x,x}\right)^2}\right].\nonumber
\end{eqnarray}
We have
\begin{eqnarray}
 \label{eqn:1st bound int by parts}
\lefteqn{\frac{\partial_{\vartheta_h}\Psi_{x,x}}{\left(\partial_{\vartheta_1}\Psi_{x,x}\right)^2
 +\left(\partial_{\vartheta_2}\Psi_{x,x}\right)^2}\,\frac{\partial\, F}{\partial\vartheta_h}}\\
 &=&\frac{\partial_{\vartheta_h}\Psi_{x,x}}{\left(\partial_{\vartheta_1}\Psi_{x,x}\right)^2
 +\left(\partial_{\vartheta_2}\Psi_{x,x}\right)^2}\,\left[\frac{\partial\, f_j}{\partial\vartheta_h}\,\mathcal{B}_{x,x}'+
 \frac{\partial\,\mathcal{B}_{x,x}'}{\partial\vartheta_h}\, f_j\right].\nonumber
\end{eqnarray}
Thus, in view of (\ref{eqn:defn di B'}), the first summand on the right hand side of (\ref{eqn:-2epsilon})
splits as a linear combination of terms as in the statement, with $a_1=a_2=1$, $b$ and $\mathbf{C}$ both zero,
or $a_1=a_2=1$, $b=1$, $\mathbf{C}=(1)$. Hence $a+b+\sum_j(c_j-1)=2$ in either case.
On the other hand, the second summand on the right hand side of (\ref{eqn:-2epsilon}) satisfies
\begin{eqnarray}
\label{eqn:2nd bound int by parts}
 \lefteqn{F\,\frac{\partial}{\partial\vartheta_h}\left[
 \frac{\partial_{\vartheta_h}\Psi_{x,x}}{\left(\partial_{\vartheta_1}\Psi_{x,x}\right)^2
 +\left(\partial_{\vartheta_2}\Psi_{x,x}\right)^2}\right]=\frac{F}{\left[\left(\partial_{\vartheta_1}\Psi_{x,x}\right)^2
 +\left(\partial_{\vartheta_2}\Psi_{x,x}\right)^2\right]^2}}\nonumber\\
 &&\cdot
 \left\{
 \partial^2_{\vartheta_h,\vartheta_h}\Psi_{x,x}\,\left[\left(\partial_{\vartheta_1}\Psi_{x,x}\right)^2
 +\left(\partial_{\vartheta_2}\Psi_{x,x}\right)^2
 \right]-2\,\partial_{\vartheta_h}\Psi_{x,x}\,
 \sum_{a=1}^2\partial_{\vartheta_a}\Psi_{x,x}\,\partial^2_{\vartheta_a\vartheta_h}\Psi_{x,x}\right\}\nonumber.
\end{eqnarray}
This is of the stated type with $a_1=a_2=2$, $b$ and $\mathbf{C}$ both zero. Hence $a=4-2=2$.

Passing to the inductive step, let us consider (\ref{eqn:defn di P^t}) with $\gamma$ given by
(\ref{eqn:general summand abC}), and assume that (\ref{eqn:bound abC r}) is satisfied. 
Let us write $\varrho^{ (l)}$ for the factor in front in (\ref{eqn:general summand abC}). 
We obtain a linear combination of expressions of the form
\begin{equation}
 \label{eqn:general summand abCgamma}
\frac{\partial}{\partial\vartheta_h}\left[\varrho^{ (b)}\,
\frac{P_{a_1+1}(\Psi_{x,x},\partial\Psi_{x,x})}{\left[\left(\partial_{\vartheta_1}\Psi_{x,x}\right)^2
 +\left(\partial_{\vartheta_2}\Psi_{x,x}\right)^2\right]^{a_2+1}}\,k^{b(1/2-\epsilon')}\,\mathcal{D}^{(\mathbf{C})}\right].
\end{equation}
It is clear that (\ref{eqn:general summand abCgamma}) splits as a linear combination of summands of the following forms:
\begin{equation}
 \label{eqn:form1}
 \varrho^{ (b)}\,\frac{P_{a'}(\Psi_{x,x},\partial\Psi_{x,x})}{\left[\left(\partial_{\vartheta_1}\Psi_{x,x}\right)^2
 +\left(\partial_{\vartheta_2}\Psi_{x,x}\right)^2\right]^{a_2+1}}\,k^{b(1/2-\epsilon')}\,\mathcal{D}^{(\mathbf{C})},
\end{equation}
with $a'\in \{a_1,a_1+1,a_1+2\}$; 
\begin{equation}
 \label{eqn:form2}
\varrho^{ (b)}\,\frac{P_{a_1+2}(\Psi_{x,x},\partial\Psi_{x,x})}{\left[\left(\partial_{\vartheta_1}\Psi_{x,x}\right)^2
 +\left(\partial_{\vartheta_2}\Psi_{x,x}\right)^2\right]^{a_2+2}}\,k^{b(1/2-\epsilon')}\,\mathcal{D}^{(\mathbf{C})};
\end{equation}
\begin{equation}
 \label{eqn:form3}
\varrho^{ (b+1)}\frac{P_{a_1+1}(\Psi_{x,x},\partial\Psi_{x,x})}{\left[\left(\partial_{\vartheta_1}\Psi_{x,x}\right)^2
 +\left(\partial_{\vartheta_2}\Psi_{x,x}\right)^2\right]^{a_2+1}}\,k^{(b+1)(1/2-\epsilon')}\,\mathcal{D}^{(\mathbf{C}')},
\end{equation}
where $\mathbf{C}'$ is of the form $\mathbf{C}'=(1,\mathbf{C})$;
\begin{equation}
 \label{eqn:form4}
\varrho^{ (b)}\,\frac{P_{a_1+1}(\Psi_{x,x},\partial\Psi_{x,x})}{\left[\left(\partial_{\vartheta_1}\Psi_{x,x}\right)^2
 +\left(\partial_{\vartheta_2}\Psi_{x,x}\right)^2\right]^{a_2+1}}\,k^{b(1/2-\epsilon')}\,\mathcal{D}^{(\mathbf{C}')},
\end{equation}
where $\mathbf{C}'$ is obtained from $\mathbf{C}$ (if the latter is not zero)
by replacing one of the $c_j$'s by $c_j+1$, and leaving all the others unchanged.

In all these cases we obtain a term of the form (\ref{eqn:general summand abC}), satisfying (\ref{eqn:bound abC r}) with
$r$ replaced by $r+1$. This completes the proof of Lemma \ref{lem:generic summand}.
\end{proof}

As $0<\epsilon'<\epsilon$, the general summand (\ref{eqn:general summand abC}) is 
\begin{eqnarray*}
 O\left(k^{a(1/2-\epsilon)+[b+\sum_j(c_j-1)](1/2-\epsilon')}\right)&=&
O\left(k^{[a+b+\sum_j(c_j-1)](1/2-\epsilon')}\right)\\
&=&O\left(k^{2r(1/2-\epsilon')}\right)=
O\left(k^{r(1-2\epsilon')}\right).
\end{eqnarray*}

Making use of the latter estimate in (\ref{eqn:integration by parts ktimes}), we obtain
the following:

\begin{cor}
 \label{cor:integrand estimate}
For any $r\in \mathbb{N}$,
 \begin{eqnarray}
\label{eqn:integration by parts ktimes1}
\int_{U_j}e^{\imath \,k\,\Psi_{x,x}}\,f(t)\,\mathcal{B}_{x,x}'\,\mathrm{d}\boldsymbol{\vartheta}
&=&O\left(k^{-2r\,\epsilon'}\right).
\end{eqnarray}
\end{cor}

The proof of Lemma \ref{lem:Pika} is thus complete.
\end{proof}

Given (\ref{eqn:integrale per Pi'3}), Proposition \ref{prop:diagonal case distance MGO} follows
from Lemmata \ref{lem:Pikb} and \ref{lem:Pika}.
\end{proof}

Thus the statement of Theorem \ref{thm:rapid decrease slow} holds true when $x=y$.
The general case follows from this and the Schwartz inequality
$$
\big|\Pi_{k\boldsymbol{\nu}}(x,y)\big|\le
\sqrt{\Pi_{k\boldsymbol{\nu}}(x,x)}\,\sqrt{\Pi_{k\boldsymbol{\nu}}(y,y)};
$$
in fact both factors on the right hand side have at most polynomial growth in $k$ by Lemma \ref{lem:a priori bound}, and 
if say (\ref{eqn:distance to GMTnu}) holds, then the first one is rapidly decreasing.
The proof of Theorem \ref{thm:rapid decrease slow} is complete.
\end{proof}

\section{Proof of Theorems \ref{thm:border pointwise asymptotics}, \ref{thm:border rescaled asymptotics} and 
\ref{thm:outer dimension estimate}}

\subsection{Preliminaries on local rescaled asymptotics}

\label{sctn:integral formula rescaled asymptotics}

In the proof of Theorems \ref{thm:border pointwise asymptotics}, \ref{thm:border rescaled asymptotics}  and 
\ref{thm:outer dimension estimate},
we are interested in the asymptotics of $\Pi_{k\boldsymbol{\nu}}(x',x'')$ when $(x',x'')$ approaches the diagonal of $X^G_{\mathcal{O}}$
in $X\times X$ along appropriate directions and at a suitable pace. 

In Theorems \ref{thm:border pointwise asymptotics}  and 
\ref{thm:outer dimension estimate}, we consider $x'=x''$  
in a shrinking \lq one-sided\rq \, neighborhood of $X^G_{\mathcal{O}}$. In Theorem \ref{thm:border rescaled asymptotics},
we shall assume that $(x',x'')$ approaches the diagonal in $X^G_{\mathcal{O}}$ along \lq horizontal\rq \, directions
orthogonal to the orbits. We shall treat the former case in detail, and then briefly discuss the necessary changes for the latter.

Suppose $x\in X^G_{\mathcal{O}}$ and let $m=\pi(x)$. Let us choose a system of HLC centered at $x$, and let us consider the
collection of points
\begin{equation}
 \label{eqn:defn di xtau}
x_{\tau,k}:=x+\frac{\tau}{\sqrt{k}}\,\Upsilon_{\boldsymbol{\nu}}(m),
\end{equation}
where $k=1,2,\ldots$, and $|\tau|\le C\,k^\epsilon$ for some fixed $C>0$ and $\epsilon \in (0,1/6)$. 
The sign of $\tau$ is chosen so that
$\tau\, \Upsilon_{\boldsymbol{\nu}}(m)$ is either zero or outer oriented. Thus $\tau\, (\nu_1+\nu_2)\ge 0$.
We shall provide an integral espression for 
the asymptotics of $\Pi_{k\boldsymbol{\nu}}(x_{\tau,k},x_{\tau,k})$ when $k\rightarrow +\infty$.

Applying as before the Weyl integration and character formulae, inserting the 
microlocal description of $\Pi$ as an FIO, and making use of the rescaling $u\mapsto k\,u$,
$\boldsymbol{\vartheta}\mapsto \boldsymbol{\vartheta}/\sqrt{k}$, we obtain that, as $k\rightarrow +\infty$,

\begin{eqnarray}
 \label{eqn:projection composedk}
\lefteqn{
\Pi_{k\boldsymbol{\nu}}(x_{\tau,k},x_{\tau,k})
}\\
&\sim&
\dfrac{k\,(\nu_1-\nu_2)}{(2\pi)^2}\,
\int_{G/T}\,\mathrm{d}V_{G/T}(gT)\,\int_{-\infty}^{\infty}\mathrm{d}\vartheta_1\,\int_{-\infty}^{\infty}\mathrm{d}\vartheta_2\,
\int_0^{+\infty}\,\mathrm{d}u\nonumber\\
&&\left[
e^{\imath\,k\,
\left[u\,\psi\left(\widetilde{\mu}_{g\,e^{-\imath\boldsymbol{\vartheta}/\sqrt{k}}\,g^{-1}}(x_{\tau,k}),x_{\tau,k}\right)
-\langle\boldsymbol{\vartheta},\boldsymbol{\nu}\rangle/\sqrt{k}
\right]}
\right.\nonumber\\
&& \left. \cdot \Delta\left(e^{\imath\boldsymbol{\vartheta}/\sqrt{k}}\right)\,
s\left(\widetilde{\mu}_{g\,e^{-\imath\boldsymbol{\vartheta}/\sqrt{k}}\,g^{-1}}(x_{\tau,k}),x_{\tau,k},k\,u\right)\right].\nonumber
\end{eqnarray}
Integration
in $\boldsymbol{\vartheta}=(\vartheta_1, \vartheta_2)$ is over a ball centered at the origin and
radius $O\left( k^\epsilon \right)$ in $\mathbb{R}^2$. A cut-off function of the form
$\varrho \left( k^{-\epsilon} \, \boldsymbol{\vartheta} \right)$
is implicitly incorporated into the amplitude.

In order to express the previous phase more explicitly, we need the following Definition.

\begin{defn}
\label{defn:def di Psi}
 Let us define $\boldsymbol{\rho}=\boldsymbol{\rho}_m:G/T\rightarrow \mathfrak{t}\cong \mathbb{R}^2$,
$g\,T\mapsto \boldsymbol{\rho}_{g\,T}$, by requiring
$$
\langle \boldsymbol{\rho}_{g\,T},\boldsymbol{\vartheta}\rangle=
\omega_m\Big(\mathrm{Ad}_g (\imath\, D_{\boldsymbol{\vartheta}})_M(m), \Upsilon_{\boldsymbol{\nu}}(m)\Big)
\quad (\mathfrak{\vartheta}\in \mathbb{R}^2).
$$

Next, let the symmetric and positive definite matrix $E( g\,T )=E_x( g\,T )$ be defined by the equality 
$$
\boldsymbol{\vartheta}^t \,E( g\,T )\, \boldsymbol{\vartheta} =
\big\|  \mathrm{Ad}_g (\imath\, D_{\boldsymbol{\vartheta}}) _X(x) \big\| _x ^2 
\quad (\mathfrak{\vartheta}\in \mathbb{R}^2).
$$

Furthermore, let us define 
$\widetilde{\Psi}(u, g\,T , \tau )=\widetilde{\Psi}_m(u, g\,T , \tau )\in \mathfrak{t}$ by setting
\begin{eqnarray*}
 \widetilde{\Psi}(u, g\,T ) :=
 u\,  
\mathrm{ diag }
\big ( \mathrm{Ad}_{g^{-1}} \big( \Phi_G' (m) \big) 
%
-  \boldsymbol{\nu}, \quad  \Phi_G' (m) := -\imath \,\Phi_G (m)  .
\end{eqnarray*}

Finally, let us pose
$$
\Psi(u, g\,T , \boldsymbol{\vartheta} )
:=\left\langle \widetilde{\Psi}(u, g\,T  ), \boldsymbol{\vartheta} \right\rangle.
$$
\end{defn}

The following Proposition is proved by a 
rather lengthy computation, along the lines of those  in the proof of Theorem \ref{thm:rapid decrease slow}
and in \cite{pao-IJM}.

\begin{prop}
\label{prop:phase border} 
\begin{eqnarray*}
\lefteqn{
\imath \, k \, \left [ 
u \, \psi \left( \widetilde{\mu}_{ g \, e^{-\imath\boldsymbol{\vartheta}/\sqrt{k}}\,g^{-1}} (x_{\tau,k}) , x_{\tau,k} \right)
- \frac{1}{\sqrt{k}}\, \langle \boldsymbol{\nu} , \boldsymbol{\vartheta} \rangle  \right]}\\
&=& 
\imath\, \sqrt{k}\, \, 
\Psi(u, g\,T , \boldsymbol{\vartheta} )  - \frac{ u }{ 2 } \, 
\boldsymbol{\vartheta} ^t \, E( g\,T )\, \boldsymbol{\vartheta} +2\,\imath\,u\,\tau \,\big\langle\boldsymbol{\rho}_{g\,T} ,
\boldsymbol{\vartheta}\big\rangle  + k \, R_3 \left( \frac{\tau}{\sqrt{k}},\,\frac { \boldsymbol{\vartheta} } {\sqrt{k}} \right). \nonumber
\end{eqnarray*}
\end{prop}

\begin{cor}
 \label{cor:projection composedk oscillatory}
(\ref{eqn:projection composedk}) may be rewritten as follows:
 \begin{eqnarray}
 \label{eqn:projection composedk1}
\lefteqn{
\Pi_{k\boldsymbol{\nu}}(x_{\tau,k},x_{\tau,k})
}\\
&\sim&
\dfrac{k\,(\nu_1-\nu_2)}{(2\pi)^2}\,
\int_{G/T}\,\mathrm{d}V_{G/T}(gT)\,\int_{-\infty}^{\infty}\mathrm{d}\vartheta_1\,\int_{-\infty}^{\infty}\mathrm{d}\vartheta_2\,
\int_0^{+\infty}\,\mathrm{d}u\nonumber\\
&&\left[
e^{\imath\, \sqrt{k}\, \, 
\Psi(u, g\,T , \boldsymbol{\vartheta} )}\,
\mathcal{A}_{k,\boldsymbol{\nu}}(u, g\,T , \tau, \boldsymbol{\vartheta} )
\right],\nonumber
\end{eqnarray}
where (leaving implicit the dependence on $x$)
\begin{eqnarray}
 \label{eqn:defn di Aknu}
\mathcal{A}_{k,\boldsymbol{\nu}}(u, g\,T , \tau, \boldsymbol{\vartheta} ) 
&:= &e^{ - \frac{ u }{ 2 } \, 
\boldsymbol{\vartheta} ^t \, E( g\,T )\, \boldsymbol{\vartheta} +2\,\imath\,u\,\tau \,\big\langle\boldsymbol{\rho}_{g\,T} ,
\boldsymbol{\vartheta}\big\rangle  
+ k \, R_3 \left(\frac{\tau}{\sqrt{k}}, \frac { \boldsymbol{\vartheta} } {\sqrt{k}} \right)}\,
\Delta\left(e^{\imath\boldsymbol{\vartheta}/\sqrt{k}}\right)\nonumber\\
&& \cdot s\left(\widetilde{\mu}_{g\,e^{-\imath\boldsymbol{\vartheta}/\sqrt{k}}\,g^{-1}}(x_{\tau,k}),x_{\tau,k},k\,u\right).
\end{eqnarray}
 
\end{cor}

Let $h_m\,T \in G/T$ be the
unique coset such that $h_m^{-1}\, \Phi_G(m)\, h_m$ is diagonal. 
Then only a rapidly decreasing contribution to the asymptotics is lost in (\ref{eqn:projection composedk1}), 
if integration 
in $\mathrm{d}V_{G/T}$ is localized in a small neighborhood of $h_m\,T$.
In the following, a $\mathcal{C}^\infty$ bump function on $G/T$, supported in a small neighborhood of $h_m \, T$
and identically equal to $1$ near $h_m \, T$, will be implicitly incorporated into the amplitude (\ref{eqn:defn di Aknu}).

For some choice of $h_m\in h_m\,T$ and $\delta>0$ sufficiently small, let us consider the real-analytic map 
$$
h: w \in B(0;\delta)\subset \mathbb{C} \mapsto h (w) := h_m \, \exp \left ( \imath 
\begin{pmatrix}
 0 & w \\
 \overline{w} & 0
\end{pmatrix} 
\right ) \in G.
$$
By composition with the projection
$\pi : G \rightarrow G/T $, we obtain a real-analytic coordinate chart on $G/T$ centered at $h_m\, T \in G/T$,
given by $w \in B(0;\delta) \mapsto h(w) \, T \in G/T$. 
The Haar volume form on $G/T$ has the form 
$\mathcal{V}_{G/T}(w) \, \mathrm{d}V_{\mathbb{C}} (w)$, where $ \mathrm{d}V_{\mathbb{C}} (w)$
is the Lebesgue measure on $\mathbb{C}$, and $\mathcal{V}_{G/T}$ is a uniquely determined 
$\mathcal{C}^\infty$ positive function on $B(0;\delta)$. 
We record the following statements, whose proofs we shall omit for the sake of brevity.

\begin{lem}
 \label{lem:VGT}
$\mathcal{V}_{G/T}$ is rotationally invariant, that is,
$$
\mathcal{V}_{G/T}(w) = \mathcal{V}_{G/T} \left( e^{ \imath \,\theta } \, w \right),
$$
for all $w \in B ( 0 ; \delta )$ and $e^{ \imath \,\theta } \in S^1$. In particular,
$\mathcal{V}_{G/T}$ is given by a convergent power series in $r^2=|w|^2$ on $B ( 0 ; \delta )$.
\end{lem}

Thus we shall write
\begin{equation}
 \label{eqn:VGT serie}
\mathcal{V}_{G/T} ( w ) = \mathcal{V}_{G/T} ( r ) = D_{ G/T } \cdot \mathcal{S} _{G/T} ( r ),
\end{equation}
where $D_{ G/T }>0$ is a constant, and $\mathcal{S} _{G/T} ( r ) = 1 + \sum_j s_j\, r^{ 2 j }$.

\begin{lem}
 \label{lem:DGT} Let $V_3$ be the total area of the unit sphere $S^3\subset \mathbb{C}^2$. Then
$$
D_{ G/T }=(2\,\pi)^{-1/2}\, V_3^{-1}.
$$

\end{lem}

Furthermore, let us introduce the real-analytic function
\begin{equation}
 \label{eqn:defn of kappa}
\kappa = \kappa_m : w \in B (0 , \delta ) \mapsto \mathrm{ diag }
\Big ( \mathrm{Ad}_{h(w) ^{-1}} \big( \Phi_G' (m) \big) \Big)\in \mathbb{R}^2.
\end{equation}

Then we also have the following.

\begin{lem}
 \label{lem:rotational invariance}
$ \kappa $
is rotationally invariant, and is given by a convergent power series of the following form
\begin{eqnarray*}
\kappa ( w ) &=&\lambda _{ \boldsymbol{\nu} } ( m )\, \left [ \boldsymbol{\nu} - r ^ 2 \, (\nu_1 - \nu _2 ) \, S_\kappa ( r )\,\mathbf{b}
\right],\quad 
\mathbf{b}=
\begin{pmatrix}
 1\\
-1
\end{pmatrix},
\end{eqnarray*}
where $r= | w | $, and $S_\kappa (r)$ is a real-analytic function of $r$, of the form
$$
S_\kappa ( r ) = 1 + \sum_{ j \ge 1 } b_j \,r^{2 j }.
$$
\end{lem}

If $w=r \, e^{ \imath \theta }$ in polar coordinates, 
we shall write accordingly 
$\mathcal{V}_{G/T}=\mathcal{V}_{G/T}( r )$ and $\kappa =\kappa ( r )$.

Recalling Definition \ref{defn:def di Psi} and (\ref{eqn:defn of kappa}), let us set
\begin{eqnarray}
 \label{eqn:phase r}
\widetilde{\Psi}_{ w }(u)
:  =  u  \, \kappa ( r ) - \boldsymbol{ \nu },
\quad
\Psi_{ w }(u, \boldsymbol{\vartheta} )
:=\left\langle \widetilde{\Psi}_{ w  }(u), \boldsymbol{\vartheta} \right\rangle.
\end{eqnarray}
We obtain the following integral formula (dependence on $x$ on the right hand sides is left implicit).

\begin{prop}
 \label{prop:integral formula for Pik}
As $k\rightarrow +\infty$ we have
\begin{eqnarray}
 \label{eqn:projection composedk4}
\lefteqn{
\Pi_{k\boldsymbol{\nu}}(x_{\tau,k},x_{\tau,k})
}\\
&\sim&
D_{G/T}\,\frac{k\,(\nu_1-\nu_2)}{(2\pi)^2}\,
\int_{-\pi}^\pi\,\mathrm{d}\theta\, \int_0 ^{ +\infty } \mathrm{d}\,r \,\left[
I_k(\tau, r , \theta )\right],\nonumber
\end{eqnarray}
where
\begin{eqnarray}
 \label{eqn:defn of Iktau}
I_k(\tau, r , \theta )=I_k(\tau, w )
& : = & \int_{-\infty}^{\infty}\mathrm{d}\vartheta_1\,\int_{-\infty}^{\infty}\mathrm{d}\vartheta_2\,
\int_0^{+\infty}\,\mathrm{d}u\\
&&\left[
e^{\imath\, \sqrt{k}\, \, 
\Psi_{w}(u, \boldsymbol{\vartheta} ) }\,
\mathcal{A}_{k,\boldsymbol{\nu}}(u, h\left(r\,e^{ \imath \, \theta } \right)\,T , \tau, \boldsymbol{\vartheta} )
\, \mathcal{S}_{G/T}(r)\,r\right].\nonumber
\end{eqnarray}
\end{prop}

Our next goal is to produce an asymptotic expansion for $I_k(\tau, r , \theta )$.

\begin{defn}
 Let us set 
$$
\mathbf{n}_1( r ) : = \frac{k ( r )}{\big\|
k ( r ) \big\| } ,
$$
and let $\mathbf{n}_2( r )  $ be uniquely determined
for $|r| < \delta  $  so that 
$
\mathcal{ B }_{ r } : = \left( \mathbf{n}_1( r )  , \,
\mathbf{n}_2( r )   \right) 
$
 is a positively oriented orthonormal basis of $\mathbb{R}^2$. We shall write the change of basis matrix
in the form
\begin{equation}
 \label{eqn:matrix change of basis}
 M^{ \mathcal{ B }_{ r } }_{ \mathcal{C}_2 } ( id _{ \mathbb{R} ^2 }) = 
\begin{pmatrix}
 C ( r ) & - S ( r ) \\
S(r)  &  C (r)
\end{pmatrix},
\end{equation}
where $\mathcal{C}_2 $ is the canonical basis of $\mathbb{R}^2$, and denote the change of coordinates by 
$ \boldsymbol{\vartheta} = \zeta_1 \, \mathbf{n}_1( w ) + \zeta_2 \, \mathbf{n}_2( w )$.
\end{defn}

A straightforward computation then yields the following.

\begin{cor}
 With $w = r \, e^{ \imath \theta }\in B(0; \delta )$ and $I_k(\tau, w )$ as in (\ref{eqn:defn of Iktau}), we have:
\begin{eqnarray}
 \label{eqn:defn of Iktau1}
I_k(\tau, w )
&=& \int_{-\infty}^{\infty}\mathrm{d}\zeta_2\,  \left[ e^{ -\imath \, \sqrt{k}  \, 
\big\langle \boldsymbol{ \nu } , \mathbf{n}_2(  w ) \big \rangle \, 
\zeta _2}\,
J_k( \tau , w ; \zeta_2 )\, \mathcal{S}_{G/T}(r)\,r
\right],
\end{eqnarray}
where
\begin{eqnarray}
 \label{eqn:defn di Jk}
 \lefteqn{ J_k( \tau , w ; \zeta_2 ) } \\
& : = & 
\int_{-\infty}^{\infty}\mathrm{d}\zeta_1\,
\int_0^{+\infty}\,\mathrm{d}u\, \left[
e^{\imath\, \sqrt{k}\,  \Upsilon_{  r }( u , \zeta_1 )   }\,
\mathcal{A}_{k,\boldsymbol{\nu}}\big(u, h\left(w \right)\,T , \tau, \boldsymbol{\vartheta}(\boldsymbol{ \zeta } ) \big)
\right],\nonumber
\end{eqnarray}
and
$$
\Upsilon_{  r }( u , \zeta_1 ) : = 
\big[ u\, \| \kappa ( r ) \|- \langle \boldsymbol{ \nu } , \mathbf{n}_1( r )   \rangle \big]\, 
\zeta_1.
$$
\end{cor}

Let us view $J_k$ (\ref{eqn:defn di Jk}) as an oscillatory integral with phase $\Upsilon_{  r }$.

\begin{lem}
 \label{lem:critical point}
$\Upsilon_{  r }$ has the unique critical point
$$
P_{  r } = \big( u (  r ) , 0 \big)
 : = \left( \frac{\langle \boldsymbol{ \nu } , \mathbf{n}_1( r )   \rangle }{ \| \kappa (r) \| } ,
0 \right).
$$
Furthermore, $\Upsilon_{  r } \big( P_{  r } \big) = 0$, and the Hessian matrix is
$$
H (\Upsilon_{  r } ) _{ P_{  r } } = 
\begin{pmatrix}
 0 & \| \kappa (r)  \| \\
\| \kappa (r)  \| & 0
\end{pmatrix}.
$$
Hence its signature is zero and the critical point is non-degenerate.
\end{lem}

In view of (\ref{eqn:defn di Aknu}), and recalling that 
$s_0(x,x)=\pi ^{-d}$,
the amplitude in (\ref{eqn:defn di Jk}) may be rewritten in the following form:
\begin{eqnarray}
 \label{eqn:re-expression of Aknu}
\lefteqn{ \mathcal{A}_{k,\boldsymbol{\nu}}\big (u, h(w)\,T , \tau, \boldsymbol{\vartheta}( \boldsymbol{\zeta} ) \big) }\\
&\sim &e^{ - \frac{ u }{ 2 } \, 
\boldsymbol{\vartheta} ( \boldsymbol{\zeta} )^t \, E( w )\, \boldsymbol{\vartheta}( \boldsymbol{\zeta} )
 +2\,\imath\,u\,\tau \,\big\langle\boldsymbol{\rho}_{h(w)\,T} ,
\boldsymbol{\vartheta}( \boldsymbol{\zeta} )\big\rangle }
\,
\left[ e^{\frac{\imath}{\sqrt{k}} \, \vartheta_1( \boldsymbol{\zeta} ) } - 
e^{\frac{\imath}{\sqrt{k}} \, \vartheta_2( \boldsymbol{\zeta}) } \right]
\,\left( \frac{ k \, u }{ \pi } \right)^d \nonumber\\
&& \cdot \left[ 1
 + \sum _{j\ge 1} a_j \big( u , w ; \tau , \boldsymbol{\vartheta} ( \boldsymbol{\zeta} ) \big) \, k^{ - j/2 } \right] ;\nonumber
 \nonumber
\end{eqnarray}
in (\ref{eqn:re-expression of Aknu}) we have set 
$E ( w ) : = \widetilde{E}\big( h(w)\,T \big)$,
and in view of the 
exponent $k\,R_3 ( \tau/\sqrt{k} , \boldsymbol{\vartheta} / \sqrt{k} )$ appearing in (\ref{eqn:defn di Aknu}), 
$a_j( u , w;\cdot, \cdot)$ is an appropriate polynomial in $(\tau,\boldsymbol{\vartheta})$ of degree $\le 3j$.

Given Lemma \ref{lem:critical point}, we may evaluate $J_k$ in (\ref{eqn:defn di Jk})
by the Stationary Phase Lemma, and obtain an asymptotic expansion in descending powers of $k^{1/2}$. 
The latter expansion may be inserted in (\ref{eqn:defn of Iktau1}), and integrated term by term, thus leading to
an asymptotic expansion for $I_k$. 
The leading order term of either expansion is determined by the contribution of the leading order term 
in the asymptotic expansion for the amplitude in (40), which is given by the following:

\begin{eqnarray}
 \label{eqn:defn di Jk'}
  J_k'( \tau , w ; \zeta_2 )& = & 
\,\left( \frac{ k }{ \pi } \right)^d \,
\int_{-\infty}^{\infty}\mathrm{d}\zeta_1\,
\int_0^{+\infty}\,\mathrm{d}u\\
&&\left[
e^{\imath\, \sqrt{k}\,  \Upsilon_{  w }( u , \zeta_1 )  }\,u^d\,
\left(e^{\frac{\imath}{\sqrt{k}} \, \vartheta_1( \boldsymbol{\zeta} ) } - 
e^{\frac{\imath}{\sqrt{k}} \, \vartheta_2( \boldsymbol{\zeta}) } \right) \right. \nonumber\\
&&\left. \cdot 
  e^{ - \frac{ u }{ 2 } \, 
\boldsymbol{\vartheta} ( \boldsymbol{\zeta} )  ^t \, E( w )\, \boldsymbol{\vartheta} ( \boldsymbol{\zeta} )  
+2\,\imath\,u\,\tau \,\big\langle\boldsymbol{\rho}_{h(w)\,T} ,
\boldsymbol{\vartheta} ( \boldsymbol{\zeta} ) \big\rangle }\right].\nonumber
\end{eqnarray}

\begin{defn}
\label{defn:defn of atauw}
Suppose $w = r \, e^{ \imath \theta }\in B(0; \delta )$ and let $C(r)$ and $S(r)$ be as in (\ref{eqn:matrix change of basis}). 
Let us set
\begin{eqnarray*}
\mathfrak{a} ( w ) & := &  
u ( r )\,
\begin{pmatrix}
 -S(r) &
C(r)
\end{pmatrix} \, 
E \big( w \big)\,
\begin{pmatrix}
 -S(r) \\
C(r)
\end{pmatrix} \\
&=&u ( r )\, 
\big\|  \mathrm{Ad}_{h(w)} \big( \mathbf{n}_2 ( r )\big) _X(x) \big\| _x ^2
\end{eqnarray*}
and
\begin{eqnarray*}
\mathfrak{r}(w) &:= &2 \, u ( r )\, \big\langle \boldsymbol{\rho}_{h(w)\,T} ,
\mathbf{n}_2 ( r )\big\rangle \\
& = & 2\,u ( r ) \,
\omega_m\Big(\mathrm{Ad}_{h (w ) } \big(\mathbf{n}_2 ( r ) \big) _M(m), \Upsilon_{\boldsymbol{\nu}}(m)\Big).
\end{eqnarray*}
\end{defn}

Given the previous considerations, an application of the Stationary Phase Lemma yields the following. 

\begin{defn}
With $|r| < \delta$, let us set  
$ \mathfrak{b} ( r ) : = \big\langle \boldsymbol{ \nu } , \mathbf{n}_2( r ) \big \rangle$, and
 $$
D_l ( r ) : = \frac{\imath ^l}{l ! \,  \| \kappa ( r ) \|  } \,\left[ C( r )^l + ( -1 ) ^{ l - 1 } \, 
S(  r ) ^l \right] .
$$
\end{defn}

The definition of $\mathfrak{b} ( r ) $ implies:

\begin{equation}
 \label{eqn:computation b(r)}
\mathfrak{b} ( r ) = -\frac{ ( \nu _1 -\nu _2)  \, (\nu_1 + \nu_2 )}{\| \boldsymbol{ \nu } \| } \, 
r ^ 2 \, S _1( r ),
\end{equation}
where $S_1$ is a real-analytic function of the form $S_1 ( r ) = 1 + \sum_{j \ge 1 } c_j \, r^{ 2 j }$.

\begin{prop} 
\label{prop:partial statement}
Suppose $x\in X^G _{\mathcal{O}}$, and let $x_{ \tau , k }$ be as
in (\ref{eqn:defn di xtau}). Then as $k\rightarrow +\infty$ we have
\begin{eqnarray}
 \label{eqn:projection composedk5}
\lefteqn{
\Pi_{k\boldsymbol{\nu}}(x_{ \tau , k } , x_{ \tau , k } )
}\\
&\sim&
D_{G/T}\,\frac{k\,(\nu_1-\nu_2)}{(2\pi)^2}\,
\int_{-\pi}^\pi\,\mathrm{d}\theta\, \int_0 ^{ +\infty } \mathrm{d}\,r \,\left[
I_k(\tau, r , \theta )\right],\nonumber
\end{eqnarray}
where $I_k(\tau, r , \theta )$ is given by an asymptotic expansion in descending powers of $k^{1/2}$,
the leading power being $k ^ { d - 1 }$.
As a function of $\tau$, aside from a phase factor, the coefficient of $k^{ d -(1+j) / 2 }$ is a polynomial of degree $\le 3 j$.
Up to non-dominant terms
we may replace $I_k(\tau, w )$ by
\begin{eqnarray}
 \label{eqn:defn of Iktau2thm}
%
I_k(\tau, w )'
& 	= &      - \left( \frac{ k }{ \pi } \right)^d \,
\left(\frac{2\pi}{\sqrt{k}}\right)\,  \mathcal{S}_{G/T}(r)\,r  \cdot u ( w )^d \, \\
&&\cdot 
\sum_{l\ge 1} \,
\frac{ D_l ( r )}{k^{l/2}}\,\int_{-\infty}^{\infty}\mathrm{d}\zeta_2\,  \left[ 
e^{   -\imath \,\sqrt{k} \, \zeta _2\, \mathfrak{ f } _k (  \tau , w  ) }\, \zeta_2 ^l \cdot  
e^{- \frac{1}{2} \, \mathfrak{a} ( w )\, \zeta _2^2   }\right]
  , \nonumber
\end{eqnarray}
where for $k=1,2,\ldots$, we have set 
 \begin{equation}
  \label{eqn:defn di fk}
  \mathfrak{ f } _k (  \tau , w  ) := \mathfrak{b} ( r ) - \frac{\tau}{k^{1/2}}\, \mathfrak{r}(w).
 \end{equation}

\end{prop}

The Gaussian integrals in (\ref{eqn:defn of Iktau2thm}) may be estimated recalling that
 \begin{eqnarray}
\label{eqn:key gaussian integrall}
 \int _{ -\infty }^{ + \infty } x ^l\, e^{ -\imath \xi \, x -\frac{1}{2} \, \lambda \, x ^2 }\mathrm{d} x
 = \sqrt{ 2 \pi } \, \frac{(-\imath)^l}{\lambda^{l+1/2}} \, P_l(\xi)\, e^{ -\frac{1}{ 2 \lambda} \, \xi ^2},
\end{eqnarray} 
where $P_l(\xi)= \xi^l + \sum_{ j\ge 1 } \, p_{ lj }\, \xi ^{ l - 2 j }$ 
is a monic \textit{polynomial} in $\xi$, of degree $l$ and parity $(-1)^l$ (thus the previous sum is \textit{finite}).
Applying (\ref{eqn:key gaussian integrall}) with 
\begin{eqnarray*}
 \xi =  k^{1/2}\, \mathfrak{f}_k(w,\tau), \quad
\lambda = \mathfrak{a} ( w )
\end{eqnarray*} 
we obtain 
the following conclusion.

\begin{prop}
 \label{prop:I_k after gaussian integrals}
 Let us set 
 \begin{equation}
  \label{eqn:defn of Fl}
  F_l (\tau , w) : = 
\frac{\sqrt{ 2 \pi }}{ l! }\, \left[ \frac{C( r )^l + ( -1 ) ^{ l - 1 } \, 
S( r ) ^l}{  \| \kappa ( r ) \|  } \right] \, 
\frac{ P_l\left (\sqrt{k}\,\mathfrak{f} _k( \tau , w ) \right)}{  k^{l/2} \, \mathfrak{a} (  w )^{ l + 1/2}  }.
\end{equation}

Up to lower order terms, we can replace $I_k'$ in (\ref{eqn:defn of Iktau2thm}) by
\begin{eqnarray}
 \label{eqn:defn of Iktau3}
%
I_k(\tau, w )''
& 	:= &     - \left( \frac{ k }{ \pi } \right)^d \,
\left(\frac{2\pi}{\sqrt{k}}\right)\, \mathcal{S}_{G/T}(r)\,r  \cdot u ( w )^d \, \nonumber\\
&&\cdot e^{ -\frac{1}{ 2 } \,k \, \frac{ \mathfrak{f}_k ( \tau , w ) ^2}{\mathfrak{a} ( w )}}\, \sum_{l\ge 1} F_l ( \tau , w ).
\end{eqnarray}

\end{prop}

 Thus the leading order asymptotics of $\Pi_{k\boldsymbol{\nu}}(x_{ \tau , k } , x_{ \tau , k } )$ are obtained by replacing
$I_k(\tau, r , \theta )$ in (\ref{eqn:projection composedk5}) by $I_k(\tau, w )''$ given by (\ref{eqn:defn of Iktau3}).

\subsection{Proof of Theorem \ref{thm:border pointwise asymptotics}}

We shall set $\tau = 0$ in (\ref{eqn:projection composedk5}) and obtain an asymptotic estimate for
$\Pi_{k\boldsymbol{\nu}}(x , x )$ when $x\in X^G_{\mathcal{O}}$ and $k \rightarrow +\infty$.

\begin{proof}
 [Proof of Theorem \ref{thm:border pointwise asymptotics}]
It follows from the definitions that
\begin{equation}
 \label{eqn:exponent for tau=0}
\frac{ \mathfrak{f}_k ( 0 , w ) ^2}{\mathfrak{a} ( w )} = \frac{ \mathfrak{b} ( r ) ^2}{\mathfrak{a} ( w )}
= \lambda_{ \boldsymbol{ \nu } } ( m ) \,
D(\boldsymbol{ \nu } ) \,r^4 \, \mathcal{S} ( r , \theta) ,
\end{equation}
where $\mathcal{S} (r, \theta) = 1 +\sum_{j\ge 1} r^j\,d_j( \theta )$, and 
\begin{equation}
 \label{eqn:defn of Dnu}
D(\boldsymbol{ \nu } ) : =
\frac{ ( \nu_1 - \nu_2 ) ^2 \, (\nu_1+\nu_2)^2 }{\| \mathrm{Ad}_{ h_m } (\boldsymbol{ \nu _\perp })_M (m) \|^2_m}.
\end{equation}

Similarly, 
\begin{eqnarray}
 \label{eqn:general summand for tau=0}
 \frac{ P_l\left (\sqrt{k}\,\mathfrak{f} _k( 0 , w ) \right)}{  k^{l/2} \, \mathfrak{a} (  w )^{ l + 1/2}  }
 &=& \frac{ P_l\left (\sqrt{k}\,\mathfrak{b} ( r )\right)}{  k^{l/2} \, \mathfrak{a} (  w )^{ l + 1/2}  }  \\
 & = &  \frac{  1 }{  \mathfrak{a} (  w )^{ l + 1/2}  } \, \left[ \mathfrak{b} ( r ) ^l 
 + \sum_{ j\ge 1 } ^{ \lfloor l/2 \rfloor } \, p_{ lj }\, k ^{- j }\,\mathfrak{b} ( r )^{ l - 2j } \right]\nonumber \\
 &=& \sum_{j = 0} ^{ \lfloor l/2 \rfloor } \frac{ 1 }{ k^j } \, r^{ 2l - 4j } \, \mathcal{ S }_{ l j } ( r , \theta),
\nonumber
\end{eqnarray}
where $\mathcal{S}_{ l j } (r, \theta) $ is a convergent power series in $r$. 
The resulting series may be integrated term by term.
The $l$-th summand in (\ref{eqn:defn of Iktau3}) then gives rise to a convergent series of 
summands of the form 
\begin{eqnarray}
 \label{eqn:key integral1}
B_{\boldsymbol{ \nu }, l, j } ( m , \theta )\,\frac{1}{ k ^j } \,\int_0^{+\infty} \widetilde{r}^{2l - 4 j + a}\,e^{ 
 -\frac{1}{ 2 } \,k \, \lambda _{ \boldsymbol{\nu} } ( m )\,
  D ( \boldsymbol{\nu} ) 
\cdot   \widetilde{ r } ^4}\, \widetilde{r}\, \mathrm{d}\widetilde{r}
= O\left(\frac{1}{ k^{\frac{ l + 1 }{ 2 } + \frac{ a }{ 4 } } }\right).
\end{eqnarray}
with $j\le \lfloor l/2 \rfloor$ and $a=0,1,2,\ldots$.

The previous discussion shows that $\Pi_{k\boldsymbol{\nu}}(x , x )$ 
is given by an asymptotic expansion in descending powers of $k^{1/4}$,
and that the leading order term occurs for $l=1$ and $a=0$.

By Lemma \ref{eqn:key gaussian integrall}, $P_1 ( \xi ) = \xi$; by Lemma \ref{lem:rotational invariance}, 
$\| \kappa ( r ) \| = \lambda _{ \boldsymbol{ \nu } } ( m ) \, \| \boldsymbol{ \nu } \|
\cdot \mathcal{S}'_{\kappa} (r) $, where $\mathcal{S}'_{\kappa} (r)$ is a convergent power series in $r^2$
with $\mathcal{S}'_{\kappa} (0) = 1$.

In view of (\ref{eqn:computation b(r)}) and (\ref{eqn:defn of Fl}), we obtain
$$
F_1 (0 , w ) = - \sqrt{2 \pi} \cdot 
\frac{ (\nu_1 - \nu_2 ) \, (\nu _1 + \nu _2)^2}{\| \mathrm{Ad}_{ h_m } (\boldsymbol{ \nu _\perp })_M (m) \|^{3}}\,
\lambda_{ \boldsymbol{ \nu } } ( m ) ^{ 1/2 }\, r^2 \,\mathcal{S}_{F_1} (r, \theta),
$$
where $\mathcal{S}_{F_1}$ is real-analytic and $\mathcal{S}'' (0,\theta) \equiv 1$. 

Hence 
the leading order term of the asymptotic expansion of $\Pi_{k\boldsymbol{\nu}}(x , x )$ is given by 
\begin{equation}
 \label{eqn:leading order Pikxx}
 D_{G/T}\,\frac{k\,(\nu_1-\nu_2)}{(2\,\pi) ^ 2}\,\int_{-\pi} ^{\pi} \,\mathrm{d} \theta \,
\int_0 ^{ +\infty } \mathrm{d}\,r \,\left[
L_k(r ,\theta )\right],
\end{equation}
where 
\begin{eqnarray}
 \label{eqn:defn of L_k}
%
L_k ( r )
& 	:= &     
2^{ 3/2 } \, \frac{ k ^{d -1 /2 } }{ \pi ^{ d - 3/2 } } \, \lambda_{ \boldsymbol{ \nu } } ( m ) ^{ -(d-1/2) }
\\
&&\cdot  
\left[\frac{ (\nu_1 - \nu_2 ) \, (\nu _1 + \nu _2)^2}{\| \mathrm{Ad}_{ h_m } (\boldsymbol{ \nu _\perp })_M (m) \|^{3}}\right]\, 
e^{ -\frac{1}{ 2 } \,k \, \lambda_{ \boldsymbol{ \nu } } ( m ) \,
D(\boldsymbol{ \nu } ) \,r^4\, \mathcal{S} ( r , \theta)}\,r^3\,\widetilde{\mathcal{S}} (r, \theta),\, \nonumber
\end{eqnarray}
where again $\widetilde{\mathcal{S}} $ is real-analytic and $\widetilde{\mathcal{S}} (0,\theta) \equiv 1$.

We need to integrate in $\mathrm{ d } r$ the product of the last two factors in (\ref{eqn:defn of L_k}).
Let us perform the coordinate change $s = \sqrt{k} \, r ^2\, \mathcal{S} ( r , \theta)^{1/2}$, 
and argue as above. To leading order, we are
reduced to computing 
$$
\frac{1}{ 2\, k }\, \int_0 ^{ + \infty } \mathrm{ d } s \left[  e^{ -\frac{1}{ 2 } \,\lambda_{ \boldsymbol{ \nu } } ( m ) \,
D(\boldsymbol{ \nu } ) \,s^2 }\,s      \right] = \frac{1}{2 \,k } \cdot \frac{1}{\lambda_{ \boldsymbol{ \nu } } ( m ) \,
D(\boldsymbol{ \nu } ) } .
$$
Inserting this in (\ref{eqn:leading order Pikxx}), we conclude that the leading order term in the asymptotic expansion of
$\Pi_k(x,x)$ is
$$
\frac{D_{ G/T }}{\sqrt{ 2 }} \, \frac{1}{ \| \Phi_G ( m ) \|^{ d +1/2 } }
\, \left( \frac{ k \, \| \boldsymbol{ \nu } \| }{ \pi } \right)^{ d -1/2 } \cdot 
\frac{ \| \boldsymbol{ \nu } \| }{ \| \mathrm{Ad}_{ h_m } (\boldsymbol{ \nu _\perp })_M (m) \| }.
$$

The proof of Theorem \ref{thm:border pointwise asymptotics} is complete.

\end{proof}

\section{Proof of Theorem \ref{thm:border rescaled asymptotics}}

The proof is a modification of the one of Theorem \ref{thm:border pointwise asymptotics}, so the discussion
will be sketchy. We shall set
$$
x_{j,k} : = x +\frac{1}{ \sqrt{k} } \, \mathbf{v}_j, \quad j=1,2.
$$

\begin{defn}
 \label{defn:Gamma di vj}
 With the previous notation, let us set
 \begin{eqnarray*}
\lefteqn{
\Gamma (\boldsymbol{\vartheta} , g\,T, \mathbf{v}_j)
} \\
& := & -\frac{ 1 }{ 2 } \, \left[  
\Big \langle \mathrm{diag}\big(\mathrm{Ad}_{g^{-1}}  (\Phi_G' (m))\big) , \boldsymbol{\vartheta} \Big \rangle ^2 
+\Big\|
\mathbf{v}_1 - \mathbf{v}_2 + 
\mathrm{Ad}_g (\imath\, D_{\boldsymbol{\vartheta}}) _M(m) 
\Big\| _m ^2\right]\\
&&+ \imath \, 
\Big[- \omega_m (\mathbf{v}_1 , \mathbf{v}_2) 
+\omega_m\big( \mathrm{Ad}_g (\imath\, D_{\boldsymbol{\vartheta}}) _M(m),\mathbf{v}_1 + \mathbf{v}_2\big)   \Big].
\end{eqnarray*}
\end{defn}

Then, the same computations leading to Proposition \ref{prop:phase border} yield the following.

\begin{prop}
\label{prop:phase border rescaled} 
\begin{eqnarray*}
\lefteqn{
\imath \, k \, \left [ 
u \, \psi \left( \widetilde{\mu}_{ g \, e^{-\imath\boldsymbol{\vartheta}/\sqrt{k}}\,g^{-1}} (x_{1,k}) , x_{2,k} \right)
- \frac{1}{\sqrt{k}}\, \langle \boldsymbol{\nu} , \boldsymbol{\vartheta} \rangle  \right]}\\
&=& 
\imath\, \sqrt{k}\, \, 
\Psi(u, g\,T , \boldsymbol{\vartheta} )  + u \, \Gamma (\boldsymbol{\vartheta} , g\,T, \mathbf{v}_j)   
+ k \, R_3 \left( \frac{\mathbf{ v }_j }{\sqrt{k}},\,\frac { \boldsymbol{\vartheta} } {\sqrt{k}} \right). \nonumber
\end{eqnarray*}
\end{prop}

\begin{rem}
 \label{rem:phase border rescaled}
 Assuming $\mathbf{v}_1, \,\mathbf{v}_2
\in \mathfrak{g}_M(m_x)^{\perp_h}$, recalling Definition \ref{defn:def di Psi} we have
\begin{eqnarray*}
\Gamma (\boldsymbol{\vartheta} , g\,T, \mathbf{v}_j)
= \psi_2 ( \mathbf{v}_1, \mathbf{v}_2 )-\frac{ 1 }{ 2 } \, \boldsymbol{\vartheta}^t \,E( g\,T )\, \boldsymbol{\vartheta}.
\end{eqnarray*}
\end{rem}

In place of Corollary \ref{cor:projection composedk oscillatory}, we then obtain the following:
 \begin{eqnarray}
 \label{eqn:projection composedk1 rescaled}
\lefteqn{
\Pi_{k\boldsymbol{\nu}}(x_{1 , k},x_{2 , k})
}\\
&\sim&
\dfrac{k\,(\nu_1-\nu_2)}{(2\pi)^2}\,
\int_{G/T}\,\mathrm{d}V_{G/T}(gT)\,\int_{-\infty}^{\infty}\mathrm{d}\vartheta_1\,\int_{-\infty}^{\infty}\mathrm{d}\vartheta_2\,
\int_0^{+\infty}\,\mathrm{d}u\nonumber\\
&&\left[
e^{\imath\, \sqrt{k}\, \, 
\Psi(u, g\,T , \boldsymbol{\vartheta} )}\,
\mathcal{A}'_{k,\boldsymbol{\nu}}(u, g\,T , \boldsymbol{\vartheta} , \mathbf{v}_j )
\right],\nonumber
\end{eqnarray}
with the new amplitude
\begin{eqnarray}
 \label{eqn:defn di Aknu rescaled}
\mathcal{A}'_{k,\boldsymbol{\nu}}(u, g\,T , \boldsymbol{\vartheta} , \mathbf{v}_j  ) 
&:= &e^{u \,  \psi_2 ( \mathbf{v}_1, \mathbf{v}_2 ) - \frac{ u }{ 2 } \, 
\boldsymbol{\vartheta} ^t \, E( g\,T )\, \boldsymbol{\vartheta} 
+ k \, R_3 \left(\frac{\tau}{\sqrt{k}}, \frac { \boldsymbol{\vartheta} } {\sqrt{k}} \right)}\,
\Delta\left(e^{\imath\boldsymbol{\vartheta}/\sqrt{k}}\right)\nonumber\\
&& \cdot s\left(\widetilde{\mu}_{g\,e^{-\imath\boldsymbol{\vartheta}/\sqrt{k}}\,g^{-1}}(x_{1,k}),x_{2,k},k\,u\right).
\end{eqnarray}
Similarly, in place of (\ref{eqn:re-expression of Aknu}) we now have the following expansion:
\begin{eqnarray}
 \label{eqn:re-expression of Aknu rescaled}
\lefteqn{ \mathcal{A}'_{k,\boldsymbol{\nu}}(u, g\,T , \boldsymbol{\vartheta} , \mathbf{v}_j  ) }\\
&\sim &e^{u \,  \psi_2 ( \mathbf{v}_1, \mathbf{v}_2 ) - \frac{ u }{ 2 } \, 
\boldsymbol{\vartheta} ^t \, E( g\,T )\, \boldsymbol{\vartheta} 
+ k \, R_3 \left(\frac{\tau}{\sqrt{k}}, \frac { \boldsymbol{\vartheta} } {\sqrt{k}} \right)}
\,
\left[ e^{\frac{\imath}{\sqrt{k}} \, \vartheta_1( \boldsymbol{\zeta} ) } - 
e^{\frac{\imath}{\sqrt{k}} \, \vartheta_2( \boldsymbol{\zeta}) } \right]
\,\left( \frac{ k \, u }{ \pi } \right)^d \nonumber\\
&& \cdot \left[ 1
 + \sum _{j\ge 1} a_j \big( u , w ; \mathbf{v}_1, \mathbf{v}_2  , 
 \boldsymbol{\vartheta} ( \boldsymbol{\zeta} ) \big) \, k^{ - j/2 } \right] ,\nonumber
 \nonumber
\end{eqnarray}
where $a_j$ is, as a function of $\mathbf{v}_1$ and $\mathbf{v}_2$, a polynomial of degree
$\le 3j$. 

With these changes, Theorem \ref{thm:border rescaled asymptotics} 
can be proved by applying the arguments in the proof of Theorem \ref{thm:border pointwise asymptotics} with minor modifications.

\section{Proof of Theorem \ref{thm:outer dimension estimate}}

\begin{proof}
Let $A'\subset X$ be a one-sided \lq outer\rq \,  tubular neighborhood
of $X^G_{ \mathcal{ O }}$, that is, the intersection of 
$A$ with a tubular neighborhood of $X^G_{ \mathcal{ O }}$ in $X$.

By Theorem \ref{thm:rapid decrease fixed}, we have 
\begin{eqnarray}
 \label{eqn:integral for Hout}
 \lefteqn{ \dim_{out} H(X)_{ k\,\boldsymbol{ \nu } } } \\
 & = & \int _A \Pi_{ k\, \boldsymbol{ \nu } } ( x, x ) \, \mathrm{d}V_X (x) \sim 
 \int_{A'}\Pi_{ k\, \boldsymbol{ \nu } } ( x, x ) \, \mathrm{d}V_X (x). \nonumber
\end{eqnarray}
Let us denote by $\sigma  ( \boldsymbol{ \nu} ) $ the sign of $\nu_1 + \nu_2$. Then, locally along
$X^G_{ \mathcal{ O }}$, for some
sufficiently small $\delta>0$ we can parametrize $A'$ by a diffeomorphism 
$$
\Gamma :X^G_{ \mathcal{ O } } \times [0, \delta ) \rightarrow A', \quad
(x,\tau) \mapsto x +\tau \, \sigma ( \boldsymbol{ \nu} ) \, \Upsilon _{ \boldsymbol{ \nu } } ( m_x ),
$$
where $m_x = \pi (x)$.
The latter expression is meant in terms of a collection of smoothly varying systems of Heisenberg local
coordinates centered at $x\in X^G_{ \mathcal{ O } }$, locally defined along
$X^G_{ \mathcal{ O }}$
(to be precise, one ought to work locally on $X^G_{ \mathcal{ O } }$, 
introduce an appropriate open cover of $X^G_{ \mathcal{ O } }$, and a subordinate
partition of unity; however for the sake of exposition we shall omit details on this). 

We shall set $x_{ \tau} : = \Gamma ( x, \tau)$, and 
write 
$$
\Gamma^* (\mathrm{d}V _X ) = \mathcal{V} _X ( x , \tau ) \, \mathrm{d}V _{ X ^G_{ \mathcal{ O } }} ( x ) \,\mathrm{d}\tau,
$$
where $\mathcal{V} _X : X^G_{ \mathcal{ O } } \times [0, \delta ) \rightarrow (0,+\infty)$ is $\mathcal{C}^\infty$ and 
$\mathcal{V} _X ( x , 0 ) = \big\| \Upsilon _{ \boldsymbol{ \nu } } ( m_x )\big\|$.

Hence we obtain
\begin{eqnarray}
 \label{eqn:integral for Hout tubular}
 \lefteqn{ \dim_{out} H(X)_{ k\,\boldsymbol{ \nu } } } \\
 & \sim & 
 \int_{X ^G_{ \mathcal{ O } }} \, \mathrm{d}V _{ X ^G_{ \mathcal{ O } }} ( x )\,
 \int _0^{\delta} \, \mathrm{ d } \tau \, \left[\mathcal{V} _X ( x , \tau ) \, \Pi_{ k\, \boldsymbol{ \nu } } ( x_{ \tau} , x_{ \tau} ) \right]. \nonumber
\end{eqnarray}

By Theorem \ref{thm:rapid decrease slow}, only a rapidly decreasing contribution to (\ref{eqn:integral for Hout tubular}) 
is lost, if
integration in (\ref{eqn:integral for Hout tubular}) is restricted to the locus where $\tau \le C \, k^{\epsilon - 1/2 }$.
Thus the asymptotics of $\dim_{out} H(X)_{ k\,\boldsymbol{ \nu } }$ are unchanged, if the integrand is multiplied by 
a rescaled cut-off function $\varrho \left (k ^{ 1/2 - \epsilon } \, \tau \right)$, where 
$\varrho$ is identically one sufficiently near the origin in $\mathbb{R}$, 
and vanishes outside a slightly larger neighborhood.

With the rescaling $\tau \mapsto \tau / \sqrt{ k }$, we obtain
\begin{eqnarray}
 \label{eqn:integral for Hout tubular rescaled}
  \dim_{out} H(X)_{ k\,\boldsymbol{ \nu } } 
 \sim
 \frac{1}{ \sqrt{ k } } \, 
 \int_{X ^G_{ \mathcal{ O } }} \, \mathrm{d}V _{ X ^G_{ \mathcal{ O } }} ( x )\,\Big[ \mathcal{H} _{ k } ( x )  \Big],
 \nonumber
\end{eqnarray}
where with 
$x_{ \tau , k } : = \Gamma \left( x, k^{ -1/2 } \, \tau \right)$ 
we have set 
\begin{equation} 
\label{eqn: defn Pk x}
\mathcal{H} _{ k } ( x ) : = \int _0^{ +\infty } \, \mathrm{ d } \tau \, 
 \left[\varrho \left (k ^{  - \epsilon } \, \tau \right) \, \mathcal{V} _X \left( x , \frac{\tau}{ \sqrt{ k } } \right) \, 
 \Pi_{ k\, \boldsymbol{ \nu } } ( x_{ \tau , k } , x_{ \tau , k} ) \right].
\end{equation}
Integration in $\mathrm{ d } \tau$ is 
now over an expanding interval of the form $\left[ 0 , C' \, k^{ \epsilon } \right)$.

Let us consider the asymptotics of (\ref{eqn: defn Pk x}).
Having in mind (\ref{eqn:defn of Iktau3}), and inserting the 
Taylor expansion of $\mathcal{V} _X $, we are led to considering double integrals of the form
\begin{eqnarray}
 \label{eqn:typical double integral}
\lefteqn{ \frac{1}{ k^{(l+j)/2} }\, \int _0^{ +\infty } \, \mathrm{ d } \tau \,\int _0^{ +\infty } \, \mathrm{ d } r } \\
&& \left[r \, C( r ) ^l  \, \tau^j \, \mathcal{S}' (r )\frac{ P_l\left (\sqrt{k}\,\mathfrak{f} _k( \tau , w ) \right)}{ \mathfrak{a} (  w )^{ l + 1/2}  }
\, \cdot e^{ -\frac{1}{ 2 } \,k \, \frac{ \mathfrak{f}_k ( \tau , w ) ^2}{\mathfrak{a} ( w )}} \right], \nonumber
\end{eqnarray}
with $l \ge 1$ and $j \ge 0$, and their analogues with $S ( r )$ in place of $ C ( r ) $; $\mathcal{S}'$ is some
real-analytic function (dependence on $\theta$ and $x$ is implicit).

In view of (\ref{eqn:defn di fk}),
we have
\begin{eqnarray*} 
 \frac{\mathfrak{ f } _k ( \sigma ( \boldsymbol{ \nu} ) \, \tau , w  )}{\sqrt{ \mathfrak{ a } ( w )}} 
 =  - \sigma ( \boldsymbol{ \nu} ) \, \left [\frac{ ( \nu _1 -\nu _2 ) \, | \nu_1 + \nu_2 |}{\| \boldsymbol{ \nu } \|\, \sqrt{ \mathfrak{ a } ( 0 )} } \, 
r ^ 2 \, S _1( r ) + \frac{\tau}{k^{1/2}}\,  \, \frac{\mathfrak{r} ( 0 ) }{ \sqrt{ \mathfrak{ a } ( 0 ) } }\, S _2 ( r ,\theta)\right] ,
\end{eqnarray*}
where again $S _2 ( 0 ,\theta)=1$. Therefore, with the change of variables 
$$ 
s := k^{ 1 / 4 } \, r\, \sqrt{ S _1( r  ) }, \quad 
\widetilde{ \tau  } : = \tau \, S _2 ( r ,\theta) 
$$
we obtain
$$
\frac{\mathfrak{ f } _k (  \sigma ( \boldsymbol{ \nu} ) \, \tau , w  )}{\sqrt{ \mathfrak{ a } ( w )}} 
= -\frac{ \sigma ( \boldsymbol{ \nu} ) \,  }{ \sqrt{ k } } \,
\left[ \frac{ ( \nu _1 -\nu _2)  \, |\nu_1 + \nu_2 |}{\| \boldsymbol{ \nu } \| \, \sqrt{ \mathfrak{ a } ( 0 )} } \, s^2 +
\frac{\mathfrak{r} ( 0 ) }{ \sqrt{ \mathfrak{ a } ( 0 ) } }  \, \widetilde{ \tau  } \right].
$$

Therefore, we also have
$$
\mathfrak{ f } _k ( \sigma ( \boldsymbol{ \nu} ) \,  \tau , w  ) = -\frac{ \sigma ( \boldsymbol{ \nu} )   }{ \sqrt{ k } } \,
\left[ \frac{ ( \nu _1 -\nu _2)  \, | \nu_1 + \nu_2 | }{\| \boldsymbol{ \nu } \|  } \, s^2 +
\mathfrak{r} ( 0 )  \,  \widetilde{ \tau  } \right]\cdot 
\left[ 1 + R_1\ \left ( \frac{s}{\sqrt[4]{k}} \right) \right].
$$
 
With the substitution $a = s^2$,
(\ref{eqn:typical double integral}) may be rewritten as a linear combination 
of summands of the form
\begin{eqnarray}
 \label{eqn:typical double integral rewritten}
\lefteqn{ \frac{1}{ k^{(l+j+1)/2} }\, \int _0^{ +\infty } \, \mathrm{ d } \tau \,\int _0^{ +\infty } \, \mathrm{ d } a } \\
&& \left[C\left (\frac{\sqrt{a}}{\sqrt[4]{k}} \right)^l \,
\left ( A_1 \, a + B_1 \, \tau \right) ^{b} 
\, \tau^j\cdot \left[ 1 + R_1 \left ( \frac{ \sqrt{a} }{\sqrt[4]{k}} \right) \right]
\cdot e^{ -\frac{1}{ 2 } \, \left ( A_1 \, a + B_1 \, \tau \right) ^2} \right]\nonumber \\
& = & O \left( \frac{1}{ k^{(l+j+1)/2} } \right) . \nonumber
\end{eqnarray}

Hence the leading contribution occurs for $l=1$, $j = 0$, and dropping the term 
$R_1 \left ( k^{ - 1/4 } \, \sqrt{a}  \right)$.
The conclusion of Theorem \ref{thm:outer dimension estimate} then follows by a fairly simple computation.

\end{proof}


\begin{thebibliography}{Dillo99}





\bibitem[A]{atiyah}
M. F. Atiyah, {\em Convexity and commuting hamiltonians}, Bull. London Math.
Soc. \textbf{14} (1982), 1--15.




\bibitem[BSZ]{bsz} P. Bleher, B. Shiffman, S. Zelditch, {\em
Universality and scaling of correlations between zeros on complex
manifolds}, Invent. Math. \textbf{142} (2000), 351--395




\bibitem[BG]{boutet-guillemin} L. Boutet de Monvel, V. Guillemin,
{\em The spectral theory of Toeplitz operators},
Annals of Mathematics Studies, \textbf{99} (1981),
Princeton University Press, Princeton, NJ;
University of Tokyo Press, Tokyo


\bibitem[BS]{boutet-sjostraend} L. Boutet de Monvel, J. Sj\"ostrand,
{\em Sur la singularit\'e des noyaux de Bergman et de Szeg\"o},
Ast\'erisque \textbf{34-35} (1976), 123--164

\bibitem[BtD]{b-td}
T. Br\"{o}cker, T. tom Dieck, {\em Representations of compact Lie groups} 
Corrected reprint of the 1985 translation. 
Graduate Texts in Mathematics, \textbf{98}. Springer-Verlag, New York, 1995. x+313 pp. ISBN 0-387-13678-9




\bibitem[Cm]{camosso}
S. Camosso, {\em
Scaling asymptotics of Szeg\"{o} kernels under commuting Hamiltonian actions}, (English summary) 
Ann. Mat. Pura Appl. (4) \textbf{195} (2016), no. 6, 2027--2059


\bibitem[Ct]{catlin} D. Catlin, {\em
The Bergman kernel and a theorem of Tian},
Analysis and geometry in several complex variables (Katata, 1997), 1--23, Trends Math.,
Birkh\"{a}auser Boston, Boston, MA, 1999

\bibitem[Ch]{charles} L. Charles,
{\em Quantization of compact symplectic manifolds}, 
J. Geom. Anal. \textbf{26} (2016), no. 4, 2664–2710


\bibitem[D]{duist} J. J. 
Duistermaat,   {\em Fourier integral operators}, Progress in Mathematics, \textbf{130}, 
Birkh\"{a}user Boston, Inc., Boston, MA, 1996. x+142 pp. 

\bibitem[GP]{gal-pao}
A. Galasso, R. Paoletti, {\em Asymptotics of Szeg\"{o} kernels under Hamiltonian $SU(2)$ actions},
in preparation


\bibitem[G]{guillemin} V. Guillemin,  {\em Toeplitz operators in 
n dimensions}, Integral Equations Operator Theory \textbf{7} (1984), no. 2, 145--205

\bibitem[GS1]{guillemin-sternberg ga} V. Guillemin, S. Sternberg, 
{\em Geometric asymptotics}, Mathematical Surveys, No. \textbf{14}. American Mathematical Society, Providence, R.I., 1977. xviii+474 pp. 


\bibitem[GS2]{guillemin-sternberg gq} V. Guillemin, S. Sternberg, 
{\em Geometric quantization and multiplicities of group representations}, Invent. Math. \textbf{67} (1982), no. 3, 515–538


\bibitem[GS3]{guillemin-sternberg hq} V. Guillemin, S. Sternberg, 
{\em Homogeneous quantization and multiplicities of group representations}, J. Funct. Anal. \textbf{47} (1982), no. 3, 344--380
%

\bibitem[GS4]{guillemin-sternberg cp1} V. Guillemin, S. Sternberg, 
{\em Convexity properties of the moment mapping}, Invent. Math. \textbf{67} (1982), no. 3, 491--513

\bibitem[GS5]{guillemin-sternberg cp2} V. Guillemin, S. Sternberg, 
{\em Convexity properties of the moment mapping. II}, Invent. Math. \textbf{77} (1984), no. 3, 533--546


\bibitem[HHL]{hhl}
H. Herrmann,  
Chin-Yu Hsiao, Xiaoshan Li, {\em Szeg\"{o} kernel expansion and equivariant embedding of CR manifolds with circle action}, 
Ann. Global Anal. Geom. \textbf{52} (2017), no. 3, 313–340

\bibitem[HH]{hh}
Chin-Yu Hsiao, Rung-Tzung Huang,
{\em $G$-invariant Szeg\"{o} kernel asymptotics and CR reduction}, 
arXiv:1702.05012v3 

\bibitem[Ki1]{kirwan} F.  Kirwan,
{\em  
Convexity properties of the moment mapping. III}, 
Invent. Math. \textbf{77} (1984), no. 3, 547--552


\bibitem[Ki2]{kirwan-cohomology} F.  Kirwan,
{\em Cohomology of quotients in symplectic and algebraic geometry}, 
Mathematical Notes, \textbf{31}. Princeton University Press, Princeton, NJ, 1984. i+211 pp. ISBN: 0-691-08370-3 




\bibitem[Ko]{k}
B. Kostant, {\em Quantization and unitary representations. 
I. Prequantization}, Lectures in modern analysis and applications, III, pp. 87--208. 
Lecture Notes in Math., Vol. \textbf{170}, 
Springer, Berlin, 1970


\bibitem[L]{lerman} E. Lerman, {\em  Symplectic cuts}, Math. Res. Lett. \textbf{2} (1995), no. 3, 247--258







\bibitem[MM]{mm}
X. Ma, G. Marinescu, {\em Berezin-Toeplitz quantization and its kernel expansion. Geometry and quantization}, 
125–166, Trav. Math., \textbf{19}, Univ. Luxemb., Luxembourg, (2011)

\bibitem[MZ]{mz}
X. Ma, W. Zhang, {\em Bergman kernels and symplectic reductions}, Ast\'{e}risque, no. \textbf{318}, SMF (2008) 

\bibitem[P1]{pao-adv}
R. Paoletti, {\em The Szeg\"{o} kernel of a symplectic quotient}, Adv. Math. \textbf{197} (2005), no. 2, 523–553

\bibitem[P2]{pao-jsg0}
R. Paoletti, {\em Scaling limits for equivariant Szeg\"{o} kernels}, J. Symplectic Geom. \textbf{6} (2008), no. 1, 9–32, and
\textsl{Corrigendum}, J. Symplectic Geom. \textbf{11} (2013), no. 2, 317–318

\bibitem[P3]{pao-IJM} R. Paoletti, {\em Asymptotics of Szeg\"{o} kernels under Hamiltonian
torus actions}, Israel Journal of Mathematics \textbf{191} (2012), no. 1, 363--403
DOI: 10.1007/s11856-011-0212-4


\bibitem[P4]{pao-loa} R. Paoletti, {\em Lower-order asymptotics for Szeg\"{o} and Toeplitz kernels under Hamiltonian circle actions},
Recent advances in algebraic geometry, 321--369, 
London Math. Soc. Lecture Note Ser., \textbf{417}, Cambridge Univ. Press, Cambridge, 2015


\bibitem[P5]{pao-jsg} R. Paoletti, {\em 
Local trace formulae for commuting Hamiltonians in Toeplitz quantization}, J. Symplectic Geom. \textbf{15} (2017), no. 1, 189–245


\bibitem[Sch]{schlichenmaier} M. Schlichenmaier, 
{\em  Berezin-Toeplitz quantization for compact K\"{a}hler manifolds. An introduction}, 
Geometry and quantization, 97–124,
Trav. Math., \textbf{19}, Univ. Luxemb., Luxembourg, 2011


\bibitem[SZ]{sz} B. Shiffman, S. Zelditch, {\em Asymptotics of almost
holomorphic sections of ample line bundles on symplectic
manifolds}, J. Reine Angew. Math. {\bf 544} (2002), 181--222


\bibitem[S]{st-gtp} S. Sternberg, {\em Group theory and physics}, 
Cambridge University Press, Cambridge, 1994. xiv+429 pp. ISBN: 0-521-24870-1 


\bibitem[T]{tian} G. Tian,
{\em On a set of polarized K\"{a}hler metrics on algebraic manifolds}, J. Differential Geom. 
\textbf{32} (1990), no. 1, 99--130

\bibitem[V]{var} V. S. 
Varadarajan, {\em An introduction to harmonic analysis on semisimple Lie groups}, 
Corrected reprint of the 1989 original. Cambridge Studies in Advanced Mathematics, \textbf{16}. 
Cambridge University Press, Cambridge, 1999. x+316 pp. ISBN: 0-521-34156-6

\bibitem[Z1]{zelditch-index-dynamics} S. Zelditch,
{\em Index and dynamics of quantized contact transformations},
Ann. Inst. Fourier (Grenoble) \textbf{47} (1997), no. 1, 305--363

\bibitem[Z2]{zelditch-theorem-of-Tian} S. Zelditch, {\em Szeg\"o kernels and a theorem of Tian},
Int. Math. Res. Not. {\bf 6} (1998), 317--331


\end{thebibliography}
\end{document}